\documentclass[12pt,reqno]{amsart}

\usepackage{color}
\usepackage{amsmath, amsthm, amssymb}
\usepackage{amsfonts}

\usepackage{hyperref}

\usepackage[ansinew]{inputenc}
\usepackage[dvips]{epsfig}
\usepackage{graphicx}
\usepackage[english]{babel}

\theoremstyle{plain}
\newtheorem{thm}{Theorem}[section]
\newtheorem{cor}[thm]{Corollary}
\newtheorem{lem}[thm]{Lemma}
\newtheorem{prop}[thm]{Proposition}

\theoremstyle{definition}
\newtheorem{defi}[thm]{Definition}

\theoremstyle{remark}
\newtheorem{rem}[thm]{Remark}

\numberwithin{equation}{section}

\newcommand{\average}{{\mathchoice {\kern1ex\vcenter{\hrule height.4pt
width 6pt depth0pt} \kern-9.7pt} {\kern1ex\vcenter{\hrule
height.4pt width 4.3pt depth0pt} \kern-7pt} {} {} }}
\newcommand{\ave}{\average\int}
\def\R{\mathbb{R}}
\newcommand{\I}{{\rm I}}

\begin{document}

\title[Boundary regularity for integro-differential equations]
{Boundary regularity for fully nonlinear integro-differential equations}

\author{Xavier Ros-Oton}
\address{The University of Texas at Austin, Department of Mathematics, 2515 Speedway, Austin, TX 78751, USA}
\email{ros.oton@math.utexas.edu}

\thanks{The authors were supported by grants MINECO MTM2011-27739-C04-01 and GENCAT 2009SGR-345}

\author{Joaquim Serra}
\address{Universitat Polit\`ecnica de Catalunya, Departament de Matem\`{a}tica  Aplicada I, Diagonal 647,
08028 Barcelona, Spain}
\email{joaquim@arcvi.io}



\maketitle

\begin{abstract}
We study fine boundary regularity properties of solutions to fully nonlinear elliptic integro-differential equations of order $2s$, with $s\in(0,1)$.

We consider the class of nonlocal operators $\mathcal L_*\subset \mathcal L_0$, which consists of infinitesimal generators of stable L\'evy processes belonging to the class $\mathcal L_0$ of Caffarelli-Silvestre.
For fully nonlinear operators $\I$ elliptic with respect to $\mathcal L_*$, we prove that solutions to $\I u=f$ in $\Omega$, $u=0$ in $\R^n\setminus\Omega$, satisfy $u/d^s\in C^{s+\gamma}(\overline\Omega)$, where $d$ is the distance to $\partial\Omega$ and $f\in C^\gamma$.

We expect the class $\mathcal L_*$ to be the largest scale invariant subclass of $\mathcal L_0$ for which this result is true.
In this direction, we show that the class $\mathcal L_0$ is too large for all solutions to behave like $d^s$.

The constants in all the estimates in this paper remain bounded as the order of the equation approaches 2.
Thus, in the limit $s\uparrow1$ we recover the celebrated boundary regularity result due to Krylov for fully nonlinear elliptic equations.
\end{abstract}


\vspace{4mm}

\section{Introduction and results}

This paper is concerned with boundary regularity for fully nonlinear elliptic integro-differential equations.

Since the foundational paper of Caffarelli and Silvestre \cite{CS}, ellipticity for a nonlinear
integro-differential operator is defined relatively to a given set $\mathcal L$ of linear translation invariant elliptic operators.
This set $\mathcal L$ is called the ellipticity class.

The reference ellipticity class from \cite{CS} is the class $\mathcal L_0= \mathcal L_0(s)$, containing all operators $L$ of the form
\begin{equation}\label{operator-L0}
Lu(x)=\int_{\R^n} \left(\frac{u(x+y)+u(x-y)}{2}-u(x)\right) K(y)\,dy
\end{equation}
with even kernels $K(y)$ bounded between two positive multiples of $(1-s)|y|^{-n-2s}$, which is the kernel of the fractional Laplacian $(-\Delta)^s$.

In the three papers \cite{CS,CS2,CS3}, Caffarelli and Silvestre studied the interior regularity for solutions $u$ to
\begin{equation}\label{primer_dirichlet}\left\{\begin{array}{rcl}
\I u &= & f \quad \mbox{in } \Omega\\
u  &=&g\quad \mbox{in }\R^n\setminus\Omega,
\end{array}\right.\end{equation}
where $\I$ is a translation invariant fully nonlinear integro-differential operator of order~$2s$ (see the definition later on in this Introduction).
They proved existence of viscosity solutions, established $C^{1+\alpha}$ interior regularity of solutions \cite{CS}, $C^{2s+\alpha}$ regularity in case of convex equations \cite{CS3}, and developed a perturbative theory for non translation invariant equations \cite{CS2}.
Thus, the interior regularity for these equations is well understood.

However, very little is known about the boundary regularity for fully nonlinear nonlocal problems.

When $\I$ is the fractional Laplacian $-(-\Delta)^s$, the boundary regularity of solutions $u$ to \eqref{primer_dirichlet} is now quite well understood.
The first result in this direction was obtained by Bogdan, who established the boundary Harnack principle for $s$-harmonic functions \cite{Bogdan} ---i.e., for solutions to $(-\Delta)^su=0$.
More recently, we proved in \cite{RS-Dir} that if $f\in L^\infty$, $g\equiv0$, and $\Omega$ is $C^{1,1}$ then $u\in C^s(\R^n)$ and $u/d^s\in C^\alpha(\overline\Omega)$ for some small $\alpha>0$, where $d$ is the distance to the boundary $\partial\Omega$.
Moreover, the limit of $u(x)/d^s(x)$ as $x\rightarrow\partial\Omega$ is typically nonzero (in fact it is positive if $f<0$), and thus the $C^s$ regularity of $u$ is optimal.
After this, Grubb \cite{Grubb} showed that when $f\in C^\gamma$ with $\gamma>0$ (resp. $f\in L^\infty$), $g\equiv0$, and $\Omega$ is smooth, then $u/d^s\in C^{\gamma+s-\epsilon}(\overline\Omega)$ (resp. $u/d^s\in C^{s-\epsilon}(\overline\Omega)$) for all $\epsilon>0$.
In particular, $f\in C^\infty$ leads to $u/d^s\in C^\infty(\overline\Omega)$.
Thus, the correct notion of boundary regularity for equations of order $2s$ is the H\"older regularity of the quotient $u/d^s$.
In a new work \cite{Grubb2}, Grubb removes the $\epsilon$ in the previous estimates, obtaining that $u/d^s$ is $C^{s+\gamma}$ whenever $f\in C^\gamma$ and neither $\gamma+s$ nor $\gamma+2s$ are integers.

Here, we obtain boundary regularity for \emph{fully nonlinear} integro-differential problems of the form \eqref{primer_dirichlet} which are elliptic with respect to the class $\mathcal L_*\subset \mathcal L_0$ defined as follows.
$\mathcal L_*$ consists of all linear operators of the form
\begin{equation}\label{L*1}
Lu(x) = (1-s)\int_{\R^n} \left(\frac{u(x+y)+u(x-y)}{2}-u(x)\right) \frac{\,\mu(y/|y|)}{|y|^{n+2s}}\,dy,
\end{equation}
with
\begin{equation} \label{L*2}
\mu\in L^\infty(S^{n-1}) \quad \mbox{satisfying} \quad \mu(\theta)=\mu(-\theta)\quad \mbox{and} \quad \lambda\le \mu\le \Lambda,
\end{equation}
where $0<\lambda\leq \Lambda$ are called ellipticity constants.
The class $\mathcal L_*$ consists of all infinitesimal generators of \emph{stable} L\'evy processes belonging to $\mathcal L_0$.
Our main result essentially establishes that when $f\in C^\gamma$, $g\equiv0$, and $\Omega$ is $C^{2,\gamma}$, viscosity solutions $u$ satisfy
\begin{equation}\label{our-results}
u/d^s\in C^{s+\gamma}(\overline\Omega).
\end{equation}
In case of flat boundary we assume $\mu\in C^{\gamma}(S^{n-1})$.
In general $C^{2,\gamma}$ domains $\Omega$, we need to assume $\mu\in C^{1,\gamma}(S^{n-1})$.

We expect the class $\mathcal L_*$ to be the largest scale invariant subclass of $\mathcal L_0$ for which this result is true.

For general elliptic equations with respect to $\mathcal L_0$, no fine boundary regularity results like \eqref{our-results} hold.
In fact, the class $\mathcal L_0$ is too large for all solutions to be comparable to $d^s$ near the boundary.
Indeed, we show in Section 2 that there are powers $0<\beta_1<s<\beta_2$ for which the functions $(x_n)_+^{\beta_1}$ and $(x_n)_+^{\beta_2}$ satisfy
\[ M^+_{\mathcal L_0} (x_n)_+^{\beta_1} =0\quad \textrm{and}\quad M^-_{\mathcal L_0} (x_n)_+^{\beta_2} =0\quad\ \mbox{in}\ \,\{x_n>0\},\]
where $M^+_{\mathcal L_0}$ and $M^-_{\mathcal L_0}$ are the extremal operators for the class $\mathcal L_0$; see their definition in Section~2.
Hence, since $(-\Delta)^s(x_n)_+^s=0$ in $\{x_n>0\}$, we have at least three functions which solve fully nonlinear elliptic equations with respect to $\mathcal L_0$ but which are not even comparable near the boundary $\{x_n=0\}$.
As we show in Section 2, the same happens for the subclasses $\mathcal L_1$ and $\mathcal L_2$ of $\mathcal L_0$, which have more regular kernels and were considered in \cite{CS,CS2,CS3}.

The constants in our estimates remain bounded as $s\uparrow1$.
Thus, in the limit we recover the celebrated boundary regularity estimate of Krylov for second order fully nonlinear elliptic equations \cite{Krylov}.

\addtocontents{toc}{\protect\setcounter{tocdepth}{1}}  

\subsection{The class $\mathcal L_*$}

The class $\mathcal L_*$ consists of all infinitesimal generators of stable L\' evy processes belonging to $\mathcal L_0$.
This type of L\'evy processes are well studied in probability, as explained next.
In that context, the function $\mu\in L^\infty(S^{n-1})$ is called the spectral measure.

Stable processes are for several reasons a natural extension of Gaussian processes.
For instance, the Generalized Central Limit Theorem states that the distribution of a sum of independent identically distributed random variables with heavy tails converges to  a stable distribution;
see \cite{ST}, \cite{Levy}, or \cite{BKS} for a precise statement of this result.
Thus, stable processes are often used to model sums of many random independent perturbations with heavy-tailed  distributions ---i.e., when large outcomes are not unlikely.
In particular, they arise frequently in financial mathematics, internet traffic statistics, or signal processing; see for instance \cite{PS,MR,MR2,Nolan2,Nolan3,NolanPM,AB,KR,Pa,HST} and the books \cite{Nolan, ST}.

Linear equations $Lu=f$ with $L$ in the class $\mathcal L_*$ have already been studied, specially by Sztonyk and Bogdan; see for instance \cite{Sztonyk3,BKK,PT,Sztonyk1,Sztonyk2,Sztonyk4}.
When the spectral measure $\mu$ in \eqref{L*1} belongs to $C^\infty(S^{n-1})$, the regularity up to the boundary of $u/d^s$ in $C^\infty$ domains follows from the recent results of Grubb \cite{Grubb} for linear pseudo-differential operators.

Notice that all second order linear uniformly elliptic operators are recovered as limits of operators in $\mathcal L_*=\mathcal L_*(s)$ as $s\to 1$.
In particular, all second order fully nonlinear equations $F(D^2u)=f(x)$ are recovered as limits of the fully nonlinear integro-differential equations that we consider.
Furthermore, when $s<1$ the class of translation invariant linear operators $\mathcal L_*(s)$ is much richer than the one of second order uniformly elliptic operators.
Indeed, while any operator in the latter class is determined by a positive definite $n\times n$ matrix, a function $\mu:S^{n-1}\rightarrow\R^+$ is needed to determine an operator in $\mathcal L_*(s)$.

A key feature of the class $\mathcal L_*$ for boundary regularity issues is that
\[L(x_n)_+^s=0\quad \textrm{in}\ \{x_n>0\}\quad \textrm{for all}\ L\in \mathcal L_*.\]
This is essential first to construct barriers which are comparable to $d^s$, and later to prove finer boundary regularity.

\subsection{Equations with ``bounded measurable coefficients''}

The first result of in this paper, and on which all the other results rely, is Proposition \ref{thm1} below.

Here, and throughout the article, we use the definition of viscosity solutions and inequalities of \cite{CS}.
Moreover, for $r>0$ we denote
\[B_r^+=B_r\cap \{x_n>0\}\quad \textrm{and}\quad B_r^-=B_r\cap \{x_n<0\},\]
and the constants $\lambda$ and $\Lambda$ in \eqref{L*2} are called ellipticity constants.

The extremal operators associated to the class $\mathcal L_*$ are denoted by $M_{\mathcal L_*}^+$ and $M_{\mathcal L_*}^-$,
\[M_{\mathcal L_*}^+u=\sup_{L\in\mathcal L_*}Lu\quad \textrm{and}\quad M_{\mathcal L_*}^-u=\inf_{L\in\mathcal L_*}Lu.\]
Note that, since $\mathcal L_*\subset\mathcal L_0$, then $M^-_{\mathcal L_0}\leq M_{\mathcal L_*}^-\leq M_{\mathcal L_*}^+\leq M^+_{\mathcal L_0}$.

\begin{prop}\label{thm1}
Let $s_0\in(0,1)$ and  $s\in[s_0,1)$.
Assume that $u\in C(B_1)\cap L^\infty(\R^n)$ is a viscosity solution of
\begin{equation}\label{equaciou}
\left\{
\begin{array}{rcll}
M_{\mathcal L_*}^+u \hspace{-5pt} &\ge & \hspace{-5pt}-C_0 \quad &\mbox{ in } B_1^+\\
M_{\mathcal L_*}^-u \hspace{-5pt} &\le & \hspace{-5pt}C_0 \quad &\mbox{ in } B_1^+\\
u \hspace{-5pt}  &=& \hspace{-5pt} 0            &\mbox{ in }B_1^-,
\end{array}
\right.
\end{equation}
for some nonnegative constant $C_0$.
Then, $u/x_n^s$ is $C^{\bar\alpha}(\overline{B_{1/2}^+})$ for some $\bar\alpha>0$, with the estimate
\begin{equation}\label{first_estimate}
\|u/x_n^s\|_{C^{\bar\alpha}(B_{1/2}^+)} \le C \left(C_0 + \|u\|_{L^\infty(\R^n)}\right).
\end{equation}
The constants ${\bar\alpha}$ and $C$ depend only on $n$, $s_0$, and the ellipticity constants.
\end{prop}

It is important to remark that the constants in our estimate remain bounded as $s\rightarrow1$.
This means that from Proposition \ref{thm1} we can recover the classical boundary Harnack inequality of Krylov \cite{Krylov}.

The estimate of Proposition \ref{thm1} is only a first step towards our results.
It is obtained via a nonlocal version of the method of Caffarelli \cite{Kazdan} for second order equations with bounded measurable coefficients; see also Section 9.2 in \cite{Caff-Cabre}.
This method has been adapted to nonlocal equations by the authors in \cite{RS-Dir}
\footnote{In \cite{RS-Dir}, we incorrectly attributed the method to Krylov. Instead, we were using a method due to Caffarelli, who gave an alternative proof of Krylov's boundary regularity theorem.},
where we proved estimate \eqref{first_estimate} for the fractional Laplacian $(-\Delta)^s$ in $C^{1,1}$ domains.

As explained before, our main result is the $C^{s+\gamma}$ regularity of $u/d^s$ in $C^{2,\gamma}$ domains for solutions $u$ to fully nonlinear integro-differential equations (see the next subsection).
Thus, for solutions to the nonlinear equations we push the small H\"older exponent $\bar\alpha>0$ in \eqref{first_estimate} up to the sharp exponent $s+\gamma$ in \eqref{our-results}.
To achieve this, new ideas are needed, and the procedure that we develop differs substantially from that in second order equations.
We use a new compactness method and the ``boundary'' Liouville-type Theorem \ref{thm-liouv-1+s}, stated later on in the Introduction.
This Liouville theorem relies on Proposition~\ref{thm1}.

\subsection{Main result}

Before stating our main result, let us recall the definition and motivations of fully nonlinear integro-differential operators.

As defined in \cite{CS}, a fully nonlinear operator $\I$ is said to be elliptic with respect to a subclass $\mathcal L\subseteq \mathcal L_0$ when
\[M_{\mathcal L}^-(u-v)(x)\leq \I u(x)-\I v(x)\leq M_{\mathcal L}^+(u-v)(x)\]
for all test functions $u,v$ which are $C^2$ in a neighborhood of $x$ and having finite integral against $\omega_s(x)=(1-s)(1+|x|^{-n-2s})$.
Moreover, if
\[ \I\left(u(x_0+\cdot)\right) (x)=( \I u )(x_0+x),\]
then we say that $\I$ is translation invariant.

Fully nonlinear elliptic integro-differential equations naturally arise in stochastic control and games.
In typical examples, a single player or two players control some parameters
(e.g. the volatilities of the assets in a portfolio) affecting  the joint distribution
of the random increments of $n$ variables $X(t)\in \R^n$.
The game ends when $X(t)$ exits for the first time a certain domain $\Omega$ (as when having automated orders to sell assets when their prices cross certain limits).

The {\em value} or {\em expected payoff} of these games $u(x)$ depends
on the starting point $X(0)=x$ (initial prices of all assets in the portfolio).
A remarkable fact is that the value $u(x)$ solves an equation of the type $\I u =0$, where
\begin{equation}\label{examples_operators}
\I u(x) = \sup_a \bigl( L_a u +c_a\bigr) \quad \mbox{ or }\quad \I u(x) = \inf_b \sup_a \bigl( L_{ab} u+c_{ab}\bigr).
\end{equation}
The first equation, known as the Bellman equation, arises in control problems (a single player), while the second one, known as the Isaacs equations, arises in zero-sum games (two players).
The linear operators $L_a$  and $L_{ab}$ are infinitesimal generators of L\'evy processes, standing for all the possible choices of the distribution of time increments of $X(t)$.
The constants $c_a$ and $c_{ab}$ are costs associated to the choice of the operators $L_a$ and $L_{ab}$.
More involved equations with zeroth order terms and right hand sides have also meanings in this context as interest rates or running costs.
See \cite{C-Survey,wiki,OS,CT,CS}, and references therein for more information on these equations.

When all $L_a$ and $L_{ab}$ belong to $\mathcal L_*$, then \eqref{examples_operators} are fully nonlinear translation invariant operators elliptic with respect to $\mathcal L_*$, as defined above.

The interior regularity for fully nonlinear integro-differential elliptic equations was mainly established by Caffarelli and Silvestre in the
well-known paper \cite{CS}.
More precisely, for some small $\alpha>0$, they obtain $C^{1+\alpha}$ interior regularity for fully nonlinear elliptic equations with respect to the class $\mathcal L_1$ made of kernels in $\mathcal L_0$ which are $C^1$ away from the origin.
For $s>\frac12$, the same result in the class $\mathcal L_0$ has been recently proved by Kriventsov \cite{Kriv}.
These estimates are uniform as the order of the equations approaches two, so they can be viewed as a natural extension of the interior
regularity for fully nonlinear equations of second order.
There were previous interior estimates by Bass and Levin \cite{Bass2} and by Silvestre \cite{S} which are not uniform as the order of the equation approaches~2.
An interesting aspect of \cite{S} is that its proof is short and uses only elementary analysis tools, taking advantage of the nonlocal character of the equations.
This is why the same ideas have been used in other different contexts \cite{Silvestre1,Silvestre2}.

For convex equations elliptic with respect to $\mathcal L_2$ (i.e., with kernels in $\mathcal L_0$ which are $C^2$ away from the origin), Caffarelli and Silvestre obtained $C^{2s+\alpha}$ interior regularity~\cite{CS3}.
This is the nonlocal extension of the Evans-Krylov theorem.
The same result in the class $\mathcal L_0$ has been recently proved by the second author \cite{Se2}.
Other important references concerning interior regularity for nonlocal equations in nondivergence form are \cite{RKS,KM,CD,BCI,GSchwab}.

Our main result reads as follows.

\begin{thm}\label{thm2}
Let $s_0\in (0,1)$ and $s\in [s_0,1)$.
Let $\bar\alpha$ be the exponent given by Proposition \ref{thm1}.

Assume that $\I$ is a fully nonlinear and translation invariant operator of the form \eqref{examples_operators}-\eqref{L*1}-\eqref{L*2}.
Assume in addition that the spectral measures satisfy
\[\|\mu_{ab}\|_{C^\gamma(S^{n-1})}\leq \Lambda.\]

Let $f\in C^\gamma\bigl({B_1^+}\bigr)$, and $u\in L^\infty(\R^n)\cap C\bigl(B_1\bigr)$ be any viscosity solution of
\begin{equation}\label{eq-thm2}
\left\{\begin{array}{rcll}
\I u \hspace{-5pt} &= & \hspace{-5pt}f \quad &\mbox{ in }B_1^+\\
u \hspace{-5pt}  &=& \hspace{-5pt} 0            &\mbox{ in }B_1^-.
\end{array}\right.
\end{equation}
Assume that $\gamma\in (0,1-s+\bar\alpha)$, and that $s+\gamma$ is not an integer.

Then, $u/(x_n)^s$ belongs to $C^{s+\gamma}\bigl(\overline{B_{1/2}^+}\bigr)$ with the estimate
\[ \bigl\|u/(x_n)^s\bigr\|_{C^{s+\gamma}(B_{1/2}^+)} \le C \bigl( \|u\|_{L^\infty(\R^n)} + \|f\|_{C^\gamma(B_1^+)}\bigr),\]
where the constant $C$ depends only on $n$, $s_0$, $\gamma$, $\lambda$, and $\Lambda$.
\end{thm}

Notice that taking $\gamma=1-s+\alpha$ in the previous result (with $\alpha>0$ small), we find that $u/(x_n)^s$ is $C^{1,\alpha}$ up to the boundary.
Thus, it gives an estimate of order $1+s+\alpha$ on the boundary, and not only $2s+\alpha$ (which is the regularity in the interior of the domain for convex equations \cite{CS3,Se2}).

As said above, before our results almost nothing was known about the boundary regularity of solutions to fully nonlinear integro-differential equations.
It was only known that solutions $u$ to these equations are $C^\alpha$ up to the boundary for some small $\alpha>0$ (a result for $u$ but not for the quotient $u/d^s$).

Theorem \ref{thm2} gives a sharp boundary regularity estimate for fully nonlinear equations elliptic with respect to $\mathcal L_*$, recovering in the limit $s\uparrow1$ the celebrated boundary regularity estimate of order $2+\alpha$ of Krylov.

Note also that our result is not only an a priori estimate for classical solutions but also applies to viscosity solutions.
For local equations of second order, the boundary regularity for viscosity solutions to fully nonlinear equations has been recently obtained by Silvestre and Sirakov \cite{SS}.
The methods that we introduce here to prove Theorem \ref{thm2} can be used to give a new proof of the results for second order fully nonlinear equations.

Apart from their intrinsic interest, we believe that the results and proofs in this paper will be useful in order to make progress in the study of other important models with nonlocal diffusion where the boundary behavior of solutions plays a crucial role.
For instance, the precise understanding of the boundary behavior of solutions given by Theorem \ref{thm2} (and Theorems \ref{thm2-curved}  and \ref{thm-liouv-1+s} stated later on in the introduction) opens the door to a detailed study of regularity properties for natural free boundary models such as:
\begin{itemize}
\item The obstacle problem with diffusion operator in $L\in \mathcal L_*$:
\[ \min\{ -Lu, u-\varphi\} =0 \quad \mbox{in }\R^n. \]

\item The (variational) ``one-phase'' problem, concerning local minimizers of
\[ \frac{1}{2}\iint \bigl(u(x)-u(x+y)\bigr)^2 \frac{\mu(y/|y|)}{|y|^{n+2s}}\,dx\,dy + \int \chi_{\{u>0\}} (x)\,dx .\]
\end{itemize}
Currently, these models have only been studied for the fractional Laplacian $L=(-\Delta)^s$; see \cite{CSSil,CRSire} and references therein.

\subsection{Estimates in $C^{2,\gamma}$ domains}

We will also establish the following boundary regularity estimate in $C^{2,\gamma}$ domains.

In this result, we consider operators
\begin{equation}\label{curved-1}
\I(u,x)=\inf_b\sup_a\bigl(L_{ab}u+c_{ab}(x)\bigr),
\end{equation}
with $L_{ab}$ of the form \eqref{L*1}-\eqref{L*2} and satisfying
\begin{equation}\label{curved-2}
\|\mu_{ab}\|_{C^{1,\gamma}(S^{n-1})}\leq \Lambda.
\end{equation}
Moreover, we assume also that
\begin{equation}\label{curved-3}
\|c_{ab}\|_{C^{\gamma}(\overline\Omega)}\leq C_0.
\end{equation}
Under these assumptions, we have the following.

\begin{thm}\label{thm2-curved}
Let $s_0\in (0,1)$ and $s\in [s_0,1)$.
Let $\bar\alpha$ be the exponent given by Proposition \ref{thm1}, and let $\gamma\in(0,1-s+\bar\alpha)$.
Assume in addition that $s+\gamma$ is not an integer, and that $\gamma\leq s$.

Let $\Omega$ be any $C^{2,\gamma}$ domain, and let $\I$ be a fully nonlinear operator of the form \eqref{curved-1} with \eqref{L*1}-\eqref{L*2}-\eqref{curved-2}-\eqref{curved-3}.
Let $d(x)$ be a $C^{2,\gamma}(\overline\Omega)$ function that coincides with ${\rm dist}\,(x,\R^n\setminus\Omega)$ in a neighborhood of $\partial\Omega$.

Let  $u\in C\bigl(\overline{\Omega}\bigr)$ be any viscosity solution of
\[\left\{\begin{array}{rcll}
\I(u,x) \hspace{-5pt} &= & \hspace{-5pt}f \quad &\mbox{ in }\Omega\\
u \hspace{-5pt}  &=& \hspace{-5pt} 0            &\mbox{ in }\R^n\setminus\Omega,
\end{array}\right.\]
with $\|f\|_{C^\gamma(\overline\Omega)}\leq C_0$.

Then, $u/d^s$ belongs to $C^{s+\gamma}\bigl(\overline{\Omega}\bigr)$ with the estimate
\[ \bigl\|u/d^s\bigr\|_{C^{s+\gamma}(\overline\Omega)} \le CC_0 ,\]
where the constant $C$ depends only on $\Omega$, $s_0$, $\gamma$, $\lambda$ and $\Lambda$.
\end{thm}

In case of local equations of second order, the Krylov estimate in $C^{2,\gamma}$ domains is usually proved by flattening the boundary; see \cite{Kazdan,SS}.
However, we will not do this here, and thus we do not deduce Theorem \ref{thm2-curved} from Theorem \ref{thm2}.

Notice that as a consequence of Theorem \ref{thm2-curved}, one can immediately obtain estimates for solutions to
\[\left\{\begin{array}{rcl}
\I u &= & f \quad \mbox{in } \Omega\\
u  &=&g\quad \mbox{in }\R^n\setminus\Omega,
\end{array}\right.\]
with $g\in C^{2s+\gamma}(\R^n)$ and $f\in C^\gamma(\overline\Omega)$.
Indeed, one only needs to consider $\tilde u=u-g$ in $\R^n$, and apply Theorem \ref{thm2-curved} to $\tilde u$ to find that $(u-g)/d^s\in C^{s+\gamma}(\overline\Omega)$.

\subsection{Ingredients of the proof}

Let us explain now the main ideas in the proofs of Theorems \ref{thm2} and \ref{thm2-curved}.

Our proof of these results differs substantially from boundary regularity methods in second order equations.
Indeed, recall that for second order equations one first shows that $D^2u$ is bounded on the boundary, and then the estimate for equations with bounded measurable coefficients implies immediately a $C^{2,\alpha}$ estimate on the boundary for solutions to fully nonlinear equations; see \cite{Kazdan, Caff-Cabre}.

This is much more delicate for nonlocal equations, and it is not easy at all to prove Theorem \ref{thm2} using Proposition \ref{thm1}.
A main reason for this is not only the nonlocal character of the estimates, but also that tangential and normal derivatives of the solution behave differently on the boundary; recall that the solution is $C^s$ and not Lipschitz up to the boundary.

In our proof, the main step towards Theorem \ref{thm2} is an iterative result of the form
\begin{equation}\label{austin-iter}
u/(x_n)^s\in C^\beta(B_{3/4}^+)\quad \Longrightarrow\quad u/(x_n)^s\in C^{\alpha+\beta}(B_{1/2}^+),
\end{equation}
where $\alpha\in (0,\bar\alpha)$ is small, and $\beta\in[0,1]$ satisfies $\alpha+\beta\leq \gamma+s$.

Essentially, this is equivalent to an estimate {\em on} the boundary, which reads as follows.
If $u$ satisfies the hypotheses of Theorem \ref{thm2}, and $u/(x_n)^s\in C^\beta(B_{3/4}^+)$, then for all $z\in \{x_n=0\}\cap \overline {B_{1/2}}$ there exist $b_z\in \R$ and $p_z\in \R^n$ for which
\begin{equation}\label{???}
\left| u(x) - b_z(x_n)_+^s - (p_z\cdot x)(x_n)_+^s \right| \le C|x-z|^{\alpha+\beta+s}\quad \mbox{for all }x\in B_1.
\end{equation}
In case that $\alpha+\beta<1$, then the term $(p_z\cdot x)(x_n)_+^s$ does not appear.

The estimate on the boundary \eqref{???} relies heavily on two ingredients, as explained next.

The first ingredient is a Liouville-type theorem for solutions in a half space.

Essentially, we want a Liouville theorem that states that any solution to
\[\left\{\begin{array}{rcll}
\I u\hspace{-4pt} &= &\hspace{-4pt} 0  &\textrm{ in }\ \{x_n>0\}\\
u \hspace{-4pt}&=& \hspace{-4pt} 0      &\textrm{ in }\ \{x_n<0\},
\end{array}\right.\]
satisfying the growth control at infinity
\[\|u\|_{L^\infty(B_R)} \le C R^{s+\alpha+\beta} \quad \mbox{for all } R\ge 1\]
must be of the form
\[u(x)= (x_n)_+^s(p\cdot x+b).\]
However, for functions having such growth at infinity (recall that $s+\alpha+\beta$ could be $2s+\gamma$), the operator $\I u$ is not defined.

The correct form of such Liouville theorem is the following.

\begin{thm}\label{thm-liouv-1+s}
Let $\bar\alpha>0$ be the exponent given by Proposition \ref{thm1}.
Assume that $u\in C(\R^n)$ satisfies in the viscosity sense
\[\left\{
\begin{array}{rcll}
M^+_{\mathcal L_*}\left\{u(\cdot+h)-u\right\}\ge 0\quad \textrm{and}\quad M^-_{\mathcal L_*} \left\{u(\cdot+h)-u\right\}\le0 &\textrm{in}& \{x_n>0\},\\
u=0 &\textrm{in}&\{x_n<0\},
\end{array}\right.
\]
for all $h\in \R^n$ such that $h_n\geq0$.

Assume that for some $\beta\in(0,1)$ and $\alpha\in(0,\bar\alpha)$, $u$ satisfies
\begin{equation}\label{growthcontrol-1+s}
[u/(x_n)_+^s]_{C^\beta(B_R)}\le C R^\alpha \quad \mbox{for all } R\ge 1.
\end{equation}
Then,
\[u(x)= (x_n)_+^s(p\cdot x+b)\]
for some $p\in \R^n$ and $b\in\R$.
\end{thm}

To prove Theorem \ref{thm-liouv-1+s}, we apply Proposition \ref{thm1} to incremental quotients of $u$ in the first $(n-1)$-variables.
After this, rescaling the obtained estimates and using \eqref{growthcontrol-1+s}, we find that such incremental quotients are zero, and thus that $u$ is a 1D solution.
Then, we use that for 1D functions all operators $L\in \mathcal L_*$ coincide up to a multiplicative constant with the fractional Laplacian $(-\Delta)^s$; see Lemma \ref{lem1d}.
Therefore, we only need to prove a Liouville theorem for solutions to $(-\Delta)^sw=0$ in $\R_+$, $w=0$ in $\R_-$, satisfying a growth control at infinity.
This is done in Lemmas~\ref{classification1D} and \ref{liouv-1+s-1D}.

The second ingredient towards \eqref{???} is a compactness argument.
With $u$ as in Theorem \ref{thm2}, and with $u/(x_n)^s\in C^\beta(B_{3/4}^+)$, we suppose by contradiction that \eqref{???} does not hold, and we blow up the fully nonlinear equation at a boundary point (after subtracting appropriate terms to the solution).
We then show that the blow up sequence converges to an entire solution in $\{x_n>0\}$.
Finally, the contradiction is reached by applying the Liouville-type theorem stated above to the entire solution in $\{x_n>0\}$.
For this, we need to develop a boundary version of a method introduced by the second author in \cite{Se, Se2}.
The method was conceived there to prove interior regularity for integro-differential equations with rough kernels.

These are the main ideas used to prove Theorem \ref{thm2}.

The proof of Theorem \ref{thm2-curved} follows the same ideas as the one of Theorem \ref{thm2}.
However, since we will not flatten the boundary of $\Omega$ (since the equation would change too much), then it requires one additional ingredient.

Indeed, we need to show that $L(d^s)$ is $C^\gamma(\overline\Omega)$ for any linear operator $L$ of the form \eqref{L*1}-\eqref{L*2}-\eqref{curved-2}.
This is given by Proposition \ref{Lds-is-Cs}.
Notice that in case of flat boundary one has that $L(x_n)_+^s=0$ for any such operator, so that there is nothing to prove.
To show this in general $C^{2,\gamma}$ domains, we need to flatten the boundary.

After flattening the boundary, any operator $L$ of the form \eqref{L*1}-\eqref{L*2}-\eqref{curved-2} becomes
\[\tilde L(u,x)=\textrm{PV}\int_{\R^n}\bigl(u(x)-u(x+z)\bigr)K(x,z)dz,\]
where
\[K(x,z)=\frac{a_1(x,z/|z|)}{|z|^{n+2s}}+\frac{a_2(x,z/|z|)}{|z|^{n+2s-1}}\,\chi_{B_1}(z)+J(x,z),\]
where $a_1$ is \emph{even} in the second variable, $a_2$ is \emph{odd} in the second variable, and $J$ has a singularity of order $n+2s-1-\gamma$ near the origin.
Moreover, $a_1$, $a_2$, and $J$ are $C^\gamma$ in the $x$-variable.

To prove that $\tilde L((x_n)_+^s,x)$ is a $C^\gamma$ function we have to take advantage of an important cancelation coming from the fact that $a_2$ is odd in the second variable.

\vspace{2mm}

The paper is organized as follows.
In Section \ref{sec-properties} we give some important results on $\mathcal L_*$ and $\mathcal L_0$.
In Section \ref{sec-barriers} we construct some sub and supersolutions that will be used later.
In Section \ref{sec-krylov-method} we prove Proposition \ref{thm1}.
In Sections \ref{sec-liouville-theorems} and \ref{sec-liouville-1+s} we show the Liouville Theorem \ref{thm-liouv-1+s}.
Then, in Section \ref{sec-regularity-compactness} we prove Theorem \ref{thm2}, and in Section \ref{sec-regularity-curved} we prove Theorem \ref{thm2-curved}.
Finally, in Section \ref{calculets} we prove Proposition \ref{Lds-is-Cs}.

\section{Properties of $\mathcal L_*$ and $\mathcal L_0$}
\label{sec-properties}

This section has two main purposes: to show that the class $\mathcal L_*\subset \mathcal L_0$ is the appropriate one to obtain fine boundary regularity results, and to give some important results on $\mathcal L_*$ and $\mathcal L_0$.

\addtocontents{toc}{\protect\setcounter{tocdepth}{1}}  

\subsection{The class $\mathcal L_*$}

For $s\in (0,1)$, we define the ellipticity class $\mathcal L_*=\mathcal L_*(s)$  as the set of all linear
operators $L$ of the form \eqref{L*1}-\eqref{L*2}.

Throughout the paper, the extremal operators (as defined in \cite{CS}) for the class $\mathcal L_*$ are denoted by $M^+$ and $M^-$, that is,
\begin{equation}\label{M+-}
M^+ u(x)= M^+_{\mathcal L_*} u(x)= \sup_{L\in \mathcal L_*} L u(x)\quad \mbox{ and } \quad M^- u(x)=M^-_{\mathcal L_*} u(x)= \inf_{L\in \mathcal L_*} L u(x).
\end{equation}

The following useful formula writes an operator $L \in \mathcal L_*$ as a weighted integral of one dimensional fractional Laplacians in all directions.
\begin{equation}\label{Lasint}
\begin{split}
Lu &= (1-s)\int_{S^{n-1}}d\theta \,\frac{1}{2}\int_{-\infty}^{\infty}dr  \left(\frac{u(x+r\theta)+u(x-r\theta)}{2}-u(x)\right) \frac{\mu(\theta)}{|r|^{n+2s}} \,r^{n-1}\\
&= -\frac{1-s}{2c_{1,s}}\int_{S^{n-1}}d\theta\, \mu(\theta)\,(-\partial_{\theta\theta})^s u(x),
\end{split}
\end{equation}
where
\[-(-\partial_{\theta\theta})^s u(x) = c_{1,s}\int_{-\infty}^{\infty} \left(\frac{u(x+\theta r)+u(x-\theta r)}{2}- u(x)\right)\,\frac{dr}{|r|^{1+2s}}\]
is the one-dimensional fractional Laplacian in the direction $\theta$, whose Fourier symbol is $-|\theta\cdot\xi|^{2s}$.

The following is an immediate consequence of the formula \eqref{Lasint}.

\begin{lem}\label{lem1d}
Let $u$ be a function depending only on variable $x_n$, i.e. $u(x)= w(x_n)$, where $w:\R\to \R$. Then,
\[ L u(x) =  -\frac{1-s}{2c_{1,s}} \left(\int_{S^{n-1}} |\theta_n|^{2s} \mu(\theta)\,d\theta\right) (-\Delta)^s_{\R} w(x_n),\]
where $(-\Delta)^s_{\R}$ denotes the fractional Laplacian in dimension one.
\end{lem}

\begin{proof}
Using \eqref{Lasint} we find
\[
\begin{split}
L u(x) &=  \frac{1-s}{2c_{1,s}} \int_{S^{n-1}}  -(-\Delta)^s_{\R} \bigl( w(x_n + \theta_n\,\cdot\,)\bigr) \,\mu(\theta)\,d\theta
\\
&=   \frac{1-s}{2c_{1,s}} \int_{S^{n-1}} -|\theta_n|^{2s} (-\Delta)^s_{\R} \bigl( w(x_n + \,\cdot\,)\bigr) \,\mu(\theta)\,d\theta,
\end{split}
\]
as wanted.
\end{proof}

Another consequence of \eqref{Lasint} is that $M^+$ and $M^-$ admit the following ``closed formulae'':
\[
M^+ u(x) = \frac{1-s}{2c_{1,s}} \int_{S^{n-1}} \left\{ \Lambda\bigl(-(-\partial_{\theta\theta})^s w(x)\bigr)^+ - \lambda\bigl(-(-\partial_{\theta\theta})^s w(x)\bigr)^-\right\}d\theta
\]
and
\[
M^- u(x) = \frac{1-s}{2c_{1,s}} \int_{S^{n-1}}  \left\{ \lambda\bigl(-(-\partial_{\theta\theta})^s w(x)\bigr)^+ - \Lambda\bigl(-(-\partial_{\theta\theta})^s w(x)\bigr)^-\right\}d\theta.
\]

In all the paper, given $\nu\in S^{n-1}$ and $\beta\in(0,2s)$ we denote by $\varphi^\beta:\R\rightarrow \R$ and $\varphi_\nu^\beta:\R^n\rightarrow \R$ the functions
\begin{equation}\label{phinubeta}
\varphi^\beta (x) := (x_+)^\beta\qquad\textrm{and}\qquad  \varphi_\nu^\beta (x) := (x\cdot\nu)_+^\beta.
\end{equation}

A very important property of $\mathcal L_*$ is the following.

\begin{lem}\label{sol1dRn}
For any unit vector $\nu\in S^{n-1}$, the function $\varphi_\nu^s$ satisfies
$M^+ \varphi_\nu^s= M^-\varphi_\nu^s =0$ in $\{x\cdot \nu>0\}$ and $\varphi_\nu^s=0$ in $\{x\cdot \nu<0\}$.
\end{lem}

\begin{proof}
We use Lemma \ref{lem1d} and the well-known fact that the function $\varphi^s(x)=(x_+)^s$ satisfies  $(-\Delta)_{\R}^s \varphi^s =0$ in $\{x>0\}$; see for instance \cite[Proposition 3.1]{RS-Dir}.
\end{proof}

Next we give a useful property of $M^+$ and $M^-$.

\begin{lem}\label{otherpowers}
Let $\beta \in(0,2s)$, and let $M^+$ and $M^-$ be defined by \eqref{M+-}.
For any unit vector $\nu\in S^{n-1}$, the function $\varphi_\nu^\beta$ satisfies
$M^+\varphi_\nu^\beta(x)  = \overline c(s,\beta)  (x\cdot \nu)^{\beta-2s}$
and
$M^-\varphi_\nu^\beta(x)  = \underline c(s,\beta)  (x\cdot \nu)^{\beta-2s}$
in $\{x\cdot \nu>0\}$, and $\varphi_\nu^\beta=0$ in $\{x\cdot \nu<0\}$.
Here, $\overline c$ and $\underline c$ are constants depending only on $s$, $\beta$, $n$, and ellipticity constants.

Moreover, $\overline c$ and $\underline c$ satisfy $\overline c\ge\underline c$, and they are continuous as functions of the variables $(s,\beta)$ in $\{0<s \le1,\ 0<\beta<2s\}$.
In addition, we have
\begin{equation}\label{adeu1}
 \overline c(s,\beta) >\underline c(s,\beta) >0\quad \mbox{for all }\beta\in(s,2s).
\end{equation}
and
\begin{equation}\label{adeu2}
\lim_{\beta\nearrow 2s} \underline c(s,\beta) =
\begin{cases}
+\infty \quad &\mbox{for all } s\in (0,1)\\
C>0 &\mbox{for } s=1.
\end{cases}
\end{equation}
\end{lem}

\begin{proof}
Given $L\in\mathcal L_*$, by Lemma \ref{lem1d} we have
\[ L\varphi_\nu^\beta(x) =  -\frac{1-s}{2c_{1,s}} \left(\int_{S^{n-1}} |\theta_n|^{2s} \mu(\theta)\,d\theta\right) (-\Delta)^s_{\R} \varphi^\beta(x\cdot\nu). \]
Hence, using the scaling properties of the fractional Laplacian and of the function $\varphi^\beta$ we obtain that, for $x\cdot\nu>0$,
\[M^+ \varphi_\nu^\beta (x) = C \,(x\cdot\nu)^{\beta-2s}\,\max\left\{-\Lambda (-\Delta)^s_{\R} \varphi^\beta(1), -\lambda (-\Delta)^s_{\R} \varphi^\beta(1) \right\} \]
and
\[M^- \varphi_\nu^\beta (x) = C \,(x\cdot\nu)^{\beta-2s}\,\min\left\{-\Lambda (-\Delta)^s_{\R} \varphi^\beta(1), -\lambda (-\Delta)^s_{\R} \varphi^\beta(1) \right\}, \]
where $C = (1-s)/(2c_{1,s})>0$.

Therefore, to prove that the two functions $\overline c$ and $\underline c$ are continuous in the variables $(s,\beta)$ in $\{0<s \le1,\, 0<\beta<2s\}$, and that \eqref{adeu1}-\eqref{adeu2} holds,
it is enough to prove the same for
\[(s,\beta) \longmapsto -(-\Delta)^s_{\R} \varphi^\beta(1).\]

We first prove continuity in $\beta$.
If $\beta$ and $\beta'$ belong to $(0,2s)$, then as $\beta' \to \beta$, we have $\varphi^{\beta'}\to \varphi^{\beta}$ in $C^2([1/2,3/2])$ and \[\int_{\R} \bigl| \varphi^{\beta'} - \varphi^{\beta}\bigr|(x) \,(1+|x|)^{-1-2s}\,dx\to 0.\]
As a consequence,  $(-\Delta)^s_{\R} \varphi^{\beta'}(1)\to (-\Delta)^s_{\R} \varphi^{\beta}(1)$.
It is easy to see that if $s$ and $s'$ belong to $(0,1]$, and $\beta<2s$, then $(-\Delta)^{s'}_{\R} \varphi^{\beta}(1)\to (-\Delta)^s_{\R} \varphi^{\beta}(1)$ as $s'\to s$.

Moreover, note that whenever  $\beta>s$, the function $\varphi^\beta$ is touched by below by the function  $\varphi^s-C$ at some point $x_0>0$ for some constant $C>0$.
Hence, we have $(-\Delta)^s_{\R} \varphi^\beta(x_0) >(-\Delta)^s_{\R}\varphi^s(x_0) = 0$.
This yields \eqref{adeu1}.

Finally, \eqref{adeu2} follows from an easy computation using the definition of $(-\Delta)^s_{\R}$, and thus the proof is finished.
\end{proof}

\subsection{The class $\mathcal L_0$}

As defined in \cite{CS},  for $s\in (0,1)$ the ellipticity class $\mathcal L_0=\mathcal L_0(s)$ consists of all operators $L$ of the form
\[
L u(x) = (1-s)\int_{\R^n} \left(\frac{u(x+y)+u(x-y)}{2}-u(x)\right) \frac{b(y)}{|y|^{n+2s}}\,dy.
\]
where
\[b\in L^\infty(\R^{n}) \quad \mbox{satisfies} \quad b(y)=b(-y)\quad \mbox{and}\quad \lambda\le b\le \Lambda.\]
It is clear that
\[\mathcal L_*\subsetneq \mathcal L_0.\]
The extremal operators for the class $\mathcal L_0$ are denoted here by $M^+_{\mathcal L_0}$ and $M^-_{\mathcal L_0}$. Since $\mathcal L_*\subset \mathcal L_0$, we have
\[M^-_{\mathcal L_0} \le M^- \le M^+\le M^+_{\mathcal L_0}.\]
Hence, all elliptic equations with respect to $\mathcal L_*$ are elliptic with respect to $\mathcal L_0$ and all the definitions and results in \cite{CS} apply to the elliptic equations considered in this paper.

As in \cite{CS,CS2} we consider the weighted $L^1$ spaces $L^1(\R^n,\omega_s)$, where
\begin{equation}\label{omegas}
\omega_s(x) = (1-s)(1+ |x|)^{-n-2s}.
\end{equation}
The utility of this weighted space is that, if $L\in\mathcal L_0(s)$, then $Lu(x)$ can be evaluated classically and is continuous in $B_{\epsilon/2}$ provided $u\in C^2(B_{\epsilon})\cap L^1(\R^n,\omega_s)$.
One can then consider viscosity solutions to elliptic equations with respect to $\mathcal L_0(s)$ which are not bounded but belong to $L^1(\R^n, \omega_s)$.
The weighted norm appears in stability results; see \cite{CS2}.

As said in the Introduction, the definitions we follow of viscosity solutions and viscosity inequalities are the ones in \cite{CS}.

Next we state the interior Harnack inequality and the $C^\alpha$ estimate from \cite{CS}.

\begin{thm}[\cite{CS}]\label{Harnack}
Let $s_0\in(0,1)$ and $s\in [s_0,1]$.
Let $u\ge 0$ in $\R^n$ satisfy in the viscosity sense $M^{-}_{\mathcal L_0} u \le C_0$ and $M^{+}_{\mathcal L_0}u\ge -C_0$ in $B_{R}$.
Then,
\[u(x)\le C\bigl( u(0) + C_0R^{2s}\bigr) \quad \mbox{for every }x\in B_{R/2},\]
for some constant $C$ depending only on $n$, $s_0$, and ellipticity constants.
\end{thm}

\begin{thm}[\cite{CS}]\label{Interior-Caff-Silv}
Let $s_0\in(0,1)$ and $s\in [s_0,1]$.
Let $u\in C(\overline {B_1})\cap L^1(\R^n,\omega_s)$ satisfy in the viscosity sense $M^{-}_{\mathcal L_0} u \le C_0$ and $M^{+}_{\mathcal L_0}u\ge -C_0$ in $B_1$.
Then, $u\in C^\alpha\bigl(\overline{B_{1/2}}\bigr)$ with the estimate
\[  \|u\|_{C^\alpha(B_{1/2})} \le C\bigl( C_0 + \|u\|_{L^\infty(B_1)}+\|u\|_{L^1(\R^n,\,\omega_s)}\bigr),\]
where $\alpha$ and $C$ depend only on $n$, $s$, and ellipticity constants.
\end{thm}

%
%

\subsection{No fine boundary regularity for $\mathcal L_0$}

The aim of this subsection is to show that the class $\mathcal L_0$ is too large for all solutions to behave comparably near the boundary.
Moreover, we give necessary conditions on a subclass $\mathcal L\subset\mathcal L_0$ to have comparability of all solutions near the boundary.
These necessary conditions lead us to the class $\mathcal L_*$.

In the next result we show that, for any scale invariant class $\mathcal L\subseteq \mathcal L_0$ that contains the fractional Laplacian $(-\Delta)^s$, and any unit vector $\nu$, there exist powers $0\leq \beta_1\leq s\leq \beta_2$ such that $M^+_{\mathcal L} \varphi_\nu^{\beta_1} = 0$ and $M^-_{\mathcal L}\varphi_\nu^{\beta_2} = 0$ in $\{x\cdot \nu >0\}$.
Before stating this result, we give the following

\begin{defi}\label{defi-scale-invariant}
We say that a class of operators $\mathcal L$ is \emph{scale invariant} of order $2s$ if for each operator $L$ in $\mathcal L$, and for all $R>0$, the rescaled operator $L_R$, defined by
\[ (L_R u)(R\,\cdot\,) = R^{-2s}L \bigl( u(R \,\cdot\,)\bigr),\]
also belongs to $\mathcal L$.
\end{defi}

The proposition reads as follows.

\begin{prop}\label{prop-2.7}
Assume that $\mathcal L\subset \mathcal L_0(s)$ is scale invariant of order $2s$.
Then,
\begin{itemize}
\item[(a)] For every $\nu\in S^{n-1}$ and $\beta\in(0,2s)$ the function $\varphi_\nu^\beta$  defined in \eqref{phinubeta} satisfies
\begin{equation}\label{simiribonanit}\begin{split}
M^+_{\mathcal L} \varphi_\nu^\beta(x)  &= \overline C(\beta,\nu)  (x\cdot \nu)^{\beta-2s}\quad \textrm{ in}\quad\{x\cdot \nu>0\},\\
M^-_{\mathcal L}\varphi_\nu^\beta(x) & = \underline C(\beta,\nu)  (x\cdot \nu)^{\beta-2s}\quad \textrm{ in}\quad\{x\cdot \nu>0\}.
\end{split}\end{equation}
Here, $\overline C$ and $\underline C$ are constants depending only on $s$, $\beta$, $\nu$, $n$, and ellipticity constants.

\item[(b)]  The functions $\overline C$ and $\underline C$ are continuous in $\beta$ and, for each unit vector $\nu$, there are $\beta_1\le \beta_2$ in $(0,2s)$ such that
\begin{equation}\label{beta12}
\overline C(\beta_1,\nu)=0\quad \textrm{and}\quad \underline C(\beta_2,\nu)=0.
\end{equation}
Moreover, for all $\beta \in(0,2s)$,
\begin{equation} \label{samesign1}
 \overline C(\beta,\nu) -\overline C(\beta_1,\nu)  \mbox{ has the same sign as } \beta-\beta_1
 \end{equation}
and
\begin{equation} \label{samesign2}
 \underline C(\beta,\nu) -\underline C(\beta_2,\nu)  \mbox{ has the same sign as } \beta-\beta_2.
\end{equation}

\item[(c)] If in addition the fractional Laplacian $-(-\Delta)^s$ belongs to $\mathcal L$, then we have $\beta_1\le s\le \beta_2$.
\end{itemize}
\end{prop}

\begin{proof}
The scale invariance of $\mathcal L$ is equivalent to a scaling property of the extremal operators $M^+_{\mathcal L}$ and $M^-_{\mathcal L}$.
Namely, for all $R>0$, we have
\[M^{\pm}_{\mathcal L} \bigl( u(R\,\cdot\, )\bigr) = R^{2s}(M^\pm_{\mathcal L}  u)(R\,\cdot\,).\]

{\rm (a)\ }  By this scaling property it is immediate to prove that given  $\beta\in (0,2s)$ and $\nu\in S^{n-1}$, the function $\varphi_\nu^\beta$ satisfies \eqref{simiribonanit},
where
\[\overline  C(\beta,\nu):= M^+_{\mathcal L}\varphi_\nu^\beta (\nu)\qquad \textrm{and}\qquad \underline  C(\beta,\nu):= M^-_{\mathcal L}\varphi_\nu^\beta (\nu).\]
Of course, $\overline C$ and $\underline C$ depend also on $s$ and the ellipticity constants, but these are fixed constants in this proof.

{\rm (b)\ } Note that, as $\beta' \to \beta\in [0,2s)$, we have $\varphi_\nu^{\beta'}\to \varphi_\nu^{\beta}$ in $C^2(\overline{B_{1/2}(\nu)})$  and in $L^1(\R^n,\omega_s)$.
As a consequence,  $\underline C$ and $\overline C$ are continuous in $\beta$ in the interval $[0,2s)$.  Since $\varphi_\nu^\beta \to \chi_{\{x\cdot \nu>0\}}$ as $\beta \to 0$, we have that
\[\underline C (\nu, 0) \le \overline C(\nu,0)<0.\]
On the other hand, it is easy to see that
\[ M^-_{\mathcal L_0} \varphi_\nu^\beta(\nu) \longrightarrow +\infty\quad \mbox{as }\beta \nearrow 2s. \]
Hence, using that $M^-_{\mathcal L_0}\le M^-_{\mathcal L}$, we obtain
\[0<\underline C (\nu, \beta) \le \overline C(\nu,\beta) \quad \mbox{for }\beta \mbox{ close to }2s.\]

Therefore, by continuity, there are $\beta_1$ and $\beta_2$ in $(0,2s)$ such that
\[ \overline C(\beta_1,\nu)=0 \quad \mbox{ and }\quad \underline C(\beta_2,\nu)=0.\]

To prove \eqref{samesign1}, we observe that if $\beta>\beta_1$ the function $\varphi_\nu^\beta$ is be touched by below by $\varphi_\nu^{\beta_1}-C$ at some $x_0\in \{x\cdot\nu>0\}$ for some $C>0$.
It follows that
\[M^+_{\mathcal L }\varphi_\nu^\beta(x_0) - M^+_{\mathcal L }\varphi_\nu^{\beta_1}(x_0) \ge M^-_{\mathcal L_0} \bigl(\varphi_\nu^\beta-\varphi_\nu^{\beta_1}\bigr)(x_0)>0.\]
Since the sign of $M^+_{\mathcal L }\varphi_\nu^\beta$ is constant in $\{x\cdot \nu>0\}$ it follows that  $\overline C(\nu,\beta)>0$ when $\beta>\beta_1$.
Similarly one proves that  $\overline C(\nu,\beta)<0$ when  $\beta<\beta_1$, and hence \eqref{samesign2}.

{\rm ({c})\ } It is an immediate consequence of the results in parts (a) and (b) and the fact that $-(-\Delta)^s \varphi_\nu^s=0$ in $\{x\cdot\nu>0\}$.
\end{proof}

Clearly, to hope for some good description of the boundary behavior of solutions to all elliptic equations with respect to a scale invariant class $\mathcal L$,
it must be $\beta_1 = \beta_2$ for every direction $\nu$.
Typical classes $\mathcal L$ contain the fractional Laplacian $-(-\Delta)^s$.
Thus, for them, we must have $\beta_1=\beta_2=s$ for all $\nu\in S^{n-1}$.
If this happens, then
\begin{equation}\label{allL}
L \varphi_\nu^s = 0  \quad \mbox{ in } \{x\cdot \nu >0\}\quad \mbox{for all }L\in \mathcal L, \mbox{ and for all }\nu \in S^{n-1},
\end{equation}
since $M^-_{\mathcal L}\leq L\leq M^+_{\mathcal L}$ for all $L\in \mathcal L$.

As a consequence, we find the following.

\begin{cor}\label{cor-condition}
Let $\beta_1$, $\beta_2$ be given by \eqref{beta12} in Proposition \ref{prop-2.7}.
Then, for the classes $\mathcal L_0$, $\mathcal L_1$, and $\mathcal L_2$ we have $\beta_1<s<\beta_2$.
\end{cor}

\begin{proof}
Let us show that for $\mathcal L=\mathcal L_0$ the condition \eqref{allL} is not satisfied.
Indeed, we may easily cook up $L\in \mathcal L_0$ so that  $L \varphi_{e_n}^s (x',1)\neq 0$ for $x'\in \R^{n-1}$.
Namely, if we take
\[b(y) = \left(\lambda+ (\Lambda-\lambda)\chi_{B_{1/2}}(y)\right),\]
then at points $x=(x',1)$ we have
\[0>L \varphi_{e_n}^s (x) = (1-s)\int_{\R^n} \left(\frac{u(x+y)+u(x-y)}{2}-u(x)\right) \frac{ b(y)}{|y|^{n+2s}}\,dy,\]
since $\varphi_{e_n}^s$ is concave in $B_{1/2}(x',1)$ and $(-\Delta)^s \varphi_{e_n}^s=0$ in $\{x_n>0\}$.

By taking an smoothed version of $b(y)$, we obtain that both $\mathcal L_1$ and $\mathcal L_2$ fail to satisfy \eqref{allL}.
\end{proof}

By the results in Subsection 2.1, we have that the class $\mathcal L_*$ satisfies the necessary condition \eqref{allL}.
Although we do not have a rigorous mathematical proof, we believe that $\mathcal L_*$ is actually the largest scale invariant subclass of $\mathcal L_0$ satisfying \eqref{allL}.

\section{Barriers}
\label{sec-barriers}

In this section we construct supersolutions and subsolutions that are needed in our analysis.
From now on, all the results are for the class $\mathcal L_*$ (and not for $\mathcal L_0$).

First we give two preliminary lemmas.

\begin{lem}\label{lem11}
Let $s_0\in(0,1)$ and $s\in[s_0,1)$. Let
 \[\varphi^{(1)}(x) = \bigl({\rm dist}(x,B_1)\bigr)^s\quad\mbox{and}\quad \varphi^{(2)}(x) = \bigl({\rm dist}(x,\R^n\setminus B_1)\bigr)^s.\]
 Then,
\begin{equation}\label{eqvarphi1}
0 \le  M^- \varphi^{(1)}(x) \le M^+ \varphi^{(1)}(x) \le  C\left\{1+(1-s)\bigl|\log(|x|-1)\bigr|\right\}  \quad \mbox{in } B_{2}\setminus B_1.
\end{equation}
and
\begin{equation}\label{eqvarphi2}
0\ge M^+ \varphi^{(2)}(x)\ge M^- \varphi^{(2)}(x) \ge   -C\left\{1+(1-s)\bigl|\log(1-|x|)\bigr|\right\} \ \mbox{in }B_{1}\setminus B_{1/2}.
\end{equation}
The constant $C$  depends only on $s_0$, $n$, and ellipticity constants.
\end{lem}

Note that the above bounds are much better than $\bigl||x|-1\bigr|^{-s}$, which would be the expected bound given by homogeneity.
This is since $\varphi^{(1)}$ and $\varphi^{(2)}$ are in some sense close to the 1D solution $(x_+)^s$.

\begin{proof}[Proof of Lemma \ref{lem11}]
Let $L\in \mathcal L_*$.
For points $x\in \R^n$ we use the notation $x=(x',x_n)$ with $x'\in \R^{n-1}$.
To prove \eqref{eqvarphi1} let us estimate $L \varphi^{(1)}(x_\rho) $ where $x_\rho = (0,1+\rho)$ for $\rho\in(0,1)$ and for a generic $L\in \mathcal L_*$.
To do it, we subtract the function $\psi(x)=(x_n-1)_+^s$, which satisfies $L\psi(x_\rho)=0$.
Note that
\[\bigl(\varphi^{(1)}-\psi\bigr)(x_\rho) = 0 \quad\mbox{for all }\rho>0\]
and that, for $|y|<1$,
\[ \bigl| {\rm dist}\,(x_\rho+y, B_1) - (1+\rho+y_n)_+ \bigr| \le C|y'|^2.\]
This is because the level sets of the two previous functions are tangent on $\{y'=0\}$.

Thus,
\[0\le \bigl(\varphi^{(1)}_1-\psi\bigr)(x_\rho+ y) \le
\begin{cases}
C \rho^{s-1} |y'|^2 \quad \mbox{for } y=(y',y_n)\in B_{\rho/2}\\
C |y'|^{2s}         \quad \mbox{for } y=(y',y_n)\in B_1\setminus B_{\rho/2} \\
C |y|^s              \quad \mbox{for } y\in \R^n \setminus B_1.
\end{cases}
\]
The bound in $B_{\rho/2}$ follows from the inequality $a^s-b^s\leq(a-b)b^{s-1}$ for $a>b>0$.

Therefore, we have
\[
\begin{split}
0\le L\varphi^{(1)}(x_\rho) &=  L\bigl(\varphi^{(1)}- \psi\bigr)(x_\rho)
\\
&= (1-s)\int \frac{\bigl(\varphi^{(1)}_1-\psi\bigr)(x_\rho+ y) + \bigl(\varphi^{(1)}_1-\psi\bigr)(x_\rho-y)}{2} \frac{\mu(y/|y|)}{|y|^{n+2s}}\,dy
\\
&\le  C(1-s)\Lambda \left(\int_{B_{\rho/2}}\frac{ \rho^{s-1}|y'|^2 dy}{|y|^{n+2s}}
+\int_{B_1\setminus B_{\rho/2}}\frac{|y'|^{2s} dy}{|y|^{n+2s}}
+\int_{\R^n\setminus B_1}\frac{|y|^s dy}{|y|^{n+2s}}
\right)
\\
&\le   C\bigl(1+(1-s)|\log\rho|\bigr).
\end{split}
\]
This establishes \eqref{eqvarphi1}. The proof of \eqref{eqvarphi2} is similar.
\end{proof}

In the next result, instead, the bounds are those given by the homogeneity.
In addition, the constant in the bounds has the right sign to construct (together with the previous lemma) appropriate barriers.

\begin{lem}\label{lem12}
Let $s_0\in(0,1)$ and $s\in[s_0,1)$.
Let
 \[\varphi^{(3)}(x) = \bigl({\rm dist}(x,B_1)\bigr)^{3s/2}\quad\mbox{and}\quad \varphi^{(4)}(x) = \bigl({\rm dist}(x,\R^n\setminus B_1)\bigr)^{3s/2}.\]
 Then,
\begin{equation}\label{eqvarphi3}
M^- \varphi^{(3)}(x) \ge  c(|x|-1)^{-s/2}  \quad \mbox{for all  }x\in B_{2}\setminus B_1.
\end{equation}
and
\begin{equation}\label{eqvarphi4}
M^- \varphi^{(4)}(x) \ge  c(1-|x|)^{-s/2}- C  \quad \mbox{for all  }x\in B_{1}\setminus B_{1/2}.
\end{equation}
The constants $c>0$ and $C$ depend only on $n$, $s_0$, and ellipticity constants.
\end{lem}

\begin{proof}
Let $L\in \mathcal L_*$. For points $x\in \R^n$ we use the notation $x=(x',x_n)$ with $x'\in \R^{n-1}$. To prove \eqref{eqvarphi4} let us estimate
$L \varphi^{(4)} (x_\rho) $ where $x_\rho = (0,1+\rho)$ for $\rho\in(0,1)$ and for a generic $L\in \mathcal L_*$. To do it we subtract the function $\psi(x)= (1-x_n)_+^{3s/2}$,
which by Lemma \ref{otherpowers} satisfies $L\psi(x_\rho)= c\rho^{-s/2}$ for some $c>0$. We note that
\[\bigl(\varphi^{(4)}-\psi\bigr)(x_\rho) = 0\]
and, similarly as in the proof of Lemma \ref{lem11},
\[0\ge\bigl(\varphi^{(4)}-\psi\bigr)(x_\rho+ y) \ge
\begin{cases}
- C \rho^{3s/2-1} |y'|^2 \quad \mbox{for } y=(y',y_n)\in B_{\rho/2}\\
- C |y'|^{3s}         \quad \mbox{for } y=(y',y_n)\in B_1\setminus B_{\rho/2} \\
- C |y|^{3s/2}              \quad \mbox{for } y\in \R^n \setminus B_1.
\end{cases}
\]
Hence,
\[
\begin{split}
L\varphi^{(4)}&(x_\rho) -c\rho^{-s/2} =  L\bigl(\varphi^{(4)} - \psi\bigr)(x_\rho)
\\
&\ge - C(1-s)\Lambda \left(\int_{B_{\rho/2}}\frac{ \rho^{3s/2-1}|y'|^2 dy}{|y|^{n+2s}}
+\int_{B_1\setminus B_{\rho/2}}\frac{|y'|^{3s} dy}{|y|^{n+2s}}
+\int_{\R^n\setminus B_1}\frac{|y|^{s/2} dy}{|y|^{n+2s}}
\right)
\\
&\ge  - C.
\end{split}
\]
This establishes \eqref{eqvarphi4}. To prove \eqref{eqvarphi3}, we now define $\psi(x)= (x_n-1)_+^{3s/2}$, and we use
Lemma \ref{otherpowers} and the fact that  $\varphi^{(3)}-\psi$ is nonnegative in all of $\R^n$ and vanishes on the positive $x_n$ axis.
\end{proof}

We can now construct the sub and supersolutions that will be used in the next section.

\begin{lem}\label{supersol}
Let $s_0\in(0,1)$ and $s\in[s_0,1)$.
There are positive constants $\epsilon$ and $C$, and a radial, bounded,  continuous function $\varphi_1$ which is $C^{1,1}$ in $B_{1+\epsilon}\setminus\overline B_1$ and
satisfies
\[
\begin{cases}
M^+\varphi_1(x) \le -1   &\mbox{ in } B_{1+\epsilon}\setminus \overline{B_1} \\
\varphi_1(x) = 0 \quad &\mbox{ in } B_1\\
\varphi_1(x) \le C\bigl(|x|-1\bigr)^s  &\mbox{ in } \R^n\setminus B_1\\
\varphi_1(x) \ge 1  &\mbox{ in } \R^n\setminus B_{1+\epsilon}
\end{cases}
\]
The constants $\epsilon$, $c$ and $C$ depend only on $n$, $s_0$, and ellipticity constants.
\end{lem}

\begin{proof}
Let
\[
\psi =
\begin{cases}
2\varphi^{(1)} - \varphi^{(3)} \ & \mbox{in }B_2 \\
1 & \mbox{in }\R^n\setminus B_2.
\end{cases}
\]

By Lemmas \ref{lem11} and \ref{lem12},  for $|x|>1$ it is
\[M^+\psi \le  C\left\{1+(1-s)\bigl|\log(|x|-1)\bigr|\right\}- c(|x|-1)^{-s/2}+C.\]
Hence, we may take $\epsilon>0$ small enough so that $M^+\psi \le -1$ in $B_{1+\epsilon}\setminus \overline{ B_1}$.
We then set $\varphi_1= C\psi$ with $C\ge1$ large enough so that $\varphi_1\ge1$ outside $B_{1+\epsilon}$.
\end{proof}

\begin{lem}\label{subsol}
Let $s_0\in(0,1)$ and $s\in[s_0,1)$.
There is $c>0$, and a radial, bounded, continuous function $\varphi_2$ that satisfies
\[
\begin{cases}
M^-\varphi_2(x) \ge c  &\mbox{ in } B_1\setminus B_{1/2}\\
\varphi_2(x) = 0 \quad &\mbox{ in } \R^n\setminus B_1\\
\varphi_2(x) \ge c\bigl(1-|x|\bigr)^s  &\mbox{ in } B_1\\
\varphi_2(x) \le 1  &\mbox{ in } \overline{B_{1/2}}.
\end{cases}
\]
The constants $\epsilon$, $c$ and $C$ depend only on $n$, $s_0$, and ellipticity constants.
\end{lem}

\begin{proof}
We first construct a subsolution $\psi$ in the annulus $B_{1}\setminus \overline{B_{1-\epsilon}}$, for some small $\epsilon>0$.
Then, using it, we will construct the desired subsolution in $B_1\setminus B_{1/2}$.
Let
\[\psi=\varphi^{(2)} + \varphi^{(4)}.\]

By Lemmas \ref{lem11} and \ref{lem12}, for $1/2<|x|<1$ it is
\[M^-\psi \ge  -C\left\{1+(1-s)\bigl|\log(1-|x|)\bigr|\right\}+ c(1-|x|)^{-s/2}- C.\]
Hence, we can take $\epsilon>0$ small enough so that  $M^-\psi \ge 1$ in $B_{1}\setminus \overline{ B_{1-\epsilon}}$.

Let us now construct a subsolution in $B_1\setminus \overline{B_{1/2}}$ from $\psi$, which is a subsolution only in $B_1\setminus \overline{B_{1-\epsilon}}$.
We consider
\[ \Psi(x) = \max_{0\le k \le N} C^{k} \psi (2^{k/N} x),\]
where $N$ is a large integer and $C>1$.
Notice that, for $C$ large enough, the set $\{ x\in B_1\ : \ \Psi(x)= \psi (x)\}$ is an annulus contained in $B_{1}\setminus \overline{ B_{1-\epsilon}}$.

Consider, for $k\geq0$,
\[A_k = \left\{ x\in B_1\ : \ \Psi(x)= C^{k}\psi (2^{k/N} x)\right\}.\]
Since $A_0\subset B_{1}\setminus \overline{ B_{1-\epsilon}}$, then $\Psi$ satisfies $M^-\Psi\ge1$ in $A_0$.

Observe that $A_k= 2^{-k/N}A_0$, since $C^{-1}\Psi(2^{1/n}x)=\Psi(x)$ in the annulus $\{1/2<|x|<2^{-1/n}\}$.
Hence, for $x\in A_k$ we have $2^{k/N}x\in A_0\subset B_{1}\setminus \overline{ B_{1-\epsilon}}$ and
\[M^-\Psi(x)> M^- \bigl(C^{k}\psi (2^{k/N} \,\cdot\,) \bigr)(x) = C^k 2^{2sk/N} M^- \psi(2^{k/N}x)>1.\]

We then set $\varphi_2= c\Psi$ with $c>0$ small enough so that $\varphi_2(x) \le 1$ in $\overline{B_{1/2}}$.
\end{proof}

\begin{rem} \label{remsub}
Notice that the subsolution $\varphi_2$ constructed above is $C^{1,1}$ by below in $B_1\setminus \overline B_{1/2}$, in the sense that it can be touched by below by paraboloids.
This is important when considering non translation invariant equations for which a comparison principle for viscosity solutions is not available.
\end{rem}

\section{The Caffarelli-Krylov method}
\label{sec-krylov-method}

The goal of this section is to prove Proposition \ref{thm1}.
Its proof combines the interior H\"older regularity results of Caffarelli and Silvestre \cite{CS} and the next key Lemma.

\begin{lem}\label{lem_main}
Let $s_0\in(0,1)$, $s\in[s_0,1)$, and $u\in C\bigl(\overline{B_1^+}\bigr)$ be a  viscosity solution of \eqref{equaciou}.
Then, there exist $\alpha\in(0,1)$ and $C$ depending only on $n$, $s_0$, and ellipticity constants, such that
\begin{equation}\label{eq:lemmain}
\sup_{B_{r}^+} u/x_n^s - \inf_{B_{r}^+} u/x_n^s  \le C {r}^\alpha\left(C_0+\|u\|_{L^\infty(\R^n)}\right)
\end{equation}
for all ${r}\leq 3/4$.
\end{lem}

To prove Lemma \ref{lem_main} we need two preliminary lemmas.

We start with the first, which is a nonlocal version of Lemma 4.31 in \cite{Kazdan}.
Throughout this section we denote
\[D_{{r}}^* := B_{9{r}/10}\cap \{x_n> 1/10\}.\]

\begin{lem}\label{lemA}
Let $s_0\in(0,1)$ and $s\in[s_0,1)$.
Assume that $u$ satisfies $u\ge0$ in all of $\R^n$ and
\[M^- u \le C_0 \quad \mbox {in } B_{r}^+,\]
for some $C_0>0$. Then,
\begin{equation}\label{eq:lemA}
\inf_{D_{{r}}^*} u/x_n^s \le C \left(\,\inf_{B_{{r}/2}^+} u/x_n^s + C_0 {r}^s\right)
\end{equation}
for all $r\leq 1$, where $C$ is a constant depending only on $s_0$, ellipticity constants, and dimension.
\end{lem}

\begin{proof}
{\em Step 1.}
Assume $C_0=0$. Let us call
\[m = \inf_{D_{r}^*} u/x_n^s\ge 0 .\]
We have
\begin{equation}\label{lowcontrolu}
u\ge m x_n^s \ge m ({r}/10)^s\quad\mbox{in }  D_{{r}}^*.
\end{equation}

Let us scale and translate the subsolution $\varphi_2$ in Lemma \ref{subsol} as follows to use it as lower barrier:
\begin{equation}\label{defpsiR}
\psi_{r}(x):= \textstyle ({r}/10)^s\, \varphi_2\bigl(\frac{10 (x-x_0)}{2{r}}\bigr)\,.
\end{equation}

We then have, for some $c>0$,
\[
\begin{cases}
M^- \psi_{r} \ge 0  & \mbox{in }B_{2{r}/10}(x_0)\setminus B_{{r}/10}(x_0) \\
\psi_{r} = 0 \quad &\mbox{in }\R^n\setminus B_{2 {r}/10}(x_0) \\
\psi_{r} \ge c\bigl(\frac{2{r}}{10}-|x|\bigr)^s & \mbox{in } B_{2/10}(x_0) \\
\psi_{r} \le ({r}/10)^s & \mbox{in }B_{{r}/10}(x_0) .
\end{cases}
\]

It is immediate to verify that $B_{{r}/2}^+$  is covered by balls of radius $2{r}/10$ such that the concentric ball of radius ${r}/10$ is contained in $D_{r}^*$, that is,
\[ B_{{r}/2}^+ \subset \bigcup \left\{B_{2{r}/10}(x_0)\,: \,B_{{r}/10}(x_0)\subset D_{r}^*\right\}.\]
Now, if we choose some ball $B_{{r}/10}(x_0)\subset D_{r}^*$ and define $\psi_{r}$ by \eqref{defpsiR}, then by \eqref{lowcontrolu} we have $u\ge m\psi_{r}$ in $B_{{r}/10}(x_0)$.
On the other hand $u\ge m\psi_{r}$ outside  $B_{2{r}/10}(x_0)$, since $\psi_{r}$ vanishes there and $u\ge 0$ in all of $\R^n$ by assumption.
Finally, $M^+ \psi_{r} \leq 0$, and since $C_0=0$, $M^-u \ge0$ in the annulus   $B_{2{r}/10}(x_0)\setminus B_{{r}/10}(x_0)$.

Therefore, it follows from the comparison principle that $u\ge m\psi_{r}$ in $B_{2{r}/10}(x_0)$.
Since these balls of radius $2{r}/10$ cover $B_{{r}/2}^+$ and  $\psi_{r} \ge c\bigl({\textstyle\frac{2{r}}{10}}-|x|\bigr)^s$ in  $B_{2/10}(x_0)$, we obtain
\[  u\ge c m x_n^s \qquad \mbox{in }B_{{r}/2}^+,\]
which yields \eqref{eq:lemA}.

{\em Step 2.}
If $C_0>0$ we argue as follows.
First, let
\[\phi(x)=\min\bigl\{1,\ 2(x_n)_+^s-(x_n)_+^{3s/2}\bigr\}.\]
By Lemma \ref{otherpowers}, we have that $M^+ \phi \le -c$ in $\{0<x_n<\epsilon\}$ for some $\epsilon>0$ and some $c>0$.
By scaling $\phi$ and reducing $c$, we may assume $\epsilon=1$.

We then consider
\[\tilde u(x) = u(x) +  \frac{C_0}{c}{r}^{2s}\phi(x/{r}).\]
The function $\tilde u$ satisfies in $\{0<x_n< r\}$
\[
M^- \tilde u - M^- u \le M^+\biggl( \frac{C_0}{c}{r}^{2s}\phi(x/{r})\biggr) \le -C_0
\]
and hence
\[ M^-\tilde u \le 0.\]
Using that $u(x)\leq \tilde u(x)\leq u(x)+CC_0r^s(x_n)_+^s$ and applying Step 1 to $\tilde u$, we obtain \eqref{eq:lemA}.
\end{proof}

The second lemma towards Proposition \ref{lem_main} is a nonlocal version of Lemma 4.35 in \cite{Kazdan}.
It is an immediate consequence of the Harnack inequality  of Caffarelli and Silvestre \cite{CS}.

\begin{lem}\label{lemB}
 Let $s_0\in(0,1)$, $s\in[s_0,1)$, ${r}\le1$ , and $u$ satisfy $u\ge0$ in all of $\R^n$ and
\[M^+ u \ge -C_0 \quad \textrm{and}\quad M^- u \le C_0 \quad \mbox{in }\ B_{r}^+.\]
Then,
\[\sup_{D_{r}^*} u/x_n^s \le C \left(\,\inf_{D_{r}^*} u/x_n^s + C_0 {r}^s\right),\]
for some constant $C$ depending only on $n$, $s_0$, and ellipticity constants.
\end{lem}

\begin{proof}
The lemma is a consequence of Theorem \ref{Harnack}.
Indeed, covering the set $D_{r}^*$ with balls contained in $B_{r}^+$ and with radii comparable to ${r}$ ---using the same (scaled) covering for all ${r}$---, Theorem \ref{Harnack} yields
\[ \sup_{D_{r}^*} u\le C\left(\inf_{D_{r}^* } u + C_0 {r}^{2s}\right).\]
Then, the lemma follows by noting that $x_n^s$ is comparable to ${r}^{s}$ in $D_{r}^*$.
\end{proof}

Next we prove Lemma \ref{lem_main}.

\begin{proof}[Proof of Lemma \ref{lem_main}]
First, dividing $u$ by a constant, we may assume that $C_0+\|u\|_{L^\infty(\R^n)}\leq1$.

We will prove that there exist constants $C_1>0$ and $\alpha\in(0,s)$, depending only on $n$, $s_0$, and ellipticity constants, and monotone sequences $(m_k)_{k\ge 1}$ and $(\overline m_k)_{k\ge 1}$ satisfying the following.
For all $k\geq1$,
\begin{equation}\label{eq:prooflem1}
 \overline m_k - m_k =  4^{-\alpha k}\,,\quad -1\le m_k\le m_{k+1}< \overline m_{k+1}\le \overline m_k\le 1\,,
\end{equation}
and
\begin{equation}\label{eq:prooflem2}
 m_k \le C_1^{-1}u/x_n^s \le \overline m_k \quad \mbox{in } B_{{r}_k}^+ \,, \quad \mbox{where } {r}_k =  4^{-k} \,.
\end{equation}
Note that  since $u=0$ in $B_1^-$ then we have that \eqref{eq:prooflem2} is equivalent to the following inequality in $B_{{r}_k}$ instead of $B_{{r}_k}^+$
\begin{equation}\label{eq:prooflem3}
 m_k (x_n)_+^s \le C_1^{-1}u \le \overline m_k (x_n)_+^s \quad \mbox{in } B_{{r}_k}\,, \quad \mbox{where } {r}_k = 4^{-k} \,.
\end{equation}

Clearly, if such sequences exist, then \eqref{eq:lemmain} holds for all ${r}\le 1/4$ with $C=4^\alpha C_1$.
Moreover, for $1/4<r\leq3/4$ the result follows from \eqref{ramon} below.
Hence, we only need to construct $\{m_k\}$ and $\{\overline m_k\}$.

Next we construct these sequences by induction.

Using the supersolution $\varphi_1$ in Lemma \ref{supersol} we find that
\begin{equation}\label{ramon}
-\frac{C_1}{2}(x_n)^s_+\leq u \leq \frac{C_1}{2}(x_n)^s_+\quad \textrm{in}\ B_{3/4}^+
\end{equation}
whenever $C_1$ is large enough.
Thus, we may take $m_1=-1/2$ and $\overline m_1=1/2$.

Assume now that we have sequences up to $m_k$ and $\overline m_k$.
We want to prove that there exist $m_{k+1}$ and $\overline m_{k+1}$ which fulfill the requirements.
Let
\[u_k = C_1^{-1}u - m_k (x_n)_+^s\,.\]

We will consider the positive part $u_{k}^+$ of $u_k$ in order to have a nonnegative function in all of $\R^n$ to which we can apply Lemmas \ref{lemA} and \ref{lemB}.
Let $u_{k}= u_k^+-u_k^-$. Observe that, by induction hypothesis,
\[u_k^+ = u_k  \quad\mbox{and}\quad u_k^-= 0 \quad \mbox{in }B_{{r}_k}\,.\]
Moreover, $C_1^{-1}u \ge m_j(x_n)_+^s$ in $B_{{r}_j}$ for each $j\le k$.
Therefore, we have
\[u_k \ge (m_j-m_k) (x_n)_+^s \ge (m_{j}-\overline m_{j}+\overline m_k-m_k) (x_n)_+^s = (-4^{-\alpha j}  +4^{-\alpha k}) (x_n)_+^s \quad \mbox{in } B_{{r}_j}.\]
But clearly $0\le (x_n)_+^s \le {r}_{j}^s$ in $B_{{r}_j}$, and therefore using ${r}_j= 4^{-j}$
\[u_k \ge -  {r}_j^s({r}_j^\alpha-{r}_k^\alpha) \quad \mbox{in }\ B_{{r}_j}\ \mbox{for each}\ \ j\le k\,.\]

Thus, since for every $x\in B_{1}\setminus B_{{r}_k}$ there is $j<k$ such that
\[ |x|< {r}_j =  4^{-j} \le 4|x|,\]
we find
\begin{equation}\label{eq:pflem3}
u_{k}(x)\ge -  {r}_k^{\alpha+s}  \biggl|\frac{4x}{{r}_k}\biggr|^s \biggl(\biggl|\frac{4x}{{r}_k}\biggr|^\alpha - 1\biggr) \quad \mbox{outside } B_{{r}_k}\,.
\end{equation}

Now let $L\in \mathcal L_*$.
Using \eqref{eq:pflem3} and that $u_k^-\equiv0$ in $B_{r_k}$, then for all $x\in B_{{r}_k/2}$ we have
\[
\begin{split}
0\le L u_{k}^- (x) &= (1-s)\int_{x+y\notin B_{{r}_k}} u_k^-(x+y)\frac{\mu(y/|y|)}{|y|^{n+2s}}\,dy\\
&\le (1-s)\,\int_{|y|\ge {r}_k/2}  {r}_k^{\alpha+s} \biggl|\frac{8y}{{r}_k}\biggr|^s \biggl(\biggl|\frac{8y}{{r}_k}\biggr|^\alpha - 1\biggr) \frac{\Lambda}{ |y|^{n+2s}}\,dy \\
&= (1-s)\Lambda {r}_k^{\alpha-s}\int_{|z|\ge 1/2} \frac{|8z|^{s}(|8z|^\alpha-1)}{|z|^{n+2s}}\,dz\\
&\le \varepsilon_0  {r}_k^{\alpha-s},
\end{split}
\]
where $\varepsilon_0= \varepsilon_0(\alpha)\downarrow 0$ as $\alpha\downarrow 0$ since $|8z|^\alpha\rightarrow 1$.
Since this can be done for all $L\in \mathcal L_*$, $u_k^-$ vanishes in $B_{{r}_k}$ and satisfies pointwise
\[ 0\le M^- u_k^- \le M^+ u_m^- \le \varepsilon_0  {r}_k^{\alpha-s} \quad \mbox{in }B_{{r}_k/2}^+.\]

Therefore, recalling that
\[u_k^+=C_1^{-1}u-m_k(x_n)_+^s+u_k^-,\]
and using that $M^+(x_n)_+^s = M^-(x_n)_+^s =0$ in $\{x_n>0\}$, we obtain
\[\begin{split}
M^- u_{k}^+ &\le C_1^{-1} M^- u + M^+  (u_k^-)\\
&\le  C_1^{-1}  + \varepsilon_0  {r}_k^{\alpha -s} \qquad\textrm{ in }B_{{r}_k/2}^+.
\end{split}\]
Also clearly
\[ M^+ u_k^+ \ge M^+ u_k \ge -C_{1}^{-1}\qquad\textrm{ in }B_{{r}_k/2}^+. \]

Now we can apply Lemmas \ref{lemA} and \ref{lemB} with $u$ in its statements replaced by $u_{k}^+$.
Recalling that
\[\textstyle u_{k}^+ = u_k =  C_1^{-1}u- m_k x_n^s \quad \mbox{in }B_{{r}_k}^+,\]
we obtain
\begin{equation}\label{eq:pflema1}
\begin{split}
\sup_{D_{{r}_k/2}^*} (C_1^{-1}u/x_n^s-m_k) &\le C \biggl(\inf_{D_{{r}_k/2}^*} (C_1^{-1}u/x_n^s-m_k) +  C_1^{-1} {r}_k^s +\varepsilon_0 {r}_k^\alpha \biggr)\\
&\le C \biggl(\inf_{B_{{r}_k/4}^+} (C_1^{-1}u/x_n^s-m_k)+ C_1^{-1}{r}_k^s +\varepsilon_0 {r}_k^\alpha \biggr)\,.
\end{split}
\end{equation}

On the other hand, we can repeat the same reasoning ``upside down'', that is, considering the functions $\overline{u}_k = \overline m_k (x_n)_+^s - u$ instead of $u_k$.
In this way we obtain, instead of \eqref{eq:pflema1}, the following
\begin{equation}\label{eq:pflema2}
\sup_{D_{{r}_k/2}^*} (\overline m_k - C_1^{-1}u/x_n^s) \le C \biggl(\inf_{B_{{r}_k/4}^+} (\overline m^k-C_1^{-1}u/x_n^s)+C_1^{-1}{r}_k^s +\varepsilon_0 {r}_k^\alpha \biggr).
\end{equation}

Adding \eqref{eq:pflema1} and \eqref{eq:pflema2} we obtain
\[\begin{split}
\overline m_k-m_k &\le C \biggl(\inf_{B_{{r}_k/4}^+} (C_1^{-1}u/x_n^s-m_k) + \inf_{B_{{r}_k/4}^+} (\overline m_k-C_1^{-1}u/x_n^s) + C_1^{-1}{r}_k^s +\varepsilon_0 {r}_k^\alpha \biggr)\\
&= C\biggl(\inf_{B_{{r}_{k+1}}^+} C_1^{-1}u/x_n^s - \sup_{B_{{r}_{k+1}}^+} C_1^{-1}u/x_n^s +\overline m_k-m_k+ C_1^{-1}{r}_k^s +\varepsilon_0 {r}_k^\alpha\biggr).
\end{split}\]
Thus, using that $\overline m_k-m_k = 4^{-\alpha k}$,  $\alpha<s$, and ${r}_k = 4^{-k}\le 1$, we obtain
\[\sup_{B_{{r}_{k+1}}^+} C_1^{-1}u/x_n^s - \inf_{B_{{r}_{k+1}}^+}  C_1^{-1}u/x_n^s \le \bigl(\textstyle \frac{C-1}{C} +C_1^{-1} +\varepsilon_0\bigr) 4^{-\alpha k}\,.\]

Now we choose $\alpha$ small and $C_1$ large enough so that
\[\frac{C-1}{C} +C_1^{-1}  +\varepsilon_0(\alpha) \le 4^{-\alpha}.\]
This is possible since $\varepsilon_0(\alpha)\downarrow 0$ as $\alpha\downarrow 0$ and the constant $C$ depends only on $n$, $s_0$, and ellipticity constants.
Then, we find
\[\sup_{B_{{r}_{k+1}}^+} C_1^{-1}u/x_n^s - \inf_{B_{{r}_{k+1}}^+}  C_1^{-1}u/x_n^s \le  4^{-\alpha (k+1)},\]
and thus we are able to choose $m_{k+1}$ and $\overline m_{k+1}$ satisfying \eqref{eq:prooflem1} and \eqref{eq:prooflem2}.
\end{proof}

To end this section, we give the

\begin{proof}[Proof of Proposition \ref{thm1}]
Let $x\in B_{1/2}^+$ and let $x_0$ be its nearest point on $\{x_n=0\}$.
Let
\[d={\rm dist\,}(x, x_0)= x_n = {\rm dist\,}(x, B_1^-).\]
By Theorem \ref{Interior-Caff-Silv} (rescaled), we have
\[ \|u\|_{C^\alpha\left(B_{d/2}(x)\right)} \le  C d^{-\alpha} \left(\|u\|_{L^\infty(\R^n)} + C_0\right).\]
Hence, since $\|(x_n)^{-s}\|_{C^\alpha\left(B_{d/2}(x)\right)}\leq Cd^{-s}$, then for ${r}\le d/2$
\begin{equation} \label{hola1}
 {\rm osc}_{B_{{r}}(x)} u/x_n^s  \le C {r}^\alpha d^{-s-\alpha} \left(\|u\|_{L^\infty(\R^n)} + C_0\right).
\end{equation}

On the other hand, by Lemma \ref{lem_main}, for all $r\geq d/2$ we have
\begin{equation}\label{hola2}
{\rm osc}_{B_{r}(x)\cap B_{3/4}^+} u/x_n^s \le   C r^\alpha \left(\|u\|_{L^\infty(\R^n)} + C_0\right).
\end{equation}
In both previous estimates $\alpha\in(0,1)$ depends only on $n$, $s_0$, and ellipticity constants.
Let us call
\[M =\left(\|u\|_{L^\infty(\R^n)} + C_0\right). \]
Then, given $\theta>1$ we have the following alternatives
\begin{enumerate}
\item[(i)] If ${r}\le d^\theta/2$ then, by \eqref{hola1},
\[{\rm osc}_{B_{{r}}(x)} u/x_n^s  \le C{r}^\alpha d^{-s-\alpha} M \le C{r}^{\alpha -(s+\alpha)/\theta} M.\]

\item[(ii)] If $ d^\theta/2 <{r} \leq d/2 $ then, by \eqref{hola2},
\[{\rm osc}_{B_{{r}}(x)} u/x_n^s  \le {\rm osc}_{B_{d/2}(x)} u/x_n^s  \leq Cd^\alpha M \le C{r}^{\alpha/\theta} M.\]

\item[(iii)] If $d/2<r$, then by \eqref{hola2}
\[{\rm osc}_{B_{{r}}(x)\cap B_{3/4}^+} u/x_n^s  \leq Cr^\alpha M.\]
\end{enumerate}
Choosing $\theta>\frac{s+\alpha}{\alpha}$ (so that the exponent in (i) is positive), we obtain
\begin{equation}\label{mmouse}
 {\rm osc}_{B_{{r}}(x)\cap B_{3/4}^+} u/x_n^s  \le C {r}^{{\alpha'}} M \quad \mbox{whenever }x\in B_{1/2}^+\quad \mbox{and}\quad {r}>0,
\end{equation}
for some ${\alpha'}\in (0,\alpha)$.
This means that $\|u/x_n^s\|_{C^{\alpha'}(B_{1/2}^+)}\leq CM$, as desired.
\end{proof}

\section{Liouville theorems in $\R^n_+$}
\label{sec-liouville-theorems}

The goal of this section is to prove Theorem \ref{thm-liouv-1+s}.

First, as a consequence of Proposition \ref{thm1} we show the following Liouville-type result involving the extremal operators.
Note that the growth condition $CR^\beta$ in this lemma holds for $\beta<s+\bar\alpha$ (with $\bar\alpha$ small), in contrast with the Liouville Theorem~\ref{thm-liouv-1+s}.

\begin{prop}\label{liouville-krylov}
Let $s_0\in (0,1)$ and $s\in [s_0,1)$.
Let $\bar\alpha>0$ be the exponent given by Proposition \ref{thm1}.
Assume that $u\in C(\R^n)$ is a viscosity solution of
\begin{eqnarray*}
M^+ u\ge 0\quad \mbox{and}\quad M^- u\le0 &\mbox{in}& \{x_n>0\},\\
 u=0 &\mbox{in}&\{x_n<0\}.
 \end{eqnarray*}
Assume that, for some positive $\beta<s+\bar\alpha$, $u$ satisfies the growth control at infinity
\begin{equation}\label{growthcontrol2}
\|u\|_{L^\infty(B_R)} \le CR^\beta \quad \mbox{for all } R\ge 1.
\end{equation}
Then,
\[u(x)= K (x_n)_+^s\]
for some constant $K\in \R$.
\end{prop}

\begin{proof}
Given $\rho\ge1$, let $v_\rho(x)= \rho^{-\beta}u(\rho x)$.
Note that for all $\rho\ge 1$ the function $v_\rho$ satisfies the same growth control \eqref{growthcontrol2} as $u$.
Indeed,
\[ \|v_\rho\|_{L^\infty(B_R)}= \rho^{-\beta}\|u\|_{L^\infty(B_{\rho R})}\le \rho^{-\beta} C(\rho R)^\beta=CR^\beta.\]
In particular $\|v_\rho\|_{L^\infty(B_1)}\le C$ and $\|v_\rho\|_{L^1(\R^n, \omega_s)}\le C$, with $C$ independent of $\rho$.
Hence, the function $\tilde v_\rho = v_\rho \chi_{B_1}$ satisfies $M^+ \tilde v_\rho\ge -C$ and $M^- \tilde v_\rho\le C$ in $B_{1/2}\cap \{x_n>0\}$, and $\tilde v_\rho=0$ in $\{x_n<0\}$.
Also, $\|\tilde v_\rho\|_{L^\infty(B_{1/2})}\le C$.
Therefore, by Proposition~\ref{thm1} we obtain that
\[  \bigl\| v_\rho/x_n^s\bigr\|_{C^{\alpha}(B_{1/4}^+)}= \bigl\| \tilde v_\rho/x_n^s\bigr\|_{C^{\alpha}(B_{1/4}^+)} \le C .\]
Scaling this estimate back to $u$ we obtain
\[ \bigl[ u/x_n^s\bigr]_{C^{\alpha}(B_{\rho/4}^+)}= \rho^{-\alpha} \bigl[ u(\rho x)/(\rho x_n)^s\bigr]_{C^{\alpha}(B_{1/4}^+)} = \rho^{\beta-s-\alpha } \bigl[ v_\rho/(x_n)^s\bigr]_{C^{\alpha}(B_{1/4}^+)}  \le C\rho^{\beta-s-\alpha}.\]
Using that $\beta<s+\alpha$ and letting $\rho\to \infty$ we obtain
\[ \bigl[ u/x_n^s\bigr]_{C^{\alpha}(\R^n\cap\{x_n>0\})}= 0,\]
which means $u=K\bigl(x_n)_+^s$.
\end{proof}

The previous Proposition will be applied to tangential derivatives of a solution to $u$ in the situation of Theorem \ref{thm-liouv-1+s}.
It will give that $u$ is in fact a function of $x_n$ alone.
To proceed, we will need the following crucial lemmas.
These are Liouville-type results for the fractional Laplacian in dimension 1, and they will be proved in the next section.

In the first one, the growth of the solution $u$ still allows to compute $(-\Delta)^s u$.

\begin{lem}\label{classification1D}
Let $u\in C(\R)$ be a function satisfying $(-\Delta)^su=0$ in $\R_+$, $u\equiv0$ in $\R_-$, and $|u(x)|\leq C(1+|x|^\beta)$ for some $\beta<2s$.
Then, $u(x)=K(x_+)^s$.
\end{lem}

The second one is for functions that grow too much at infinity, so that one cannot compute $(-\Delta)^s u$ ---as in Theorem \ref{thm-liouv-1+s}.

\begin{lem}\label{liouv-1+s-1D}
Let $\phi:\R\longrightarrow\R$ be defined by
\[\phi_{a,b}(x)=a(x_+)^s+b(x_+)^{1+s}.\]
Then,
\begin{itemize}
\item[(a)] For any $a$ and $b$, the function $\phi_{a,b}$ satisfies
\[(-\Delta)^s\left\{\phi_{a,b}(\cdot+h)-\phi_{a,b}\right\}=0\quad\textrm{in}\ (0,\infty).\]

\item[(b)] Let $u\in C(\R)$ be any function such that $u\equiv0$ in $\R_-$ and satisfying
\[(-\Delta)^s\left\{u(\cdot+h)-u\right\}=0\quad\textrm{in}\ \R_+\]
for any $h>0$.
Assume in addition that, for some $\gamma\in (0,1)$ and $\beta\in (0,2s)$,
\[[u/(x_+)^s]_{C^\gamma([0,R])}\leq CR^\beta\qquad \textrm{for all}\ R\geq1.\]
Then,
\[u(x)=\phi_{a,b}(x)\]
for some $a$ and $b$.
\end{itemize}
\end{lem}

We will also need the following observation.

\begin{lem}\label{lem-1d-functions}
Assume that $u$ is a function in $\R^n$ depending only of one variable, i.e., $u(x)=\zeta(x_n)$.
Then, we have
\[M^+u(x)=-c_1(-\Delta)^s_{\R}\zeta(x_n)\]
and
\[M^-u(x)=-c_2(-\Delta)^s_{\R}\zeta(x_n)\]
in the viscosity sense, where $c_1$ and $c_2$ are positive constants.
\end{lem}

\begin{proof}
It follows immediately from Lemma \ref{lem1d}.
\end{proof}

Furthermore, we will use also the following.

\begin{lem} \label{lem-planes}
Let $a\in \R^n$ and $b\in \R$, and define
\[\phi(x)=(x_n)^s_+(a\cdot x+b).\]
Then, for all $h\in\R^n$ with $h_n\geq0$, we have
\[M^+\left\{\phi(\cdot+h)-\phi\right\}=M^-\left\{\phi(\cdot+h)-\phi\right\}=0\quad\textrm{in}\ \{x_n>0\}.\]
\end{lem}

\begin{proof}
It follows immediately from Lemmas \ref{lem-1d-functions} and \ref{liouv-1+s-1D} (a).
\end{proof}

We can now give the:

\begin{proof}[Proof of Theorem \ref{thm-liouv-1+s}]
Take first $h\in S^{n-1}$ such that $h_n=0$, and define
\[v(x)=u(x+h)-u(x).\]
Then, $v$ satisfies
\[\left\{
\begin{array}{rcll}
M^+v\ge 0\quad \textrm{and}\quad M^-v\le0 &\textrm{in}& \{x_n>0\},\\
v=0 &\textrm{in}&\{x_n<0\}.
\end{array}\right.\]
Moreover, by \eqref{growthcontrol-1+s} it also satisfies the growth control
\[\|v/(x_n)_+^s\|_{L^\infty(B_R)}\le C R^\alpha \quad \mbox{for all } R\ge 1,\]
and hence
\[\|v\|_{L^\infty(B_R)}\le C R^{\alpha+s} \quad \mbox{for all } R\ge 1.\]
Thus, it follows from Proposition \ref{liouville-krylov} that
\[v(x)=K(x_n)_+^s.\]
Therefore, we have
\[u(x+h)-u(x)=K(x_n)_+^s\]
whenever $h_n=0$, and this implies that
\[u(x)=(x_n)_+^s(a\cdot x+b)+\psi(x_n)\]
for some 1D function $\psi:\R\longrightarrow\R$.

Now, by Lemmas \ref{lem-planes} and \ref{lem-1d-functions}, we have that for all $h\in\R^n_+$ and all $x\in\R^n_+$
\[M^+\left\{u(\cdot+h)-u\right\}(x)=-c_1\,(-\Delta)^s\left\{\psi(\cdot+h_n)-\psi\right\}(x_n),\]
and the same with $M^-$.
Thus, for any $h>0$ this 1D function $\psi$ satisfies
\[\left\{\begin{array}{rcll}
(-\Delta)^s\left\{\psi(\cdot+h)-\psi\right\}= 0 &\textrm{in}& (0,+\infty),\\
\psi=0 &\textrm{in}&(-\infty,0).
\end{array}\right.\]
Moreover, notice that, by the assumptions of the Theorem, the function $\psi$ satisfies
\[[\psi/(x_+)^s]_{C^\beta([0,R])}\leq CR^\alpha\qquad \textrm{for all}\ R\geq1.\]
Hence, by Lemma \ref{liouv-1+s-1D}, we find that $\psi(x_n)=K_1(x_n)_+^{1+s}+K_2(x_n)_+^s$, and
\[u(x)=(x_n)_+^s(\tilde a\cdot x+\tilde b),\]
as desired.
\end{proof}

\section{Liouville theorems in dimension 1}
\label{sec-liouville-1+s}

The aim of this section is to prove Lemmas \ref{liouv-1+s-1D} and \ref{classification1D}.

To prove them, we need the following result.
It classifies all homogeneous solutions (with no growth condition) that vanish in a half line of the extension problem of Caffarelli and Silvestre \cite{CSext} in dimension $1+1$.

\begin{lem}\label{lemODE}
Let $s\in(0,1)$.
Let $(x,y)$ denote a point in $\R^2$, and $r>0$, $\theta\in (-\pi,\pi)$ be polar coordinates defined by the relations $x=r\cos \theta$, $y= r\sin \theta$.
Assume that $\nu>-s$, and $q_\nu = r^{s+\nu} \Theta_\nu(\theta)$ is even with respect $y$ (or equivalently with respect to $\theta$) and solves
\begin{equation}\label{extension-pb}
\begin{cases}
{\rm div}\,(|y|^{1-2s}\nabla q_\nu)=0 &\mbox{in }  \{y\neq 0\}\\
\lim_{y\to 0} |y|^{1-2s} \partial_y  q_\nu =0 \quad &\mbox{on } \{y=0\}\cap\{x>0\} \\
q_\nu=0 &\mbox{on } \{y=0\}\cap\{x<0\}.
\end{cases}\end{equation}
Then,
\begin{enumerate}
\item[(a)] $\nu$ belongs to $\mathbb N\cup\{0,-1\}$ and
\[
\Theta_\nu(\theta) =
K\,|\sin \theta|^s \, P^{s}_{\nu}\bigl(\cos \theta\bigr),
\]
where $P^\mu_\nu$ is the associated Legendre function of first kind.
Equivalently,
\[
\Theta_\nu(\theta) =
C\,\left|\cos \left(\frac{\theta}{2}\right)\right|^{2s} \phantom{\,}_2F_1\left(-\nu, \nu+1; 1-s; \frac{1-\cos\theta}{2}\right),
\]
where $\,_2F_1$ is the hypergeometric function.

\item[(b)] The functions $\bigl\{\Theta_\nu\bigr\}_{\nu\in\mathbb N\cup\{0\}}$ are a complete orthogonal system in the subspace of even functions of the weighted space $L^2\bigl((-\pi,\pi), |\sin\theta|^{1-2s}\bigr)$.
\end{enumerate}
\end{lem}

\begin{proof}
We differ the proof to the Appendix.
\end{proof}

Using the previous computation, we can now show Lemma \ref{classification1D}.

\begin{proof}[Proof of Lemma \ref{classification1D}]
Let
\[P_s(x,y)  = \frac{p_{1,s}}{y} \frac{1}{\bigl(1+ (x/y)^2\bigr)^{\frac{1+2s}{2}}}\]
be the Poisson kernel for the extension problem of Caffarelli and Silvestre; see \cite{CSext,CabreSire}.

Given the growth control $|u(x)|\le C |x|^\beta$ at infinity with $\beta<2s$, and $|u(x)|\leq C|x|^{\delta-1}$ near the origin with $\delta>0$, the convolution
\[ v(\,\cdot\, ,y) =  u\ast P_s(\,\cdot\,, y)\]
is well defined and is a solution of the extension problem
\[
\begin{cases}
{\rm div} (y^{1-2s} \nabla v) = 0\quad &\mbox{in } \{y>0\}\\
v(x,0)=u(x) & \mbox{for } x\in \R.
\end{cases}
\]
Since $(-\Delta)^s u = 0$ in $\{x>0\}$ and $u=0$ in $\{x<0\}$, the function $v$ satisfies
\[\lim_{y\searrow 0} y^{1-2s} \partial_ y v(x,y) = 0 \quad  \mbox{for } x>0  \quad \mbox{ and }\quad  v(x,0)= 0 \quad  \mbox{for } x<0.\]
Hence, $v$ solves \eqref{extension-pb}.

Let $\Theta_\nu$, $\nu \in \mathbb N \cup \{0\}$, be as in Lemma \ref{lemODE}.
Recall that $r^{s+\nu}\Theta_\nu(\theta)$ also solve \eqref{extension-pb}.
By standard separation of variables, in every ball $B_R^+(0)$ of $\R^2$ the function $v$ can be written as a series
\begin{equation}\label{series-v}
 v(x,y) = v(r\cos\theta, r\sin\theta) = \sum_{\nu=0}^\infty a_\nu r^{s+\nu }\Theta_\nu(\theta).
\end{equation}
To obtain this expansion we use that, by Lemma \ref{lemODE} (b),  the functions
$\bigl\{\Theta_\nu\bigr\}_{\nu\in\mathbb N\cup\{0,-1\}}$ are a complete orthogonal system in the subspace of even functions
in the weighted space $L^2\bigl((-\pi,\pi), |\sin\theta|^{1-2s}\bigr)$, and hence are complete in $L^2\bigl((0,\pi), (\sin\theta)^{1-2s}\bigr)$.

Moreover, by uniqueness, the coefficients $a_\nu$ are independent of $R$ and hence the series \eqref{series-v} provides a representation
formula for $v(x,y)$ in the whole $\{y>0\}$.

Now, we claim that the growth control  $\|u\|_{L^\infty(-R,R)} \le CR^{\beta}$ with $\beta\in (0,2s)$ is transferred to
$v$ (perhaps with a bigger constant $C$), that is,
\[ \|v\|_{L^\infty(B_R^+)}\le CR^{\beta}.\]
To see this, consider the rescaled function $u_R(x)=R^{-\beta}u(Rx)$, which satisfy the same growth control of $u$. Then,
\[v_R = R^{-\beta}v(R\,\cdot\,) = u_R \ast P_s.\]
Since the growth control for $u_R$ is independent of $R$ we find a bound for $\|v_R\|_{L^\infty(B_1^+)}$ that is independent of $R$, and this means that $v$ is controlled by $CR^{\beta}$ in $B_R^+$, as claimed.

Next, since we may assume that $\int_{0}^\pi |\Theta_\nu(\theta)|^2 |\sin \theta|^a \,d\theta =1$ for all $\nu\geq0$, Parseval's identity yields
\begin{equation}\label{Parseval1}
 \int_{\partial^+ B_R} \bigl|v(x,y)\bigr|^2 y^{a} \,d\sigma = \sum_{\nu=0}^\infty |a_\nu|^2 R^{2s+2\nu+1 +a},
\end{equation}
where $\partial^+ B_R=\partial B_R\cap \{y>0\}$.
But by the growth control, we have
\begin{equation}\label{Parseval2}
\int_{\partial^+ B_R} \bigl|v(x,y)\bigr|^2 y^{a} \,d\sigma  \le CR^{2\beta} \int_{\partial^+ B_R} y^{a} \,d\sigma = CR^{2\beta+1+a}.
\end{equation}
Finally, since $2\beta<4s<2s+2$, this implies $a_\nu=0$ for all $\nu\ge1$, and hence
$u(x)= K(x_+)^s$, as desired.
\end{proof}

To establish Lemma \ref{liouv-1+s-1D}, we will need the following extension of Lemma \ref{classification1D}.

\begin{lem}\label{classification1D-1+s}
Let $u$ satisfy $(-\Delta)^s u=0$ in $\R_+$ and $u=0$ in $\R_-$.
Assume that, for some $\delta>0$ and $\beta\in (0,2s)$, $u$ satisfies the growth conditions
\begin{itemize}
\item $|u(x)|\leq C|x|^{\delta-1}$ for all $x\in(0,1)$.
\item $|u(x)|\leq C|x|^\beta$ for all $x\geq1$.
\end{itemize}
Then $u(x) = a(x_+)^s+b(x_+)^{s-1}$.
\end{lem}

\begin{proof}
It is a slight modification of the proof of Lemma \ref{classification1D}.

Indeed, we may consider the extension $v(x,y)$ of $u(x)$, which solves \eqref{extension-pb}.

Now, we consider $\tilde v(x,y)=\int_{-\infty}^x v(x,y)dx$, which also satisfies \eqref{extension-pb}, and satisfies the growth condition
\[ \|v\|_{L^\infty(B_R^+)}\le CR^{\beta+1},\]
with $\beta+1<1+2s$.

Finally, writing $\tilde v$ as in \eqref{series-v}, and using \eqref{Parseval1}-\eqref{Parseval2} and that $2(\beta+1)<2+4s$, we find that $a_\nu=0$ for all $\nu\geq2$.
This yields $\tilde v(x,0)=(x_+)^s(ax+b)$, and hence $u(x)=a(x_+)^s+b(x_+)^{s-1}$.
\end{proof}

We finally give the:

\begin{proof}[Proof of Lemma \ref{liouv-1+s-1D}]
(a) It follows easily by using the extension of Caffarelli-Silvestre \cite{CSext}.

(b) First, notice that $u\in C^\delta([0,1])$, with $\delta=\min\{\gamma,s\}$.
Hence, for each $h>0$, the function $v(x)=u(x+h)-u(x)$ satisfies $v\equiv0$ in $(-\infty,-h)$, $|v|\leq C|h|^\delta$ in $[-h,0]$, and $(-\Delta)^sv=0$ in $(0,\infty)$, and $|v(x)|\leq C |h|^\gamma (1+|x|^{\beta+s})$ in $(0,\infty)$.

Thus, by standard interior regularity (see for example \cite{RS-Dir}), we have that $[v]_{C^{0,1}([h,2h])}\leq C|h|^{\delta-1}$.
In particular,
\[|u'(x+h)-u'(x)|=|v'(x)|\leq C|h|^{\delta-1}\qquad\textrm{for all}\ x\in[h,2h],\ h\in(0,1).\]
And this implies (summing a geometric series) that
\[|u'(x)|\leq C|x|^{\delta-1}\qquad \textrm{for}\ x\in(0,1).\]

On the other hand, since $|v(x)|\leq C|h|^\gamma|x|^{\beta+s}$ for $x>1$, then
\[|v'(R)|\leq [v]_{C^{0,1}([R,2R])}\leq \frac{C}{R}\|v\|_{L^\infty([R/2,3R])}\leq  C|h|^\gamma R^{\beta+s-1}\quad \textrm{for}\ R\geq1.\]
Therefore, it follows that for all $x>1$
\[\begin{split}
|u'(x)|&\leq |u'(1)-u'(2)|+\cdots+|u'(x-1)-u'(x)|\\
& \leq C\left(1^{\beta+s-1}+2^{\beta+s-1}+\cdots+x^{\beta+s-1}\right)\\
&\leq C|x|^{\beta+s}.\end{split}\]

Thus, the function $u'$ satisfies
\begin{itemize}
\item $(-\Delta)^s(u')=0$ in $(0,\infty)$
\item $|u'(x)|\leq C|x|^{\delta-1}$ for $x\in(0,1)$
\item $|u'(x)|\leq C|x|^{s+\beta+\gamma-1}$ for $x>1$
\end{itemize}
and then it follows from Lemma \ref{lemODE} that
\[u'(x)=a(x_+)^s+b(x_+)^{s-1}.\]
Hence, since $u(0)=0$, we find
\[u(x)=a(x_+)^{s+1}+b(x_+)^s,\]
as desired.
\end{proof}

\section{Boundary regularity: flat boundary}
\label{sec-regularity-compactness}

In this section we prove Theorem \ref{thm2}.
The main step towards this result will be Proposition \ref{prop_contr_1+s+alpha} below.

Notice that throughout this Section the operator $\I$ will not be translation invariant but of the form \eqref{op7}, with $L_{ab}$ translation invariant.
Hence, $\I(u,x)$ belongs to a restricted class of non translation invariant operators.
Within this class we can truncate solutions.
Thanks to this, we may assume for example that in $u/(x_n)^s$ is $C^\beta$ in all of $\R^n_+$ and not only in $B_{3/4}^+$ (recall that we want to show \eqref{austin-iter}).

In the following, given $\beta\in[0,1]$ and  $A\subset\R^n$ we denote by
\[[u]_{\beta, A} := \sup_{x,y\in A, x\neq y} \frac{|u(x)-u(y)|}{|x-y|^\beta}.\]
That is, for $\beta=0$ this gives the oscillation and for $\beta\in (0,1)$ this gives the $C^\beta$ seminorm.
We also denote
\[\|u\|_{\beta, A} := \|u\|_{L^\infty(A)} + [u]_{\beta;A}.\]
This notation is appropriate in order to treat at the same time the case $\beta=0$ and $\beta\in (0,1)$.

\begin{prop}\label{prop_contr_1+s+alpha}
Let $s_0\in(0,1)$, and let $\bar \alpha>0$ be the constant given by Proposition \ref{thm1}.

Let $s\in(s_0,1)$, $\alpha\in(0,\bar\alpha)$, $\gamma\in[0,1)$, and $\beta\in[0,1]$ such that $\alpha+\beta\leq \gamma+s$.
Assume in addition that $\alpha+\beta\neq 1$.

Let $u$ be any solution of $\I(u,x)=0$ in $B_1^+$ and $u=0$ in $\R^n_-$, where $\I$ is any fully nonlinear operator elliptic with respect to $\mathcal L_*(s)$ of the form
\begin{equation}\label{op7}
\I(u,x) = \inf_{b\in \mathcal B}\sup_{a\in\mathcal A} \bigl( L_{ab} u(x) + c_{ab}(x)\bigr) \quad \mbox{with }\|c_{ab}\|_{\gamma;B_1^+} \le 1
\end{equation}
for all $a\in \mathcal A$ and $b\in \mathcal B$.

If $\beta>0$ assume in addition that the following estimate holds for some $\beta'\in(\beta,\beta+\alpha)$ and for all $u$ and $\I$ as above:
\begin{equation}\label{recursive_estimate}
 [u/(x_n)^s]_{\beta';B_{1/4}^+}\le C\bigl(\|u/(x_n)^s\|_{0;B_{1}^+}  +\sup_{R\ge 1} R^{-\alpha}[u/(x_n)^s]_{\beta;B_{R}^+}\bigr),
\end{equation}
where $C$  depends only on $n$, $s_0$, ellipticity constants, $\alpha$,  $\beta$, and $\beta'$.

Then, for all $r>0$
\[ r^{-\alpha}[ u/(x_n)^{s} -P_{r}(\cdot)]_{\beta; B_r^+} \le C\bigl(\|u/(x_n)^s\|_{0;B_{1}^+}  +\sup_{R\ge 1} R^{-\alpha}[u/(x_n)^s]_{\beta;B_{R}^+}\bigr),\]
for some constant $C$ that depends only on $n$, $s_0$, ellipticity constants, $\alpha$,  $\beta$, and $\beta'$,
where $P_r(x)$ is defined as the polynomial of degree at most $\lfloor \alpha+\beta\rfloor$ (zero or one) which best fits the function $u/(x_n)^{s}$ in $B_r^+$.
That is,
\[P_r := {\rm arg \, min}_{ P \in \mathcal P}  \int_{B_r^+} \bigl( u(x)/(x_n)^s - P(x)\bigr)^2 \,dx,\]
where $\mathcal P$ is the space of polynomials with real coefficients and degree at most $\lfloor \alpha+\beta\rfloor$.
\end{prop}

We will need the following preliminary results.

\begin{lem}\label{lemweak2}
Let $s_0 \in (0,1)$, $s_m\in [s_0,1]$ be a converging sequence with $s_m\to s$, and $M^+_{s_m}$ and $M^-_{s_m}$ denote the extremal Pucci type operators for the class $\mathcal L_*(s_m)$ of order $2s_m$.
Then,  $M^+_{s_m} \rightarrow M^+_{s}$ and $M^-_{s_m} \rightarrow M^-_{s}$ weakly with respect to the weight $\omega_{s_0}(y) = (1+|y|)^{-n-2s_0}$.
\end{lem}

\begin{proof}
It follows straightforward from the definition of weak convergence of nonlocal elliptic operators in \cite{CS2}.
\end{proof}

The second one reads as follows.

\begin{lem}\label{lem-Calpha-bdry}
Let $s_0\in(0,1)$ and $s\in(s_0,1)$.
There exists $\delta>0$ such that the following statement holds.

If $M^+w\geq-C_0$ and $M^-w\leq C_0$ in $B_1^+$, $w=0$ in $B_1^-$, and $\|w\|_{L^1(\R^n,\,\omega_s)}\leq C_0$, then
\[\|w\|_{C^\delta(B_{3/4}^+)}\leq CC_0,\]
where $C$ and $\delta$ depend only on $n$, $s_0$, and ellipticity constants.
\end{lem}

\begin{proof}
First, using the barrier given by Lemma \ref{supersol}, we find that $|w(x)|\leq CC_0(x_n)^s_+$ in $B_{7/8}^+$.
Then, the result follows by using the interior estimates in \cite{CS}; see also \cite[Section 3]{CS2}.
\end{proof}

We next give the

\begin{proof}[Proof of Proposition \ref{prop_contr_1+s+alpha}]
The proof is by contradiction.
If the conclusion of the proposition is false, then there are sequences $u_k$, $\I_k$, $s_k$, and $\gamma_k$ satisfying
\begin{itemize}
\item $\I_k (u_k,x) = 0$ in $B_1^+$ and $u_k=0$ in $B_1^-$;
\item $\I_k(u_k,x) = \inf_{b\in \mathcal B_k}\sup_{a\in \mathcal A_k} \bigl( L_{ab} u_k(x) + c_{ab}(x)\bigr)$ with $\|c_{ab}\|_{ \gamma_k ;B_1^+} \le 1$ for all $a\in \mathcal A_k$  and $b\in \mathcal B_k$;
\item $\{L_{ab} \,: \, a\in \mathcal A_k, \ b\in\mathcal B_k \}\subset \mathcal L(s_k)$ with $s_k \in [s_0,1]$;
\item $\|u/(x_n)^s\|_{0;B_{1}^+}  +\sup_{R\ge 1} R^{-\alpha}[u/(x_n)^s]_{\beta;B_{R}^+} \le 1$;
\item $\gamma_k\geq\min\{0,\alpha+\beta-s_k\}$
\end{itemize}
for which
 \begin{equation}\label{supsup}
\sup_k \sup_{r>0} r^{-\alpha}\ [ u_k/(x_n)^{s_k} -P_{k,r}]_{\beta;B_r^+} = +\infty ,
\end{equation}
where
\[ P_{k,r} := {\rm arg \, min}_{P\in \mathcal P} \int_{B_r^+} \bigl( u_k(x)/(x_n)^{s_k} -P\bigr) \,dx\]
(recall that $\mathcal P$ denotes the real polynomials of degree at most $\lfloor \alpha+\beta\rfloor$).

To prove that this is impossible, let us start defining
\[
 \theta(r) := \sup_k  \sup_{r'>r}  (r')^{-\alpha}\,\bigl[u_k/(x_n)^{s_k} -P_{k,r}\bigr]_{\beta;B_{r'}^+} \, ,
\]
The function $\theta$ is monotone nonincreasing and we have $\theta(r)<+\infty$ for $r>0$ since we are assuming that
\begin{equation}\label{le1}
\sup_{R\ge 1} R^{-\alpha}[u/(x_n)^s]_{\beta;B_{R}^+} \le 1
\end{equation}
In addition, by \eqref{supsup} we have  $\theta(r)\nearrow +\infty$ as $r\searrow0$.
For every positive integer $m$, by definition of $\theta(1/m)$ there are $r'_m\ge 1/m$ and $k_m$ for which
\begin{equation}\label{nondeg2}
(r'_m)^{-\alpha}  \bigl[u_k/(x_n)^s -P_{k_m,r'_m}\bigr]_{\beta; B_{r'_m}} \ge \frac{1}{2} \theta(1/m) \ge \frac{1}{2}  \theta(r'_m).
\end{equation}
Here we have used that $\theta$ is non-increasing.
Note that we will have $r'_m\searrow0$ since $\theta(1/m)\nearrow +\infty$ and \eqref{le1} holds.

From now on in this proof we will use the notations
\[u_m= u_{k_m}, \ P_m = P_{k_m, r'_m},\ s_m= s_{k_m},\  {\rm and }\  \gamma_m= \gamma_{k_m}.\]

Let us consider the blow up sequence
\begin{equation}\label{eqvm}
 v_m(x) = \frac{u_{m}(r'_m x)/(r'_m x_n)^s-P_{m} (r'_m x)}{(r'_m)^{\alpha+\beta}\theta(r'_m)}.
 \end{equation}
For all $m\ge 1$ we have
\begin{equation}\label{2}
 \int_{B_1^+} v_m(x)  P(x) \,dx =0\quad \mbox{for all } P\in \mathcal P,
\end{equation}
since this is the optimality condition in the least squares minimization.

Note also that \eqref{nondeg2} implies the following inequality for all $m\geq1$:
\begin{equation}\label{nondeg35}
[v_m]_{\beta;B_1^+}\ge 1/2,
\end{equation}

Let us show now that
\begin{equation}\label{growthc1}
[v_{m}]_{\beta;B_R^+} \le  CR^{\alpha}
\end{equation}
for all $R\ge 1$.

We will do the proof of \eqref{growthc1} in the most difficult case $\alpha+\beta>1$, the case $\alpha+\beta\in (0,1)$ is very similar.
To prove \eqref{growthc1} we need to estimate the difference in the coefficients of $P_{k_m, Rr'_m} -P_{k_m, r'_m}$.
Let us denote
\[P_{k,r}(x) = p_{k,r}\cdot x + b_{k,r} \quad \mbox{where }p_{k,r}\in \R^n  \mbox{ and }b_{k,r} \in  \R.\]
Note that since we are doing the case $\alpha+\beta>1$ we will have  $\lfloor \alpha+\beta\rfloor=1$ and hence $\mathcal P$ contains  all affine functions.
Let $R=2^k$ and let us show that
\begin{equation} \label{esta's}
\bigl|p_{k_m, Rr'_m} -p_m\bigr| = \bigl|p_{k_m, 2^kr'_m} -p_{k_m, r'_m}\bigr|\le C\theta(r'_m) (Rr'_m)^{\alpha+\beta-1}.
\end{equation}

Indeed, we use by definition of $\theta(2r)$ and $\theta(r)$ we have
\[
\begin{split}
\frac{|p_{k,2r}-p_{k,r}| r^{1-\beta}}{r^\alpha \theta(r)} &\le
\frac{ \bigl[ (p_{k,2r} -p_{k,r})\cdot x\bigr]_{\beta; B_r^+}}{r^\alpha \theta(r)} = \frac{ \bigl[ P_{k,2r} -P_{k,r}\bigr]_{\beta; B_r^+}}{r^\alpha \theta(r)}
 \\
&\le   \frac{2^\alpha\theta(2r)}{\theta(r)} \frac{ \bigl[ u_k/(x_n)^{s_k} -P_{k,2r}\bigr]_{\beta; B_{2r}^+} }{(2r)^\alpha\theta(2r)} + \frac{ \bigl[ u_k/(x_n)^{s_k } -P_{k,r}\bigr]_{\beta; B_r^+} }{r^\alpha\theta(r) }
 \\
 &\le 2^\alpha+1\le 3,
 \end{split}
 \]
where we have used the definition of $\theta$ and its monotonicity.

Thus,
\[
\bigl|p_{k_m, 2^kr'_m} -p_{k_m, r'_m}\bigr| \le 3 \sum_{j=0}^{k-1} \theta(2^jr'_m) (2^{j}r'_m)^{\alpha+\beta-1} \le C \theta(r'_m) (2^k r'_m)^{\alpha+\beta-1}.
\]
Note that it is here where we use that $\alpha+\beta-1>0$.

Therefore, we can estimate as follows
\[
\begin{split}
[v_{m}]_{\beta;B_R^+} &= \frac{1}{\theta({r'_m})(r'_m)^{\alpha}} \bigl[u_{m}/(x_n)^{s_m} - P_m  \bigr]_{\beta;B_{Rr'_m}^+}
\\
&= \frac{R^{\alpha}}{\theta({r'_m}) (Rr'_m )^\alpha } \biggl( \bigl[u_{k_m}/(x_n)^{s_m} - P_{k_m, Rr'_m } \bigr]_{\beta;B_{Rr'_m}^+}  +  \bigl[ P_{k_m, Rr'_m }  - P_{m } \bigr]_{\beta;B_{Rr'_m}^+}  \biggr)
\\
&
\le R^\alpha  \frac{\theta(Rr'_m)}{\theta(r'_m)} + 2|p_{k_m, Rr'_m} -p_m| (Rr'_m)^{1-\beta}
\\
&\le  R^{\alpha} + CR^\alpha,
\end{split}
\]
where we have used \eqref{esta's}.
This proves \eqref{growthc1} in the case $\alpha+\beta>1$. As said above, the proof of \eqref{growthc1} in the case $\alpha+\beta\in(0,1)$ its easier since in this case $\bigl[ P_{k_m, Rr'_m }  - P_{m } \bigr]_{\beta;B_{Rr'_m}^+}  =0$
and we do not need to estimate the difference of the coefficients.

When $R=1$, \eqref{growthc1} implies that $\|v_m- b \|_{L^\infty(B_1)}\le C$, for some $b\in \R$.
Thus, \eqref{2} implies that
\begin{equation}\label{boundedinB1}
\|v_m\|_{L^\infty(B_1^+)}\le C.
\end{equation}
Then, using \eqref{growthc1}  and \eqref{boundedinB1} we easily obtain that
\begin{equation} \label{growthc0}
\|v_{m}\|_{L^\infty(B_R^+)} \le CR^{\alpha+\beta}.
\end{equation}
Note that in the case $\beta=0$ the difference between \eqref{growthc1} and \eqref{growthc0} is that we pass from a oscillation bound to an $L^\infty$ bound.

Next we prove the following

\vspace{5pt}
\noindent {\em Claim. Let $w_m(x) = v_m(x) (x_n)_+^s $. Then, up to subsequences we have $s_m \to s\in [s_0,1]$ and $w_m\to w$  in $C^\beta_{\rm loc}\bigl(\R^n\bigr)$, where $w\in C^{\beta'}\bigl(\R^n\bigr)$.
Moreover, the limiting function $w$  satisfies the assumptions of the Liouville-type Theorem \ref{thm-liouv-1+s}.}
\vspace{5pt}

In the case $\beta>0$, it follows from \eqref{growthc1} and \eqref{recursive_estimate} (rescaled) that 
\[ [v_m]_{\beta'; B_R} \le C R^{\alpha+\beta-\beta'}.\]
Thus, recalling that  $\beta'>0$ and using  \eqref{growthc0}, by Arzel\`a-Ascoli Theorem $v_m$ converges (up to a subsequence) in $C^{\beta}_{\rm loc}\bigl(\{x_n\ge 0\}\bigr)$ to some $v\in  C^{\beta'}\bigl(\{x_n\ge 0\}\bigr)$.
The convergence of $w_m$ to $w=v(x_n)_+^s$ is then immediate.

In the case $\beta=0$, the functions $w_m$ satisfy $M^+w_m\leq C(K)$ and $M^-w_m\geq-C(K)$ in every half-ball $B_K^+$, and satisfy $\|w_m\|_{L^\infty(B_R)}\leq CR^{s_m+\alpha}$ for every $R\geq1$.
Hence, we have $\|w_m\|_{C^\delta(B_K)}\leq C(K)$ by Lemma \ref{lem-Calpha-bdry}.
Thus, the functions $w_m$ converge to some $w$ uniformly in compact sets.

Moreover, by passing to the limit \eqref{growthc1} we find that the assumption \eqref{growthcontrol-1+s} of Theorem \ref{thm-liouv-1+s} is satisfied by $w$.

Now, each $u_k$ satisfies
\[ \I_k(u_k,x) := \inf_{b\in \mathcal B_k}\sup_{a\in \mathcal A_k} \bigl( L_{ab} u_k(x) + c_{ab}(x)\bigr)=0\quad  {\rm in}\  B_1^+.\]
Thus,   for  every $\bar h \in B_{1/2}^+$   we have
\[  \inf_{b\in  \mathcal B_{k_m}}\sup_{a\in \mathcal A_{k_m}}  \bigl( L_{ab} u_{m}(\bar x +\bar h) + c_{ab}(\bar x +\bar h)\bigr)=0 \quad \mbox{for }\bar x \in B_{1/2}^+\]
in the viscosity sense.

Using that $[c_{ab}]_{C^{\gamma_m}}\le 1$ (for all $a\in \mathcal A_{k_m}$ and $b\in \mathcal B_{k_m}$), it follows that
\[  \inf_{b\in  \mathcal B_{k_m}}\sup_{a\in \mathcal A_{k_m}}  \bigl( L_{ab} u_{m}(\bar x +\bar h) + c_{ab}(0)\bigr) \ge -|\bar x+\bar h|^{\gamma_m}  \quad \mbox{for }\bar x \in B_{1/2}^+ \]
and
\[  \inf_{b\in  \mathcal B_{k_m}}\sup_{a\in \mathcal A_{k_m}}  \bigl( L_{ab} u_{m}(\bar x) + c_{ab}(0)\bigr) \le  |\bar x|^{\gamma_m}  \quad \mbox{for }\bar x \in B_{1/2}^+ \]
in the viscosity sense.

Therefore, using Lemma 5.8 in \cite{CS} we obtain that
\begin{equation}\label{uuuu}
 M^+_{s_m} \bigl( u_{m}(\cdot +\bar h)- u_m\bigr) \ge - |\bar x|^{\gamma_m} -|\bar x+ \bar h|^{\gamma_m} \quad \mbox{in }B_{1/2}^+
 \end{equation}
in the viscosity sense.

Next, by Lemma \ref{lem-planes}, for every affine or constant function $P \in \mathcal P$, the function $\varphi(x) = P (x_n)_+^s$ satisfies
\[  M^+ \bigl( \phi(\cdot + \bar h)- \phi\bigr)  = M^-   \bigl( \phi(\cdot + \bar h)- \phi\bigr)  = 0 \quad \mbox{in }\R^n_+ \]
pointwise an in the classical sense for every $\bar h$ with $\bar h_n \ge 0$.
Using this property, the value of the operator does not change when adding  to test functions multiples of  $ \phi(\cdot + \bar h)- \phi$. Hence,  recalling that
\[  w_m(x)= v_m(x)(x_n)_+^{s_m} = \frac{v_m(x)(r'_mx_n)_+^{s_m}}{(r'_m)^{s_m}}  = \frac{u_m(r'_mx) -P_{m}(r'_mx_n)_+^{s_m}}{ (r'_m)^{\alpha+\beta+s_m}\theta(r'_m)} \]
and the definition of $v_m$ from \eqref{eqvm}, we can translate \eqref{uuuu} from $u_m$ to $w_m$.
Indeed, setting $\bar h = r'_m h$ and $\bar x =r'_m x$, we obtain
\[
-(r'_m)^{\gamma_m} 3K^{\gamma_m} \le  \frac{\theta (r'_m) (r'_m )^{\alpha+\beta+s_m}}{(r'_m)^{2s_m}}  M^+   \bigl( w_m ( \cdot + h) -  w_m  \bigr) \quad \mbox{in }B_K^+.
\]
whenever $h_n\ge0$, $|h|<K$, and $r'_m<\frac{1}{2K}$.

Therefore
\begin{equation}\label{wwwwww}
-3K^{\gamma_m}\frac{ (r'_m)^{s_m+\gamma_m}}{\theta (r'_m) (r'_m )^{\alpha+\beta}}  \le M^+   \bigl( w_m ( \cdot + h) -  w_m \bigr) \quad \mbox{ in }  B_K^+
\end{equation}
in the viscosity sense for all $h\in B_K$ whenever $r'_m<\frac{1}{2K}$.

Since $w_m \to w$ locally uniformly in $\{x_n\ge 0\}$ (up to subsequences), then we have
\begin{equation}\label{locallyuniformly}
\bigl(w_m(\cdot+h)-w_m\bigr) \rightarrow \bigr(w(\cdot+h)-w\bigr)  \quad \mbox{locally uniformly in  }  \R^n.
\end{equation}

Let us check that, for some $\epsilon>0$ small enough, we have that for every $h\in \R^n$ with $h_n\ge0$
\begin{equation}\label{L1weightedconv}
\bigl(w_m(\cdot+h)-w_m\bigr) \rightarrow \bigr(w(\cdot+h)-w\bigr)  \quad \mbox{in } L^1\bigr(\R^n, \omega_{s-\epsilon} \big).
\end{equation}
Recall that in all the paper  we denote $\omega_{s}(y) = (1-s)(1+|y|)^{-n-2{s}}$.

To show \eqref{L1weightedconv}, observe that
\[ w_m(x+h)-w_m(x) = \bigl(v_m(x+h) - v_m(x)\bigr) (x_n+h_n)_+^{s_m} +  v_m(x)\bigl((x_n+h_n)_+^{s_m} - (x_n)_+^{s_m}\bigr)\]
and hence, using \eqref{growthc1} and \eqref{growthc0},
\[\begin{split}
&\bigl| w_m(x+h)-w_m(x)\bigr|
\,\le
\\
&\hspace{5mm}\le
\begin{cases}
C |h|^\beta(1+|x|)^{\alpha} (x_n+h_n)^{s_m} + C(1+|x|)^{\alpha+\beta}h_n (x_n)^{s_m-1} \quad& \mbox{if }x_n> 0,
\\
C ( 1+|x|)^{\alpha+\beta} (h_n)^{s_m}  & \mbox{if }-h_n < x_n< 0,
\\
0 & \mbox{if }x_n <-h_n.
\end{cases}
\end{split}
\]
Therefore, we have
\[\begin{split}
\bigl| w_m(x+h)-&w_m(x)\bigr| \le g(x) \\
&:=  C\bigl(1 +  (|x|^\alpha(x_n)^{s_m} +|x|^\alpha(x_n)^{s_m-1}) \chi_{(0,+\infty)} (x_n)+ |x|^{\alpha+\beta}\chi_{(-C,0)}(x_n) \bigr)
\end{split}
\]
where $C$ (and $g$) depend on $h$.
Since $s_m\to s$ we will have $s_m\ge s-\epsilon$ for $m$ large enough, and using that $\beta\le1$ and $\bar\alpha<s_0$ we readily show that $g\in  L^1\bigr(\R^n, \omega_{s-\epsilon} \big)$.
Therefore, \eqref{L1weightedconv} follows from \eqref{locallyuniformly} using the dominated convergence theorem.

Finally,  using  \eqref{L1weightedconv} and \eqref{locallyuniformly} it follows from Lemma \ref{lemweak2} and Lemma 5 in \cite{CS2} we can pass to the limit in \eqref{wwwwww} in every ball $B_K^+$ to obtain that
\[0\le M^+\bigl\{ w(\cdot + h)- w\bigr\} \quad \mbox{in } B_K^+ .\]
Thus, since this can be done for any $K>0$, we have
\[0\le M^+\bigl\{ w(\cdot + h)- w\bigr\} \quad \mbox{in } \R^n_+ .\]

Analogously, we will have that
\[0\ge M^-\bigl\{ w(\cdot + h)- w\bigr\} \quad \mbox{in } \ \R^n_+,\]
and this finishes the proof the Claim.

\vspace{5pt}

We have thus proved that $w$ satisfies all the assumptions of Theorem \ref{thm-liouv-1+s} and hence we conclude that $v = w/(x_n)^s$ is an affine function.
On the other hand, passing \eqref{2} to the limit we obtain that $v$ is orthogonal to every affine function and hence it must be $v\equiv 0$.
But then passing \eqref{nondeg35} to the limit we obtain that $v$ cannot be constantly zero in $B_1$; a contradiction.
\end{proof}

To prove Theorem \ref{thm2} we will need the following Lemma, that matches Proposition \ref{prop_contr_1+s+alpha}.

\begin{lem}\label{lemcontrarec}
Let $\alpha\in(0,1]$ and $\beta\in[0,1]$ with $\alpha+\beta \neq 1$ and let and $v$ satisfy, for all $r>0$
\[
\sup_{r>0} r^{-\alpha}\ [v-P_r]_{\beta;B_r^+} \le C_0,
\]
where $P_r$  some polynomial of degree at most $\lfloor \alpha+\beta\rfloor$ (zero or one) depending on $r$.
In the case $\alpha+\beta>1$, assume in addition that $P_1(x) = p_1\cdot x + b_1$, with $|p_1|\le C_0$.
Then, the limit $P=\lim_{r\searrow 0} P_r$ exist and for all $r>0$ we have
\[ \|v-P\|_{L^\infty(B_r^+)} \le C C_0 r^{\beta+\alpha}, \quad \mbox{and}\quad |p| \le CC_0,\]
where $P(x)=p\cdot x+b$, and $p=0$ if $\alpha+\beta<1$.
The constant $C$ depends only on $\alpha$ and $\beta$.
\end{lem}

\begin{proof}
We will do the most difficult case $\alpha+\beta>1$.

Let $P_r(x)=p_r\cdot x+b_r$.
We have, for all $r>0$,
\[  [ v-p_{r}\cdot x]_{\beta;B_r^+} \le C_0 r^\alpha.\]
Thus,
\[
\begin{split}
|p_r-p_{r/2}| (r/2)^{1-\beta} &\le \bigl[ (p_{r} -p_{r/2})\cdot x\bigr]_{\beta;B_{r/2}^+}
 \\
 &\le [ v-p_{r}\cdot x]_{\beta;B_{r/2}^+} + [ v-p_{r}\cdot x]_{C^\beta;B_{r/2}^+}
 \\
 &\le C_0 r^{\alpha} + C_0 (r/2)^\alpha
\end{split}
\]
and hence
\[  |p_r-p_{r/2}| \le C C_0 r^{\alpha+\beta-1}.\]

It follows (developing the expressions as  telescopic sums and summing the geometric series) that $p= \lim_{r\searrow 0} p_r$ exists and
\[  |p_r-p|\le C C_0 r^{\alpha+\beta-1}.\]
In particular
\[ |p|\le |p_1| +|p_1-p| \le C_0 + C C_0.\]
Then we  obtain that
\[ \begin{split}
{\rm osc}_{B_r^+}[ v -p\cdot x] &\le  {\rm osc}_{B_r^+}[ v-p_r\cdot x] +  {\rm osc}_{B_r^+}[(p_r-p) \cdot x]   \\
  &\le   [ v-p_r\cdot x]_{\beta;B_r^+} r^{\beta} +  |p_r-p| r
\\
  &\le \ C C_0 r^{\beta+\alpha}
\end{split}\]
and the lemma now follows.
\end{proof}

We give now a second step towards Theorem \ref{thm2}.
This result follows from the interior estimates.

\begin{lem}\label{lem-interior}
Let $s_0\in(0,1)$, and let $\bar \alpha>0$ be the constant given by Proposition \ref{thm1}.

Let $s\in(s_0,1)$, $\alpha\in(0,\bar\alpha)$, $\gamma\in(0,1)$, and $\beta\in[0,1)$ such that $\alpha+\beta\leq \gamma+s$.
Assume in addition that $\alpha+\beta\neq 1$.

Let $e_n = (0,\dots,0,1)$ and $w$ be a solution of $\I(w,x)=0$ in $B_1(e_n)$, where $\I$ is any fully nonlinear operator elliptic with respect to $\mathcal L_*(s)$ of the form
\[\I(w,x) = \inf_{b\in \mathcal B}\sup_{a\in\mathcal A} \bigl( L_{ab} w(x) + c_{ab}(x)\bigr),\]
with $L_{ab}$ given by \eqref{L*1}-\eqref{L*2},
\[\|c_{ab}\|_{\gamma;B_1(e_n)}\leq 1\quad\textrm{and}\quad \|\mu_{ab}\|_{C^\gamma(S^{n-1})}\leq \Lambda.\]

Assume that
\[ \|w\|_{L^\infty(\R^n)} <\infty, \quad \|w\|_{L^{\infty}(B_r(0))} \le C_0 r^{\alpha+\beta+s},\quad \mbox{and}\quad \|w\|_{\beta; B_r(2re_n)} \le C_0 r^{\alpha+s}\]
for all $r>0$.

Then,
\[ [w]_{C^{\alpha+\beta} (B_{r/2} (2re_n))} \le C_0 r^s, \]
for all $r\in(0,1)$, for some constant $C$ that depends only on $n$, $s_0$, ellipticity constants, $\alpha$, $\beta$.
\end{lem}

\begin{proof}
We differ the proof to the Appendix.
\end{proof}

We next show the following result.
It follows from Proposition \ref{prop_contr_1+s+alpha} and Lemma \ref{lemcontrarec} by truncating the solution $u$.
We will also use the interior estimates from the previous Lemma \ref{lem-interior}.

\begin{prop}\label{utterable}
Let $s_0\in(0,1)$, and let $\bar \alpha>0$ be the constant given by Proposition \ref{thm1}.

Let $s\in(s_0,1)$, $\alpha\in(0,\bar\alpha)$, $\gamma\in(0,1)$, and $\beta\in[0,1]$ such that $\alpha+\beta\leq \gamma+s$.
Assume in addition that $\alpha+\beta\neq 1$.

Let $u$ be any solution of $\I(u,x)=0$ in $B_1^+$ and $u=0$ in $\R^n_-$, where $\I$ is any fully nonlinear operator elliptic with respect to $\mathcal L_*(s)$ of the form
\begin{equation}\label{op7}
\I(u,x) = \inf_{b\in \mathcal B}\sup_{a\in\mathcal A} \bigl( L_{ab} u(x) + c_{ab}(x)\bigr)
\end{equation}
where $\|\mu_{ab}\|_{C^\gamma(S^{n-1})}\leq 1$ and $\|c_{ab}\|_{\gamma;B_1^+} \le 1$ for all $a\in \mathcal A$ and $b\in \mathcal B$.

If $\beta>0$ assume in addition that the following estimate holds for some $\beta'\in(\beta,\beta+\alpha)$ and for all $u$ and $\I$ as above:
\begin{equation}
\label{recursive_estimate2}
 [u/(x_n)^s]_{\beta';B_{1/4}^+}\le C \bigl(\|u/(x_n)^s\|_{0;B_{1}^+}  +\sup_{R\ge 1} R^{-\alpha}[u/(x_n)^s]_{\beta;B_{R}^+}\bigr),
\end{equation}
where $C$  depends only on $n$, $s_0$, ellipticity constants, $\alpha$,  $\beta$, and $\beta'$.

Then \eqref{recursive_estimate2} hold also for $\beta' = \beta+\alpha$ and moreover we have the following local estimate
\[ \|u/(x_n)^{s} \|_{C^{\beta+\alpha}(B_{1/4}^+)} \le C \bigl(\|u/(x_n)^s\|_{\beta;B_{3/4}^+}+\|u\|_{L^\infty(\R^n)} \bigr) ,\]
for some constant $C$ that depends only on $n$, $s_0$, ellipticity constants, $\alpha$, $\beta$.
\end{prop}

\begin{proof}
To prove the proposition we assume that either
\[\|u/(x_n)^s\|_{0;B_{1}^+}  +\sup_{R\ge 1} R^{-\alpha}[u/(x_n)^s]_{\beta;B_{R}^+} \le 1\]
or
\[\|u/(x_n)^s\|_{\beta;B_{3/4}^+}+\|u\|_{L^\infty(\R^n)}  \le 1\]
and we will show that
\[ \|u/(x_n)^{s} \|_{C^{\beta+\alpha}(B_{1/4}^+)} \le C.\]

Let us consider $\bar u=u\eta$, where $\eta\in C^\infty_c(B_{3/4})$ satisfies $\eta\equiv1$ in $B_{5/8}$.
Then, using that the kernels of the operators are $C^\gamma$ outside the origin, we find that $L_{ab}\bar u=L_{ab}u+\bar c_{ab}(x)$ in $B_{1/2}^+$, where $\bar c_{ab}$ satisfy
\[\|\bar c_{ab}\|_{\gamma;B_{1/2}^+}\leq C_0.\]
Hence, we have that
\[\bar \I(\bar u,x)= \inf_{b\in \mathcal B}\sup_{a\in\mathcal A} \bigl( L_{ab} \bar u(x) - \bar c_{ab}(x)+c_{ab}(x)\bigr)=0\quad \textrm{in}\ B_{1/2}^+.\]
Moreover, we have
\[\|\bar u/(x_n)^s\|_{\beta;\{x_n\ge 0\}}\le C_0.\]
Hence, by Proposition \ref{prop_contr_1+s+alpha} and Lemma \ref{lemcontrarec}, we find the following.
For each $z\in B_{1/4}\cap\{x_n=0\}$ there exist $p(z)\in \R^n$ and $b(z)\in\R$ such that
\begin{equation}\label{bound-alpha+beta}
\|u/(x_n)^s-p(z)\cdot x-b(z)\|_{L^\infty(B_r^+)} \le C C_0 r^{\beta+\alpha}, \quad \mbox{for all}\quad r<1/4,
\end{equation}
and
\begin{equation}\label{bound-alpha}
\|u/(x_n)^s-p(z)\cdot x-b(z)\|_{\beta;\,B_r^+} \le C C_0 r^{\alpha}, \quad \mbox{for all}\quad r<1/4,
\end{equation}
with
\[|p(z)| \le CC_0,\quad |b(z)|\leq CC_0.\]
We have used that $\bar u=u$ in $B_{1/2}$.

Let us see next that \eqref{bound-alpha+beta}-\eqref{bound-alpha} imply
\[\|u/(x_n)^s\|_{C^{\alpha+\beta}(B_{1/4}^+)}\leq CC_0.\]
For it, we define
\[Q_z(x):=(p(z)\cdot x+b(z))(x_n)^s_+\chi_{B_2}(x),\]
and observe that \eqref{bound-alpha+beta} and \eqref{bound-alpha} give
\[\|u-Q_z\|_{L^\infty(B_R(0))}\leq CR^{s+\alpha+\beta}\]
and
\[\|u-Q_z\|_{\beta; B_R(x_0)}\leq CR^{s+\alpha}\]
for every ball $B_R(x_0)$ such that $2R=\textrm{dist}(x_0,\{x_n=0\})$ and $B_{2R}(x_0)\subset B_{1/4}^+$, where $z$ is the projection of $x_0$ on $\{x_n=0\}$.

Using the previous two inequalities, Lemma \ref{lem-interior} yields
\begin{equation}\label{austin2}
\|u-Q_z\|_{C^{\beta+\alpha}(B_R(x_0))}\leq CR^s
\end{equation}
in every such ball $B_R(x_0)$.

Finally, using $\|(x_n)^{-s}\|_{L^\infty(B_R(x_0))}\leq CR^{-s}$ and $\|(x_n)^{-s}\|_{C^{\alpha+\beta}(B_R(x_0))}\leq CR^{-s-\alpha-\beta}$, we find
\[\|u/(x_n)^s\|_{C^{\alpha+\beta}(B_R(x_0))}\leq C.\]
Since this can be done for all balls $B_R(x_0)$ with $2R=\textrm{dist}(x_0,\{x_n=0\})$ and $B_{2R}(x_0)\subset B_{1/4}^+$, we have $\|u/(x_n)^s\|_{C^{\alpha+\beta}(B_{1/4}^+)}\leq C$, and thus the proposition is proved.
\end{proof}

Finally, we give the

\begin{proof}[Proof of Theorem \ref{thm2}]
We will show the result by applying inductively Proposition~\ref{utterable}.

First, using the supersolution in Lemma \ref{supersol}, we find that
\[\|u/(x_n)^s\|_{L^\infty(B_{3/4}^+)}\leq CC_0.\]
Hence, using Proposition \ref{utterable} with $\beta=0$, we find that
\[\|u/(x_n)^s\|_{C^\alpha(B_{1/4}^+)}\leq CC_0.\]
We also have that the estimate \label{recursive_estimate2} holds for with $\beta'$  replace by $\alpha$ whenever $\I$ and $u$ satisfy the assumptions of Proposition \eqref{utterable}.

Since this is for any solution $u$, then by a standard covering argument, we will have the estimate
\[\|u/(x_n)^s\|_{C^\alpha(B_{3/4}^+)}\leq CC_0.\]
Using this and  Proposition \ref{utterable} with $\beta = \alpha-\epsilon$, for some $\epsilon>0$, and with $\beta' = \alpha$, we find that
\[\|u/(x_n)^s\|_{C^{2\cdot\alpha-\epsilon}(B_{1/4}^+)}\leq CC_0,\]
and that the estimate \label{recursive_estimate2} holds with  $\beta'$ replaced by $2\alpha-\epsilon$
Iterating this procedure, we find that $u/(x_n)^s\in C^{k(\alpha-\epsilon)}(B_{1/4})$ whenever $k\cdot\alpha\leq s+\gamma$.

More precisely, after a finite number of steps we find that, for any solution $u$, we have the estimate
\[\|u/(x_n)^s\|_{C^{s+\gamma}(B_{1/4}^+)}\leq CC_0.\]
Thus, the Theorem is proved.
\end{proof}

\section{Boundary regularity: curved boundary}
\label{sec-regularity-curved}

In this section we prove Theorem \ref{thm2-curved}.
For it, we will follow the same steps as in the previous Section.

The main ingredient towards Theorem \ref{thm2-curved} will be Proposition \ref{prop_contr_1+s+alpha} below.
Before stating it, we need the following.

\begin{defi}\label{defiGamma}
We say that $\Gamma$ is a global $C^{2,\gamma}$ surface given by a graph of $C^{2,\gamma}$ norm smaller than one, splitting $\R^n$ into $\Omega^+$ and $\Omega^-$, if the following happens.
\begin{itemize}
\item The surface $\Gamma\subset\R^n$ is the graph of a global $C^{2,\gamma}$ and bounded function, whose $C^{2,\gamma}$ norm is smaller than one.
\item The two disjoint domains $\Omega^+$ and $\Omega^-$ partition $\R^n$, i.e., $\R^n=\overline {\Omega^+}\cup\overline {\Omega^-}$.
\item We have $\Gamma=\partial \Omega^+=\partial \Omega^-$, and $0\in \Gamma$.
\item The origin $0$ belongs to $\Gamma$ and the normal vector to $\Gamma$ at $0$ is $\nu(0)=e_n$.
\end{itemize}
Moreover, we let $d(x)$ be any $C^{2,\gamma}(\overline{\Omega^+})$ function that coincides with $\textrm{dist}(x,\Omega^-)$ in a neighborhood of $\Gamma\cap B_4$ and $d\equiv0$ outside $B_5$.
\end{defi}

The Proposition reads as follows.

\begin{prop}\label{prop_contr_1+s+alpha8}
Let $s_0\in(0,1)$, and let $\bar \alpha>0$ be the constant given by Proposition \ref{thm1}.

Let $s\in(s_0,1)$, $\alpha\in(0,\bar\alpha)$, $\gamma\in(0,s]$, and $\beta\in[0,1)$ such that $\alpha+\beta\leq \gamma+s$.
Assume in addition that $\alpha+\beta\neq 1$.

Assume that $\Gamma$ is a global $C^{2,\gamma}$ surface given by a graph of $C^{2,\gamma}$ norm smaller than one, splitting $\R^n$ into $\Omega^+$ and $\Omega^-$; see Definition \ref{defiGamma}.

Let $u\in C_c(B_2)$ be any solution of $\I(u,x)=0$ in $B_1\cap\Omega^+$ and $u=0$ in all of $\Omega^-$, where $\I$ is any fully nonlinear operator elliptic with respect to $\mathcal L_*(s)$ of the form
\[\I(u,x) = \inf_{b\in \mathcal B}\sup_{a\in\mathcal A} \bigl( L_{ab} u(x) + c_{ab}(x)\bigr),\]
where $\|\mu_{ab}\|_{C^\gamma(S^{n-1})}\leq \Lambda$ and $\|c_{ab}\|_{\gamma;B_1\cap \Omega^+} \le 1 \quad \mbox{for all }a\in \mathcal A \mbox{ and }b\in \mathcal B$.

If $\beta>0$ assume in addition that the following estimate holds for some $\beta'\in(\beta,\beta+\alpha)$ and for all $Gamma$, $u$, and $\I$ as above:
\begin{equation}\label{recursive_estimate8}
 [u/(x_n)^s]_{\beta';B_{1/4}\cap \Omega^+}\le C\bigl(\|u/(x_n)^s\|_{0;B_{1}\cap \Omega^+}  +\sup_{R\ge 1} R^{-\alpha}[u/(x_n)^s]_{\beta;B_{R}\cap \Omega^+}\bigr),
\end{equation}
where $C$  depends only on $n$, $s_0$, ellipticity constants, $\alpha$,  $\beta$, and $\beta'$.

Then,
\[ r^{-\alpha}[ u/d^{s} -P_{r}(\cdot)]_{\beta; B_r\cap \Omega^+} \le C\bigl(\|u/(x_n)^s\|_{0;B_{1}\cap \Omega^+}  +\sup_{R\ge 1} R^{-\alpha}[u/(x_n)^s]_{\beta;B_{R}\cap \Omega^+}\bigr),\]
for some constant $C$ that depends only on $n$, $s_0$, ellipticity constants, $\alpha$, $\beta$, and $\beta'$
where $P_r(x)$ is defined as the polynomial of degree at most $\lfloor \alpha+\beta\rfloor$ which best fits the function $u/(x_n)^{s}$ in $B_r\cap\Omega^+$.
That is,
\[P_r := {\rm arg \, min}_{ P \in \mathcal P}  \int_{B_r\cap \Omega^+} \bigl( u(x)/d^s - P(x)\bigr)^2 \,dx,\]
where $\mathcal P$ is the space of polynomials with real coefficients and degree at most $\lfloor \alpha+\beta\rfloor$.
\end{prop}

An important ingredient of this proof is the following.

\begin{lem}\label{Lds-is-Cs-C}
Let $s_0\in(0,1)$ and $s\in (s_0,1)$, and $\gamma\in(0,s]$.
Let $\Gamma$ and $d$ be as in Definition \ref{defiGamma}.
Let $L$ be any operator of the form \eqref{L*1} with $\|\mu_{ab}\|_{C^{1,\gamma}(S^{n-1})}\leq 1$.

Then, for any function $\eta\in C^{2,\gamma}(\R^n)$, we have
\[\|L(d^s\eta)\|_{C^\gamma(B_1\cap \Omega^+)}\leq C\|\eta\|_{C^{2,\gamma}},\]
where $C$ is a constant that depends only on $n$ and $s_0$.
\end{lem}

\begin{proof}
It follows easily from Proposition \ref{Lds-is-Cs-B}.
\end{proof}

\begin{rem}
The hypothesis $\gamma\leq s$ of Proposition \ref{prop_contr_1+s+alpha8} and of Theorem \ref{thm2-curved} is only needed to show Lemma \ref{Lds-is-Cs-C}.
In particular, if one can show this result for all $\gamma\in (0,1)$, then the hypothesis $\gamma\leq s$ in Theorem \ref{thm2-curved} can be removed.
\end{rem}

Let us now proceed with the proof of Proposition \ref{prop_contr_1+s+alpha8}.

We will skip some details of this proof, since it is very similar to the one of Proposition \ref{prop_contr_1+s+alpha}.

\begin{proof}[Proof of Proposition \ref{prop_contr_1+s+alpha8}]
We argue by contradiction.
If the conclusion of the proposition is false, then there are sequences $u_k$, $\I_k$, $\Gamma_k$, $\gamma_k$ and $s_k$ satisfying
\begin{itemize}
\item $\Gamma_k$ is a global $C^{2,\gamma_k}$ surface given by a graph of $C^{2,\gamma_k}$ norm smaller than one, splitting $\R^n$ into $\Omega^+_k$ and $\Omega^-_k$; see Definition \ref{defiGamma}.
\item $\I_k (u_k,x) = 0$ in $B_1\cap \Omega^+_k$ and $u_k=0$ in $\Omega^-_k$;
\item $\I_k(u_k,x) = \inf_{b\in \mathcal B_k}\sup_{a\in \mathcal A_k} \bigl( L_{ab} u_k(x) + c_{ab}(x)\bigr)$
\item $\{L_{ab} \,: \, a\in \mathcal A_k, \ b\in\mathcal B_k \}\subset \mathcal L(s_k)$ with $s_k \in [s_0,1]$;
\item $\|\mu_{ab}\|_{C^\gamma(S^{n-1})}\leq \Lambda$ and $\|c_{ab}\|_{\gamma;B_1\cap \Omega^+} \le 1$, for all $a\in \mathcal A_k$  and $b\in \mathcal B_k$;
\item $ \|u/(x_n)^s\|_{0;B_{1}\cap \Omega^+_k}  +\sup_{R\ge 1} R^{-\alpha}[u/(x_n)^s]_{\beta;B_{R}\cap \Omega^+_k} \le 1$
\item $\gamma_k+s_k \ge \alpha+\beta$;
\end{itemize}
for which
\begin{equation}\label{supsup8}
\sup_k \sup_{r>0} r^{-\alpha}\ [ u_k/d_k^{s_k} -P_{k,r}]_{\beta;B_r\cap \Omega^+_k} = +\infty ,
\end{equation}
where
\[ P_{k,r} := {\rm arg \, min}_{P\in \mathcal P} \int_{B_r\cap \Omega^+_k} \bigl( u_k(x)/d_k^{s_k} -P\bigr) \,dx.\]

To prove that this is impossible we proceed as in the Proof of Propostion \ref{prop_contr_1+s+alpha}. We define
\[
 \theta(r) := \sup_k  \sup_{r'>r}  (r')^{-\alpha}\,\bigl[u_k/d_k^{s_k} -P_{k,r}\bigr]_{\beta;B_{r'}\cap \Omega^+_k} \, ,
\]
The function $\theta$ is monotone nonincreasing, and $\theta(r)<+\infty$ for $r>0$ since
\begin{equation}\label{le18}
\sup_{R\ge 1} R^{-\alpha}[u/(x_n)^s]_{\beta;B_{R}\cap \Omega^+_k} \le 1
\end{equation}
In addition, by \eqref{supsup8} we have  $\theta(r)\nearrow +\infty$ as $r\searrow0$ and there are sequences $r'_m \searrow 0$ and $k_m$ for which
\begin{equation}\label{nondeg2}
(r'_m)^{-\alpha}  \bigl[u_k/d_k^{s_k} -P_{k_m,r'_m}\bigr]_{\beta; B_{r'_m}\cap \Omega_{k_m}^+} \ge \frac{1}{2}  \theta(r'_m).
\end{equation}

From now on in this proof we will use the notations
\[u_m= u_{k_m}, \ P_m = P_{k_m, r'_m},\ s_m= s_{k_m},\  {\rm and }\  \gamma_m= \gamma_{k_m},\]
and
\[\bar\Gamma_m=\frac{1}{r_m'}\Gamma_{k_m},\quad \bar \Omega^+_m=\frac{1}{r_m'}  \Omega^+_{k_m},\quad  \bar d_m(x)=\textrm{dist}(x, \R^n\setminus \bar\Omega^+_m) = \frac{d_{k_m}(r'_m \,\cdot\,)}{r'_m}.\]
Notice that $\bar\Gamma_m$ is a \emph{rescaled} version of $\Gamma_{k_m}$ (so that, $d_{k_m}$ does \emph{not} coincide with $\bar d_m$).

Since $r_m\rightarrow0$, then $\bar\Gamma_m$ will converge to $\{x_n=0\}$ as $m\rightarrow\infty$.
Also, $\bar\Omega^+_m$ will converge to $\R^n_+$ as $m\rightarrow\infty$.

We consider the blow up sequence
\begin{equation}\label{eqvm8}
\begin{split}
 v_m(x) &= \frac{u_{m}(r'_m x)/[d^{s_m}_{k_m}(r_m'x)]-P_{m} (r'_m x)}{(r'_m)^{\alpha+\beta}\theta(r'_m)}\\
 &=\frac{u_{m}(r'_m x)/[(r_m')^{s_m}\bar d^{s_m}_m(x)]-P_{m} (r'_m x)}{(r'_m)^{\alpha+\beta}\theta(r'_m)}.
 \end{split}
 \end{equation}
As in the proof of Proposition \ref{prop_contr_1+s+alpha}, for all $m\ge 1$ we have
\begin{equation}\label{28}
 \int_{B_1\cap \bar\Omega_m^+} v_m(x)  P(x) \,dx =0\quad \mbox{for all } P\in \mathcal P,
\end{equation}
\begin{equation}\label{nondeg358}
[v_m]_{\beta;B_1\cap \bar\Omega_m^+}\ge 1/2.
\end{equation}
and
\begin{equation}\label{growthc18}
[v_{m}]_{\beta;B_R\cap \bar\Omega_m^+} \le  CR^{\alpha}
\end{equation}
for all $R\ge 1$.

Furthermore, we also have
\begin{equation} \label{growthc08}
\|v_{m}\|_{L^\infty(B_R\cap \bar\Omega_m^+)} \le CR^{\alpha+\beta}.
\end{equation}

Next we prove the following

\vspace{5pt}
\noindent {\em Claim. Let $w_m(x) = v_m(x)\bar d_m^s $.
Then, up to subsequences we have $s_m \to s\in [s_0,1]$ and $w_m\to w$  in $C^\beta_{\rm loc}\bigl(\R^n\bigr)$, where $w\in C^{\beta'}\bigl(\R^n\bigr)$.
Moreover, the limiting function $w$ satisfies the assumptions of the Liouville-type Theorem \ref{thm-liouv-1+s}.}
\vspace{5pt}

First recall that, by definition of $\bar\Gamma_m$ and $\bar d_m$, we have that $\bar d_m^s$ converges locally uniformly to $(x_n)_+^s$.

In the case $\beta>0$, it follows from \eqref{growthc18} and \eqref{recursive_estimate8} (rescaled) that 
\[ [v_m]_{\beta'; B_R\cap \bar\Omega_m^+}  \le C R^{\alpha+\beta-\beta'}.\]
Thus, recaling $\beta'>\beta$ and \eqref{growthc08}, by Arzel\`a-Ascoli Theorem $v_m$ converges (up to a subsequence) in $C^{\beta}_{\rm loc}(\R^n_+)$ to some $v\in  C^{\beta'}
(\{x_n\ge 0\})$.
The convergence of $w_m$ to $w=v(x_n)_+^s$ is then immediate.

In the case $\beta=0$, the functions $w_m$ satisfy $M^+w_m\leq C(K)$ and $M^-w_m\geq-C(K)$ in $B_K\cap \bar\Omega_m^+$ for any $K>0$, and satisfy $w_m=0$ in $\bar\Omega_m^-$ and $\|w_m\|_{L^\infty(B_R)}\leq CR^{s_m+\alpha}$ for every $R\geq1$.
Hence, we will have $\|w_m\|_{C^\delta(B_K)}\leq C(K)$ for some $\delta>0$.
Thus, the functions $w_m$ converge to some $w$ uniformly in compact sets.

Passing to the limit \eqref{growthc18} we find that the assumption (i) of Theorem \ref{thm-liouv-1+s} is satisfied by $w$.

Let $\mathcal L_*^{\gamma_m}(s_m)$ be the class consisting of all the operators in $\mathcal L_*$ whose spectral measures have $C^{1,\gamma_m}(S^{n-1})$ norm smaller or equal than $\Lambda$, and let $M^+_{\mathcal L_*^{\gamma_m}(s_m)} $  and  $M^-_{\mathcal L_*^{\gamma_m}(s_m)} $ denote the extremal Pucci operators for this class.
Note that $M_{\mathcal L_*^{\gamma_m}(s_m)}\leq M^+$.

Let us prove that, similarly as in Proposition \ref{prop_contr_1+s+alpha}, there is a function $\delta(r)$ with $\lim_{r\searrow 0 }\delta(r)=0$ such that, for all $h\in B_K$ and $r'_m <\frac{1}{2K}$ we have
\begin{equation}\label{wwwwww8}
-C(K)\delta(r'_m) \le M^+_{\mathcal L_*^{\gamma_m}(s_m)}   \bigl( w_m ( \cdot + h) -  w_m \bigr)\quad \mbox{in }\bar\Omega_m^+\cap (\bar\Omega_m^+-h)\cap  B_K
\end{equation}

To prove \eqref{wwwwww8} we use that, by definition of $\theta$,
\[ [u_k/d_k^{s_k} - P_{k,r}]_{\beta,B_r\cap \Omega_k}\le \theta(r)r^\alpha\quad \mbox{for all }\ k\  \mbox{and }\  r>0.\]
Thus, using that $P_{k,r}$ is the best fitting polynomial in $\mathcal P$ for $u_k$ in $B_r$ we obtain that
\[ \| u_k/d_k^{s_k} -P_{k,r}\|_{L^\infty(B_r\cap \Omega_k)} \le C r^{\alpha+\beta} \theta(r),\]
where $C$ depends only on the dimension.
Hence, for all $r>0$ and $k$ we have
\begin{equation}\label{difPs}
\| P_{k,2r} -P_{k,r}\|_{L^\infty(B_r\cap \Omega_k)} \le  \| P_{k,2r} -u_k\|_{L^\infty(B_{2r}\cap \Omega_k)} + \| u_k- P_{k,r} \|_{L^\infty(B_r\cap \Omega_k)} \le C\theta(r)r^{\alpha+\beta}
\end{equation}

Let us now denote
\[  P_{k,r}(x) = p_{k,r} \cdot x + b_{k,r},\]
where $p_{k,r} \in \R^n$ is non-zero only if $\alpha+\beta>1$ and where $b_{k,r} \in \R$ .
Using \eqref{le18} and observing that $b_{k,r}= \ave_{B_r\cap \Omega^+_k} u_k \,dr $ we obtain
\begin{equation} \label{bbb}
|b_{k,r}| \le 1 \quad \mbox{for all }k \mbox{ and }r>0.
\end{equation}

On the other hand, when $\alpha+\beta>1$ using \eqref{difPs} we obtain
\[ |p_{k,2r} - p_{k_r}| \le C\theta(r) r^{\alpha+\beta-1}.\]
Therefore, for $r= 2^{-i}$ we have
\begin{equation}\label{difp}
 \frac{|p_{k, r} -p_{k,1}|}{\theta (r)}  \le  C\sum_{j=0}^{i}    \frac{\theta(2^{-j})}{\theta(r)} (1/2)^{j(\alpha+\beta-1)} .
\end{equation}
But using again \eqref{le18} we obtain that $|p_{k,1}| \le C$ and thus from \eqref{difp} and \eqref{bbb} we have that for $r \in [2^{-i} , 2^{-i+1}]$
\begin{equation} \label{eqpb}
\frac{|p_{k,r}|  + |b_{k,r}|}{ \theta(r)} \le C\sum_{j=0}^{i}    \frac{\theta(2^{-j})}{\theta(r)} (1/2)^{j(\alpha+\beta-1)}=:\psi(r)
\end{equation}
Note that $\psi(r) \le C$ for all $r\le1$ and that moreover $\psi(r) \rightarrow 0$ as $r\searrow 0$ since  $\frac{\theta(2^{-j})}{\theta(r)}  \rightarrow 0$ for every fixed $j$.

Hence, using Lemma \ref{Lds-is-Cs-C} and the assumption that $\Gamma_k$ is a global $C^{2,\gamma_k}$ surface given by a graph of $C^{2,\gamma_k}$ norm smaller than one, we obtain
\begin{equation}\label{Cgammabound}
 \left[ L \left( \frac{d_{k_m}^{s_m} P_m}{\theta(r'_m)}\right) \right]_{C^{\gamma_m}(B_1\cap \Omega_{k_m}^+)} \le C\psi(r_m') \quad \mbox{for all }L\in\mathcal L_*^{\gamma_m}(s_m),
\end{equation}
which rescaling is
\[(r'_m)^{-\gamma_m}\left[ (r'_m)^{-2s_m} L \left( \frac{d_{k_m}^{s_m}(r_m'\,\cdot\,) P_m (r_m'\,\cdot\,) }{\theta(r'_m)}\right) \right]_{C^{\gamma_m}\left( (r'_m)^{-1} (B_1\cap\Omega_{k_m}^+)\right)} \le C\psi(r_m').\]
Equivalently, since $\bar\Omega_m=\frac{1}{r'_m}\Omega_{k_m}^+$ and $\bar d_m^{s_m}=(r_m')^{-s_m}d_{k_m}^{s_m}(r_m'\,\cdot\,)$, we obtain
\begin{equation}\label{Cgammaboudrescaled2}
\left[ L \left( \frac{\bar d_m^{s_m} P_m (r_m'\,\cdot\,) }{\theta(r'_m)(r'_m)^{\alpha+\beta}}\right) \right]_{C^{\gamma_m}(B_{1/r'_m}\cap\bar\Omega_{m}^+)} \le \frac{C\psi(r_m') (r'_m)^{s_m+\gamma_m}}{(r'_m)^{\alpha+\beta}} \quad \mbox{for all }L\in\mathcal L_*^{\gamma_m}(s_m).
\end{equation}

Now, recall that $\gamma_m+s_m \ge \alpha+\beta$ (for all $m$) and that
\[w_m(x)=v_m(x)\bar d^{s_m}_m(x)=
\frac{u_{m}(r'_m x) }{(r'_m)^{\alpha+\beta+s_m}\theta(r'_m)}-\frac{P_{m}(r'_m x)\cdot \bar d^{s_m}_m(x)}{(r'_m)^{\alpha+\beta}\theta(r'_m)}.\]
Therefore, using \eqref{Cgammaboudrescaled2} and the same argument as in the proof of Proposition \ref{prop_contr_1+s+alpha}, we find
\[M^+_{\mathcal L_*^{\gamma_m}(s_m)} \bigl(w_m(\cdot+h)-w_m\bigr)\geq -\frac{3K^\gamma_m}{\theta(r_m')}-CK^{\gamma_m}\psi(r_m')\quad \textrm{in}\ \bar\Omega_m^+\cap (\bar\Omega_m^+-h)\cap  B_K\]
for all $h\in B_K$.
Since $\theta(r_m')\rightarrow\infty$ and $\psi(r_m')\rightarrow0$ as $r_m'\rightarrow0$, \eqref{wwwwww8} follows.

On the other hand, since $w_m \to w$ locally uniformly in $\R^n$ (up to subsequences), then we have
\begin{equation}\label{locallyuniformly8}
\bigl(w_m(\cdot+h)-w_m\bigr) \rightarrow \bigr(w(\cdot+h)-w\bigr)  \quad \mbox{locally uniformly in  }  \R^n.
\end{equation}
Also, similarly as in Proposition \ref{prop_contr_1+s+alpha}, we have that for every $h\in \R^n$ with $h_n\ge0$
\begin{equation}\label{L1weightedconv8}
\bigl(w_m(\cdot+h)-w_m\bigr) \rightarrow \bigr(w(\cdot+h)-w\bigr)  \quad \mbox{in } L^1\bigr(\R^n, \omega_{s-\epsilon} \big)
\end{equation}
for some $\epsilon>0$.

Thus, using \eqref{L1weightedconv8} and \eqref{locallyuniformly8} we can pass to the limit in \eqref{wwwwww8} to obtain that, for every $K\geq1$ and for every $h\in B_K^+$,
\[0\le M^+_{\mathcal L_*^{\gamma}(s)}\bigl\{ w(\cdot + h)- w\bigr\} \quad \mbox{in } B_K^+ .\]
This yields
\[0\le M^+_{\mathcal L_*^{\gamma}(s)}\bigl\{ w(\cdot + h)- w\bigr\} \quad \mbox{in } \R^n_+\]
whenever $h_n\geq0$.

Analogously, we will have that
\[0\ge M^-_{\mathcal L_*^{\gamma}(s)}\bigl\{ w(\cdot + h)- w\bigr\} \quad \mbox{in } \ \R^n_+.\]
Since $M^+_{\mathcal L_*^{\gamma}(s)}\leq M^+$ and $M^-_{\mathcal L_*^{\gamma}(s)}\geq M^-$, this finishes the proof the Claim.

\vspace{5pt}

Hence, $w$ satisfies all the assumptions of Theorem \ref{thm-liouv-1+s}, and thus $v = w/(x_n)^s$ is an affine function.
On the other hand, passing \eqref{28} to the limit we obtain that $v$ is orthogonal to every affine function and hence it must be $v\equiv 0$.
But then passing \eqref{nondeg358} to the limit we obtain that $v$ cannot be constantly zero in $B_1$; a contradiction.
\end{proof}

\begin{proof}[Proof of Theorem \ref{thm2-curved}]
Using Proposition \ref{prop_contr_1+s+alpha8} instead of Proposition \ref{prop_contr_1+s+alpha}, the proof follows exactly the same steps as the one of Theorem \ref{thm2}.
\end{proof}

\section{Flattening the boundary: proof of Proposition \ref{Lds-is-Cs}}
\label{calculets}

The aim of this section is to prove Propositions \ref{Lds-is-Cs} and \ref{Lds-is-Cs-B}.

Throughout this section, $d(x)$ is any $C^{2,\gamma}(\overline\Omega)$ function that coincides with $\textrm{dist}(x,\R^n\setminus\Omega)$ in a neighborhood of $\partial\Omega$.

\begin{prop}\label{Lds-is-Cs}
Let $s_0\in(0,1)$ and $s\in (s_0,1)$, and $\gamma\in(0,s]$.
Let $\Omega$ be any $C^{2,\gamma}$ bounded domain, and $L$ be any operator of the form \eqref{L*1}, with $\mu\in C^{1,\gamma}(S^{n-1})$.

Then, the function $d^s$ satisfies
\[\|L(d^s)\|_{C^\gamma(\overline\Omega)}\leq C,\]
where $C$ is a constant that depends only on $n$, $s_0$, $\Omega$, and $\|\mu\|_{C^{1,\gamma}(S^{n-1})}$.
\end{prop}

In fact, we will need also the following:

\begin{prop}\label{Lds-is-Cs-B}
Let $s_0\in(0,1)$ and $s\in (s_0,1)$, and $\gamma\in(0,s]$.
Let $\Omega$ be any $C^{2,\gamma}$ bounded domain, and $L$ be any operator of the form \eqref{L*1}, with $\mu\in C^{1,\gamma}(S^{n-1})$.

Then, for any function $\eta\in C^{2,\gamma}(\R^n)$, we have
\[\|L(d^s\eta)\|_{C^\gamma(\overline\Omega)}\leq C\|\eta\|_{C^{2,\gamma}},\]
where $C$ is a constant that depends only on $n$, $s_0$, $\Omega$, and $\|\mu\|_{C^{1,\gamma}(S^{n-1})}$.
\end{prop}

\begin{rem}
The dependence on $\Omega$ of the constant $C$ in Propositions \ref{Lds-is-Cs} and \ref{Lds-is-Cs-B} is through the $C^{2,\gamma}$ norm of the diffeomorphism that flattens the boundary $\partial\Omega$.
In particular, this constant $C$ is uniform among all domains with a uniform bound on this $C^{2,\gamma}$ norm.
\end{rem}

To prove these two propositions, we will need to flatten the boundary of $\Omega$.
In the following result we show how these operators change when we flatten the boundary.

\begin{prop}\label{prop-flatten}
Let $\bar L$ be any operator of the form \eqref{L*1}, with $\mu\in C^{1,\gamma}(S^{n-1})$.
Let $\Omega$ be any bounded $C^{2,\gamma}$ domain, and let $\bar u$ be any function satisfying
\[\bar L\bar u=\bar f\quad in\ \Omega,\qquad \bar u=0\quad in\ \R^n\setminus\Omega.\]
Let $\phi:\R^n\rightarrow\R^n$ be a $C^{2,\gamma}$ diffeomorphism that flattens the boundary $\partial\Omega$ and such that $(\phi_n)_+^s=d^s$.
In particular, $\phi(B_1^+)=\Omega\cap \{d<1\}$.

Then, the function $u=\bar u\circ \phi$ satisfies the equation
\[L(u,x)=f(x)\quad in\ B_1^+, \qquad u=0\quad in\ B_1^-,\]
where $f=\bar f\circ\phi$ and
\[L(u,x):=\bar L(u\circ \phi^{-1})(\phi(x)).\]
Moreover, $L(u,x)$ can be written as
\[L(u,x)=\int_{\R^n}\bigl(u(x)-u(x+z)\bigr)K(x,z)\frac{dz}{|z|^{n+2s}},\]
and
\[K(x,z)=a_1\left(x,\,\frac{z}{|z|}\right)+|z|\,a_2\left(x,\,\frac{z}{|z|}\right)+|z|^{1+\gamma}J(x,z)\qquad \textrm{for}\ |z|\leq2.\]
The functions $a_1$ and $a_2$ belong to $C^{1,\gamma}(S^{n-1})$ and $C^\gamma(S^{n-1})$ respectively, and $J$ is $C^\gamma$ with respect to $x$.

Furthermore,
\[a_1(x,-\theta)=a_1(x,\theta)\quad\textrm{for all}\ \theta\in S^{n-1},\]
and
\[a_2(x,-\theta)=-a_2(x,\theta)\quad\textrm{for all}\ \theta\in S^{n-1}.\]
The $C^\gamma$ norms of $a_1$, $a_2$, and $J$, depend only on $n$, $s$, the $\|\phi\|_{C^{2,\gamma}}$, and $\|\mu\|_{C^{1,\gamma}(S^{n-1})}$.
\end{prop}

\begin{proof}
By definition,
\[L(u,x)=\int_{B_2} \bigl\{ u(x)-u(\phi^{-1}(\phi(x)+y)) \bigr\} \frac{\mu(y/|y|)}{|y|^{n+2s}}\,dy.\]
Thus, making the change of variables
\[y=\phi(x+z)-\phi(x),\]
i.e., $z=\phi^{-1}(\phi(x)+y))-x$, we will have
\[y=A(x)z+z^tB(x)z+|z|^{2+\gamma}\psi(x,z),\]
where
\[A(x)=D\phi(x),\qquad B(x)=D^2\phi(x),\]
and $\psi(x,y)$ is bounded and $C^\gamma$ in the $x$-variable.

We have used that $\phi$ is $C^{2,\gamma}$.
Moreover, notice also that $A(x)$ is $C^{1+\gamma}$ and $B(x)$ is $C^\gamma$.

Writing now $z=r\theta$, with $r=|z|$ and $\theta\in S^{n-1}$, one finds
\[y=rA(x)\theta+r^2\theta^tB(x)\theta+r^{2+\gamma}\psi(x,r,\theta).\]
Therefore, this yields
\[|y|=r|A(x)\theta|+r^2\left[\frac{A(x)\theta}{|A(x)\theta|}\cdot(\theta^tB(x)\theta)\right]+r^{2+\gamma}\psi_1(x,r,\theta)\]
and also
\[\frac{1}{|y|}=
\frac{1}{r|A(x)\theta|}\left\{1-\frac{r}{|A(x)\theta|^2}\left[(A(x)\theta)\cdot(\theta^tB(x)\theta)\right]+r^{1+\gamma}\psi_2(x,r,\theta) \right\}.\]
Thus,
\begin{equation}\label{y/|y|}
\frac{y}{|y|}=\frac{A(x)\theta}{|A(x)\theta|}
+r\left\{\frac{\theta^tB(x)\theta}{|A(x)\theta|}-\frac{A(x)\theta}{|A(x)\theta|^3}\left[(A(x)\theta)\cdot(\theta^tB(x)\theta)\right]\right\}
+r^{1+\gamma}\psi_3(x,r,\theta).
\end{equation}
Moreover, the functions $\psi_1$, $\psi_2$, and $\psi_3$ are bounded and $C^\gamma$ in the $x$-variable.

Now, using that $\mu\in C^{1,\gamma}(S^{n-1})$ and \eqref{y/|y|}, one finds
\[\mu(y/|y|)=a_1(x,\theta)+r\,a_2(x,\theta)+r^{1+\gamma}\psi_4(x,r,\theta),\]
where
\[a_1(x,\theta)=a\left(\frac{A(x)\theta}{|A(x)\theta|}\right),\]
and
\[a_2(x,\theta)=\nabla_{S^{n-1}}a\left(\frac{A(x)\theta}{|A(x)\theta|}\right)
\cdot\left\{\frac{\theta^tB(x)\theta}{|A(x)\theta|}-\frac{A(x)\theta}{|A(x)\theta|^3}\left[(A(x)\theta)\cdot(\theta^tB(x)\theta)\right]\right\}.\]
Moreover, the function $\psi_4$ is bounded and it is $C^\gamma$ in the $x$-variable.

Finally notice that, since $\mu(y/|y|)=\mu(-y/|y|)$, it immediately follows from the expressions of $a_1$ and $a_2$ that $a_1(x,\theta)=a_1(x,-\theta)$ and that
$a_2(x,-\theta)=-a_2(x,\theta)$.
\end{proof}

We will also need the following Lemmas.

\begin{lem}\label{lem-Lds-1}
Let $s_0\in (0,1)$, $s\in (s_0,1)$, and $\gamma\in (0,s]$.
Let $a_1(x,\theta)$ be a function in $L^\infty(\R^n\times S^{n-1})$ which is $C^\gamma$ in $x$ and which satisfies
\[a_1(x,-\theta)=a_1(x,\theta)\quad\textrm{for all}\ \theta\in S^{n-1}.\]
Define
\[I_1(x):=\int_{B_2}((x_n)_+^s\eta(x)-(x_n+y_n)_+^s\eta(x+y))\frac{a_1(x,y/|y|)}{|y|^{n+2s}}\,dy.\]
Then, we have $I_1\in C^{\gamma}(\overline{B_1^+})$, and
\[\|I_1\|_{C^{\gamma}(B_1^+)}\leq \frac{C}{1-s},\]
where $C$ depends only on $n$, $s_0$, the $C^{2,\gamma}$ norm of $\eta$, and the $C^{\gamma}$ norm of $a_1$.
\end{lem}

\begin{proof}
\textbf{Case 1}: Assume $\eta\equiv1$.

Then, since $a_1$ is even, we have
\begin{equation}\label{cosa-util-53}
\int_{\R^n}((x_n)_+^s-(x_n+y_n)_+^s)\frac{a_1(x,y/|y|)}{|y|^{n+2s}}\,dy=c(x)(-\Delta)^s(x_+)^s=0.
\end{equation}
Therefore,
\[I_1(x)=\int_{\R^n\setminus B_2}((x_n)_+^s-(x_n+y_n)_+^s)\,\frac{a_1(x,y/|y|)}{|y|^{n+2s}}\,dy.\]
Now, using that $(x_n)_+^s$ is $C^\gamma$, we have
\[\left|(x_n^{(1)})_+^s-(x_n^{(1)}+y_n)_+^s-(x_n^{(2)})_+^s+(x_n^{(2)}+y_n)_+^s\right|\leq C|x_1-x_2|^\gamma|y|^{s-\gamma}\]
for any $x_1$ and $x_2$ in $B_1^+$.
Thus, using also that $a_1$ is $C^\gamma$ with respect to $x$, we find
\[\begin{split}
|I_1(x_1)-I_1(x_2)| \leq& \int_{\R^n\setminus B_2}C|x_1-x_2|^\gamma|y|^{s-\gamma}\frac{C}{|y|^{n+2s}}\,dy +\int_{\R^n\setminus B_2}C|y|^s\frac{C|x_1-x_2|^\gamma}{|y|^{n+2s}}\,dy\\
\leq &\,C|x_1-x_2|^\gamma.
\end{split}\]

\textbf{Case 2}: Assume that $\eta$ is a linear function, $\eta(x)=b\cdot x+c$.
Then,
\[\begin{split}
I_1(x)=&\textrm{PV}\int_{B_2}(x_n+y_n)_+^s(b\cdot y)\frac{a_1(x,y/|y|)}{|y|^{n+2s}}\,dy\\
&+(b\cdot x+c)\int_{\R^n\setminus B_2}((x_n)_+^s-(x_n+y_n)_+^s)\,\frac{a_1(x,y/|y|)}{|y|^{n+2s}}\,dy,\end{split}\]
where we have used \eqref{cosa-util-53}.

The second term is $C^\gamma$, as we already proved that in Case 1.
Hence, it remains to see that the first term is $C^\gamma$ also.

Since $a_1$ is $C^\gamma$ in $x$, it suffices to prove that
\[\left|\textrm{PV}\int_{B_2}\bigl[(x_n+y_n)_+^s-(x_n+h+y_n)_+^s\bigr](b\cdot y)\frac{a_1(x,y/|y|)}{|y|^{n+2s}}\,dy\right|\leq \frac{C}{1-s}|h|.\]
To prove this, we show next that the function
\[I_{1,2}(x):=\textrm{PV}\int_{B_2}(x_n+y_n)_+^{s-1}\,(b\cdot y)\frac{a_1(x,y/|y|)}{|y|^{n+2s}}\,dy\]
satisfies
\[\left|I_{1,2}(x)\right|\leq \frac{C}{1-s}.\]
Here, we denoted $(t)_+^{s-1}=(|t|^{s-2}t)_+$.

To bound $I_{1,2}(x)$, we first notice that
\[\textrm{PV}\int_{\R^n}(x_n+y_n)_+^{s-1}(b\cdot y)\frac{a_1(x,y/|y|)}{|y|^{n+2s}}\,dy=0.\]
Indeed, this follows from
\[\begin{split}
\textrm{PV}\int_{\R^n}(x_n+y_n)_+^{s-1}\,(b\cdot y)\, &\,\frac{a_1(x,y/|y|)}{|y|^{n+2s}}\,dy \\
&=\textrm{PV}\int_{S^{n-1}}|\theta_n|^{2s}|b\cdot\theta| a_1(x,\theta)d\theta\int_{-\infty}^{+\infty} (x_n+r)_+^{s-1}\frac{r\,dr}{|r|^{1+2s}}\\
&= c(x)(x_n)^{-s}\,\textrm{PV}\int_{-\infty}^{+\infty} (1+r)_+^{s-1}\frac{r\,dr}{|r|^{1+2s}},
\end{split}\]
and
\begin{equation}\label{calcul-integral-1D}\begin{split}
\textrm{PV}\int_{-\infty}^{+\infty} (1+r)_+^{s-1}\frac{r\,dr}{|r|^{1+2s}}&= \textrm{PV}\int_{0}^{+\infty} t^{s-1}\frac{t-1}{|t-1|^{1+2s}}\,dr\\
&= \lim_{\epsilon\rightarrow0}\left\{ -\int_0^{1-\epsilon} t^{s-1}\frac{dt}{(1-t)^{2s}}+\int_{1+\epsilon}^{\infty}t^{s-1}\frac{dt}{(t-1)^{2s}}\right\}\\
&= \lim_{\epsilon\rightarrow0}\left\{ -\int_0^{1-\epsilon} t^{s-1}\frac{dt}{(1-t)^{2s}}+\int_0^{\frac{1}{1+\epsilon}}z^{1-s}\frac{z^{2s-2}dz}{(1-z)^{2s}}\right\}\\
&= \lim_{\epsilon\rightarrow0}\int_{1-\epsilon}^{\frac{1}{1+\epsilon}}t^{s-1}\frac{dt}{(1-t)^{2s}}\\
&=0,
\end{split}\end{equation}

Thus,
\[I_{1,2}(x)=-\int_{\R^n\setminus B_2}(x_n+y_n)_+^{s-1}(b\cdot y)\frac{a_1(x,y/|y|)}{|y|^{n+2s}}\,dy,\]
and using that $x\in B_1$,
\[|I_{1,2}(x)|\leq \int_{\R^n\setminus B_2}|y_n|_+^{s-1}|y|\frac{C}{|y|^{n+2s}}\,dy=\frac{C}{1-s},\]
as desired.

\textbf{Case 3}: Let us do now the general case $\eta\in C^{2,\gamma}$.
We have
\[\begin{split}
I_1(x)=&\textrm{PV}\int_{B_2}(x_n+y_n)_+^s\bigl[\eta(x)-\eta(x+y)\bigr]\frac{a_1(x,y/|y|)}{|y|^{n+2s}}\,dy\\
&+\eta(x)\int_{\R^n\setminus B_2}((x_n)_+^s-(x_n+y_n)_+^s)\,\frac{a_1(x,y/|y|)}{|y|^{n+2s}}\,dy,\end{split}\]
where we have used \eqref{cosa-util-53}.

The second term is $C^\gamma$, as we already proved that in Case 1.
Hence, it remains to see that the first term is $C^\gamma$ also.

Let us denote $I_{1,3}(x)$ this first term, i.e.,
\[I_{1,3}(x)=\textrm{PV}\int_{B_2}\xi(x,y)\frac{a_1(x,y/|y|)}{|y|^{n+2s}}\,dy,\]
with
\[\xi(x,y)=(x_n+y_n)_+^s\bigl[\eta(x)-\eta(x+y)\bigr].\]
Using that $a_1$ is $C^\gamma$ with respect to $x$, and that
\begin{equation}\label{calculet-eta}
\begin{split}
|\xi(x,y)+\xi(x,-y)|\leq&\, (x_n+y_n)_+^s|2\eta(x)-\eta(x+y)-\eta(x-y)|+\\
&\,+|(x_n+y_n)_+^s-(x_n-y_n)_+^s|\cdot|\eta(x)-\eta(x-y)|\\
\leq & C|y|^{1+s}.
\end{split}\end{equation}
we find
\[\begin{split}|I_{1,3}(x)-I_{1,3}(x+he_i)|\leq &\left| \textrm{PV}\int_{B_2}\bigl[\xi(x,y)-\xi(x+he_i,y)\bigr]\frac{a_1(x,y/|y|)}{|y|^{n+2s}}\,dy\right|+\\
&+\int_{B_2}C|y|^{1+s}\frac{C|h|^\gamma}{|y|^{n+2s}}\,dy.\end{split}\]

We now claim that, for any $i=1,...,n$,
\begin{equation}\label{jronyaquejronya}
\left| \textrm{PV}\int_{B_2}\bigl[\xi(x,y)-\xi(x+he_i,y)\bigr]\frac{a_1(x,y/|y|)}{|y|^{n+2s}}\,dy\right|\leq  \frac{C}{1-s}|h|^\gamma.
\end{equation}
Indeed, since $\eta$ is $C^{2,\gamma}$, then
\[\|2\eta(\cdot)-\eta(\cdot+y)-\eta(\cdot-y)\|_{C^\gamma}\leq C|y|^2,\]
and hence a similar computation as in \eqref{calculet-eta} yields
\[\begin{split}
|\xi(x,y)+\xi(x,-y)-\xi(x+he_i,y)-\xi(x+he_i,-y)|\leq  C|y|^{1+s}|h|^\gamma,
\end{split}\]
and therefore \eqref{jronyaquejronya} follows.

Hence, we have showed that
\[|I_{1,3}(x)-I_{1,3}(x+he_i)|\leq \frac{C}{1-s}\,|h|^\gamma,\]
and thus the lemma is proved.
\end{proof}

\begin{lem}\label{lem-Lds-2}
Let $s_0\in (0,1)$, $s\in (s_0,1)$, and $\gamma\in (0,s]$.
Let $a_2(x,\theta)$ be a function in $L^\infty(\R^n\times S^{n-1})$ which is $C^\gamma$ in $x$ and which satisfies
\[a_2(x,-\theta)=-a_2(x,\theta)\quad\textrm{for all}\ \theta\in S^{n-1}.\]
Let $\eta\in C^{2,\gamma}_c(\R^n)$, and define
\[I_2(x):=\int_{B_2}((x_n)_+^s\eta(x)-(x_n+y_n)_+^s\eta(x+y))\frac{a_2(x,y/|y|)}{|y|^{n+2s-1}}\,dy.\]
Then, we have $I_2\in C^{\gamma}(\overline{B_1^+})$, and
\[\|I_2\|_{C^{\gamma}(B_1^+)}\leq \frac{C}{1-s},\]
where $C$ depends only on $n$, $s_0$, the $C^{2,\gamma}$ norm of $\eta$, and the $C^{\gamma}$ norm of $a_2$.
\end{lem}

\begin{proof}
\textbf{Case 1}: Assume $\eta\equiv1$.

Since the function $(x_n)^s_+$ does not depend on the first $n-1$ variables, and $a_2$ is $C^\gamma$ with respect to $x$, then it is clear that
\[|I_2(x)-I_2(x+he_i)|\leq \int_{B_2}|y|^s\frac{C|h|^\gamma}{|y|^{n+2s-1}}\,dy\leq \frac{C}{1-s}\,|h|^\gamma\]
for $i=1,2,...,n-1$.

Moreover, we also have
\[\begin{split}
I_2(x)&-I_2(x+he_n)= \\& \int_{B_2}((x_n+h)_+^s-(x_n+h+y_n)_+^s)\frac{a_2(x,y/|y|)-a_2(x+he_n,y/|y|)}{|y|^{n+2s-1}}\,dy\\
&+\int_{B_2}\left\{(x_n)_+^s-(x_n+y_n)_+^s-(x_n+h)_+^s+(x_n+h+y_n)_+^s \right\}\frac{a_2(x,y/|y|)}{|y|^{n+2s-1}}\,dy.
\end{split}\]
As before, the first term is bounded by
\[\left|\int_{B_2}((x_n+h)_+^s-(x_n+h+y_n)_+^s)\frac{a_2(x,y/|y|)-a_2(x+he_n,y/|y|)}{|y|^{n+2s-1}}\,dy\right|\leq \frac{C}{1-s}\,|h|^\gamma.\]
Thus, it only remains to see that the second term is also bounded by $C|h|^\gamma$.

We will show that, in fact, we have
\begin{equation}\label{we-want}
\biggl|\int_{B_2}\left\{(x_n)_+^s-(x_n+y_n)_+^s-(x_n+h)_+^s+(x_n+h+y_n)_+^s \right\}\frac{a_2(x,y/|y|)}{|y|^{n+2s-1}}\,dy \biggr|\leq \frac{C}{1-s}\,|h|.
\end{equation}
Indeed, it is clear that \eqref{we-want} is equivalent to
\begin{equation}\label{we-want2}
\left|\int_{B_2}\left\{(x_n)_+^{s-1}-(x_n+y_n)_+^{s-1}\right\}\frac{a_2(x,y/|y|)}{|y|^{n+2s-1}}\,dy \right|\leq \frac{C}{1-s},
\end{equation}
where we denoted (abusing a little bit the notation) $(x_n)_+^{s-1}=(|x_n|^{s-2}x_n)_+$.

Let us define
\[\tilde I_2(x):=\int_{B_2}\left\{(x_n)_+^{s-1}-(x_n+y_n)_+^{s-1}\right\}\frac{a_2(x,y/|y|)}{|y|^{n+2s-1}}\,dy.\]
Notice that, since $a_2$ is odd, then
\[\tilde I_2=-\textrm{PV}\int_{B_2}(x_n+y_n)_+^{s-1}\frac{a_2(x,y/|y|)}{|y|^{n+2s-1}}\,dy.\]

We now claim that
\[\hat I_2(x):=\textrm{PV}\int_{\R^n}(x_n+y_n)_+^{s-1}\frac{a_2(x,y/|y|)}{|y|^{n+2s-1}}\,dy=0.\]
Indeed, we have
\[\begin{split}
\hat I_2(x)&=\textrm{PV}\int_{-\infty}^{+\infty}\int_{S^{n-1}}(x_n+r\theta_n)_+^{s-1}a_2(x,\theta)\,\frac{r}{|r|^{1+2s}}\,d\theta\,dr\\
&=\textrm{PV}\int_{-\infty}^{+\infty}\int_{S^{n-1}}(x_n+r)_+^{s-1}|\theta_n|^{2s-1}\theta_n a_2(x,\theta)\,\frac{r}{|r|^{1+2s}}\,d\theta\,dr\\
&=c(x)\textrm{PV}\int_{-\infty}^{+\infty}(x_n+r)_+^{s-1}\frac{r}{|r|^{1+2s}}\,dr\\
&= c(x)(x_n)^{-s} \textrm{PV}\int_{-\infty}^{+\infty}(1+r)_+^{s-1}\frac{r}{|r|^{1+2s}}\,dr,
\end{split}\]
and hence the claim follows from \eqref{calcul-integral-1D}.

Therefore, we have
\[\tilde I_2(x)=\int_{\R^n\setminus B_2}(x_n+y_n)_+^{s-1}\frac{a_2(x,y/|y|)}{|y|^{n+2s-1}}\,dy,\]
and then, using that $x\in B_1$,
\[\begin{split}
  |\tilde I_2(x)|=&C\int_{\R^n\setminus B_2}(x_n+y_n)_+^{s-1}\frac{dy}{|y|^{n+2s-1}}\leq
  C\int_{\R^n_+\setminus B_2^+}(y_n)_+^{s-1}\frac{dy}{|y|^{n+2s-1}}\leq \frac{C}{1-s}.
\end{split}\]
Hence, \eqref{we-want2} is proved, and the lemma follows.

\textbf{Case 2}: Assume now that $\eta$ is any $C^{2,\gamma}$ function.

Then, using the result in Case 1, it suffices to show that the function
\[I_{2,2}(x):=\int_{B_2}(x_n+y_n)_+^s\bigl[\eta(x)-\eta(x+y)\bigr]\frac{a_2(x,y/|y|)}{|y|^{n+2s-1}}\,dy\]
is $C^\gamma$.

For $i=1,...,n-1$ we have
\[\begin{split} |I_{2,2}(x)&-I_{2,2}(x+he_i)|\leq \\
\leq &\,\int_{B_2}(x_n+y_n)_+^s\bigl|\eta(x+h)-\eta(x+h+y)\bigr|\frac{|a_2(x,y/|y|)-a_2(x+h,y/|y|)|}{|y|^{n+2s-1}}\,dy\\
&+\left|\int_{B_2}(x_n+y_n)_+^s\bigl[\eta(x)-\eta(x+y)-\eta(x+h)+\eta(x+h+y)\bigr]\frac{a_2(x,y/|y|)}{|y|^{n+2s-1}}\,dy\right|.
\end{split}\]
Since $a_2$ is $C^\gamma$, then
\[\begin{split}
\int_{B_2}(x_n+y_n)_+^s&\bigl|\eta(x+h)-\eta(x+h+y)\bigr|\frac{|a_2(x,y/|y|)-a_2(x+h,y/|y|)|}{|y|^{n+2s-1}}\,dy\leq \\
\leq & \int_{B_2}C|y|\frac{C|h|^\gamma}{|y|^{n+2s-1}}\,dy\leq \frac{C}{1-s}\,|h|^\gamma.\end{split}\]
On the other hand, we have that
\[\begin{split}&\left|\int_{B_2}(x_n+y_n)_+^s\bigl[\eta(x)-\eta(x+y)-\eta(x+h)+\eta(x+h+y)\bigr]\frac{a_2(x,y/|y|)}{|y|^{n+2s-1}}\,dy\right|\leq \\
&\hspace{7cm}\leq  \frac{C}{1-s}\,|h|.\end{split}\]
Indeed, this is equivalent to
\[\left|\int_{B_2}(x_n+y_n)_+^s\bigl[\partial_{x_i}\eta(x)-\partial_{x_i}\eta(x+y)\bigr]\frac{a_2(x,y/|y|)}{|y|^{n+2s-1}}\,dy\right|\leq \frac{C}{1-s},\]
and this follows immediately from the fact that $\eta$ is $C^{2,\gamma}$.

For $i=n$, the reasoning is very similar, and we only have to bound the additional term
\[\begin{split}
\int_{B_2}|(x_n+y_n+h)_+^s-(x_n+y_n)_+^s|\cdot |\eta(x)&-\eta(x+y)|\,\frac{|a_2(x,y/|y|)|}{|y|^{n+2s-1}}\,dy\leq \\
&\leq \int_{B_2}|h|^s|y|\,\frac{dy}{|y|^{n+2s-1}}\leq \frac{C}{1-s}\,|h|^s.\end{split}\]
Thus, the lemma is proved.
\end{proof}

\begin{lem}\label{lem-Lds-3}
Let $s_0\in (0,1)$, $s\in (s_0,1)$, and $\gamma\in (0,s]$.
Let $J(x,y)$ be a function in $L^\infty(\R^n\times \R^n)$ which is $C^\gamma$ in $x$.
Let $\eta\in C^{2,\gamma}_c(\R^n)$, and define
\[I_3(x):=\int_{B_2}((x_n)_+^s\eta(x)-(x_n+y_n)_+^s\eta(x+y))\frac{J(x,y)}{|y|^{n+2s-1-\gamma}}\,dy.\]
Then, we have $I_3\in C^{\gamma}(\overline{B_1^+})$, and
\[\|I_3\|_{C^{\gamma}(B_1^+)}\leq \frac{C}{1-s},\]
where $C$ depends only on $n$, $s_0$, the $C^{2,\gamma}$ norm of $\eta$, and the $C^{\gamma}$ norm of $J$.
\end{lem}

\begin{proof}
\textbf{Case 1}: Assume first $\eta\equiv1$.

Using \eqref{cosa-util-53} and that $J$ is $C^\gamma$ with respect to $x$, we have
\[\begin{split}
|I_3(x_1)-I_3(x_2)|&\leq \int_{B_2} C|x_1-x_2|^\gamma |y|^{s-\gamma}\frac{C}{|y|^{n+2s-1-\gamma}}\,dy+ \int_{B_2} C|y|^s\frac{C|x_1-x_2|^\gamma}{|y|^{n+2s-1-\gamma}}\,dy\\
&\leq  \frac{C}{1-s}\,|x_1-x_2|^\gamma+C|x_1-x_2|^\gamma,
\end{split}\]
and the result follows.

\textbf{Case 2}: Assume now that $\eta$ is any $C^{2,\gamma}$ function.

Then, one only needs to use that $g(x):=(x_n)_+^s \eta(x)$ is a $C^s$ function.
Indeed, one then have
\[|g(x_1)-g(x_1+y)-g(x_2)+g(x_2+y)|\leq C|x_1-x_2|^s\]
and
\[|g(x_1)-g(x_1+y)-g(x_2)+g(x_2+y)|\leq C|y|^s,\]
and interpolating these two inequalities,
\[|g(x_1)-g(x_1+y)-g(x_2)+g(x_2+y)|\leq C|x_1-x_2|^\gamma|y|^{s-\gamma}.\]
Using this, the proof is the same as in Case 1.
\end{proof}

Using the previous lemmas, we can now give the

\begin{proof}[Proof of Propositions \ref{Lds-is-Cs} and \ref{Lds-is-Cs-B}]
It follows immediately from Proposition \ref{prop-flatten} and Lemmas \ref{lem-Lds-1}, \ref{lem-Lds-2}, \ref{lem-Lds-3}.
\end{proof}

\begin{rem}
In case that both the domain $\Omega$ and the spectral measure $\mu$ are $C^\infty$, the result in Proposition \ref{Lds-is-Cs} is well known, and can be proved by Fourier transform methods; see \cite{Grubb}.
In this case, one has that $L(d^s)$ is $C^\infty(\overline\Omega)$.
\end{rem}

\section*{Appendix I: Proof of Lemma \ref{lemODE}}

In this appendix we give the

\begin{proof}[Proof of Lemma \ref{lemODE}]
Let us show first the statement (a).
Denote
\[a=1-2s.\]
We first note that the Caffarelli-Silvestre extension equation $\Delta u  + \frac{a}{y} \partial_y u = 0$ is written in polar coordinates $x=r\cos\theta$,  $y= r\sin \theta$, $r>0$, $\theta\in(0,\pi)$ as
\[ u_{rr} + \frac{1}{r} u_r + \frac{1}{r^2} u_{\theta\theta} + \frac{a}{r\sin \theta}\left(\sin\theta \,u_r + \cos\theta\, \frac{u_{\theta}}{r}\right)=0.\]
Note the homogeneity of the equation in the variable $r$.
If we seek for (bounded at $0$) solutions of the form $u= r^{s+\nu}\Theta_\nu(\theta)$, then it must be $\nu>-s$ and
\[ \Theta_\nu^{\prime\prime} + a{\rm \,cotg\,} \theta \,\Theta_\nu^\prime + (s+\nu)(s+\nu+a)\Theta_\nu =0.\]
If we want $u$ to satisfy  the boundary conditions
\[u(x,0) = 0\quad \mbox{for }x<0\quad \mbox{ and }\quad |y|^a \partial_y u(x,y) \to 0 \quad \mbox{as }y\to 0,\]
then  $\Theta_\nu$ must satisfy
\begin{equation}\label{bdryconditions}
\begin{cases}
\Theta_\nu(\theta) = \Theta_\nu(0) + o \bigl((\sin \theta)^{2s}\bigr) \to 0 \quad \mbox{as }\theta \searrow 0\\
\Theta_\nu(\pi)=0.
\end{cases}
\end{equation}
We have used that, for $x>0$
\[ \lim_{y\searrow 0} y^{a}\partial_y u(x,y) = 0 \quad \Rightarrow \quad u(x,y)= u(x,0)+ o(y^{2s}),\]
since $a=1-2s$.

To solve this ODE, consider
\[\Theta_\nu(\theta) = (\sin \theta)^s h(\cos \theta).\]
After some computations
and the change of variable $z= \cos\theta$ one obtains the following ODE for $h(z)$:
\[ (1-z^2)h^{\prime\prime}(z) -2zh^\prime(z) + \left(\nu+\nu^2-\frac{s^2}{1-z^2}\right) h(z)=0.\]
This is the so called ``associated Legendre differential equation''.
All solutions to this second order ODE solutions are given by
\[ h(z)= C_1 P_\nu^s(z)+ C_2Q_\nu^s(z), \]
where $P_\nu^s$ and $Q_\nu^s$ are the ``associated Legendre functions'' of first and second kind, respectively.

Translating \eqref{bdryconditions} to the function $h$, using that $\sin\theta \sim (1-\cos\theta)^{1/2}$ as $\theta\searrow 0$
and $\sin\theta \sim (1+\cos\theta)^{1/2}$ as $\theta\nearrow\pi$, we obtain
\begin{equation}\label{bdryconditionsh}
\begin{cases}
(1-z)^{s/2} h(z) = c + o\bigl((1-z)^s\bigr) \quad \mbox{as }z\nearrow 1\\
\lim_{z\searrow -1} (1+z)^{s/2}h(z)=0.
\end{cases}
\end{equation}

Let us prove that $P_\nu^s$ fulfill all these requirements only for $\nu=0,1,2,3,\dots$, while $Q_\nu^s$ have to be discarded.
To have a good description of the singularities of $P_\nu^s(z)$ at $z=\pm 1$ we use its expression as an hypergeometric
function
\[ P_\nu^s(z) =\frac{1}{\Gamma(1-s)} \frac{(1+z)^{s/2}}{(1-z)^{s/2}}\, \,_2F_1\left(-\nu, \nu+1; 1-s;\frac{1-z}{2}\right).\]
Using this and the definition of $\,_2F_1$ as a power series we obtain
\[ P_\nu^s(z) = \frac{1}{\Gamma(1-s)}\frac{2^{s/2}}{(1-z)^{s/2}}\left\{1 - \frac{\nu(\nu+1)}{1-s} \frac{1-z}{2} + o\bigl((1-z)^2\bigr)\right\}\quad \mbox{as }z\nearrow1. \]
Hence,  $(1-z)^{s/2}P_\nu^s(z) = c + O\bigl(1-z\bigr) = c+ o\bigl((1-z)^s\bigr)$ as desired.

For the analysis as $z\searrow -1$ we need  to use  Euler's transformation
\[ \,_2F_1(a,b;c;x) = (1-x)^{c-b-a} \,_2F_1(c-a, c-b; c; x),\]
obtaining
\[ P_\nu^s(z) = \frac{1}{\Gamma(1-s)} \frac{(1+z)^{s/2}}{2^{s/2}} \left(\frac{1+z}{2}\right)^{-s} \bigl\{ \,_2F_1(1-s-\nu, -s-\nu; 1-s; 1)  + o(1)\bigr\}\]
as $z\searrow-1$.
It follows that the zero boundary condition is satisfied if and only if
\[ \,_2F_1(1-s-\nu, -s-\nu; 1-s; 1) =\frac{\Gamma(1-s)\Gamma(s)}{\Gamma(-\nu)\Gamma(1+\nu)} =0.\]
This implies $\nu=0,1,2,3,\dots$, so that $\Gamma(-\nu)=\infty$.

With a similar analysis one easily finds that the functions $Q_\nu^s(x)$ do not satisfy \eqref{bdryconditionsh} for any $\nu\geq-s$.

The statement $(b)$ of the Lemma could be proved for example by using singular Sturm-Liouville theory after observing that the ODE
\[ \Theta_\nu^{\prime\prime} + a{\rm \,cotg\,} \theta \,\Theta_\nu^\prime - \lambda\Theta_\nu =0\]
can be written as
\[\bigl( |\sin \theta|^a\, \Theta_\nu^\prime\bigr)^\prime = \lambda  |\sin\theta|^a \Theta_\nu.\]
However, it is not necessary to do it because we have already computed the eigenfunctions to this ODE, and they are given by
\[\Theta_k(\theta)=(\sin \theta)^s P_k^s(\cos \theta),\]
where $P_\nu^s$ are the associated Legendre functions of first kind.
The functions $\{P_k^s(x)\}_{k\geq0}$ have been well studied, and they are known to be a complete orthogonal system in $L^2\bigl((0,1),dx\bigr)$; see \cite{Lebedev,WangGuo}.
Therefore, it immediately follows (after a change of variables) that $\{\Theta_k(\theta)\}_{k\geq0}$ are a complete orthogonal system in $L^2\bigl((0,\pi),(\sin\theta)^ad\theta\bigr)$.
Thus, the Lemma is proved.
\end{proof}

\section*{Appendix II: Interior regularity}

We give here the proof of the interior estimate in Lemma \ref{lem-interior}.

For it, we will need the following.

\begin{lem}\label{lem-liouv}
Let $\bar\alpha>0$ be the exponent given by Proposition \ref{thm1}.
Assume that $u\in C(\R^n)$ satisfies in the viscosity sense
\[
M^+\left\{u(\cdot+h)-u\right\}\ge 0\quad \textrm{and}\quad M^- \left\{u(\cdot+h)-u\right\}\le0 \quad \textrm{in}\  \R^n
\]
for all $h\in \R^n$.

Assume that for some $\beta\in(0,1)$ and $\alpha\in(0,\bar\alpha)$, $u$ satisfies
\begin{equation}\label{growthcontrol-1+sapp}
[u]_{C^\beta(B_R)}\le C R^\alpha \quad \mbox{for all } R\ge 1.
\end{equation}
Then,
\[u(x)= p\cdot x+b\]
for some $p\in \R^n$ and $b\in\R$.
\end{lem}

\begin{proof}
Given  $\rho\ge1$, let  $v(x) = \frac{u(\rho x+\rho h)-u(\rho x)}{ \rho^\alpha |\rho h|^\beta}$. By assumption we have
\[
M^+v \ge 0\quad \textrm{and}\quad M^- v \le0 \quad \textrm{in}\  B_1
\]
and
\[
\|v\|_{L^\infty(B_R)} \le C R^\alpha
\]
for all $R\ge 1$.

Hence it follows form the interior estimate in \cite{CS} (recall that $M^-_{\mathcal L_0}\le M^-\le  M^+\le M^+_{\mathcal L_0}$) that
\[ \|v\|_{C^{\bar \alpha}(B_1)}\le  C.\]
We use now the following well-known result: if all incremental quotients of order $\beta$ of $u$ are uniformly $C^\alpha$, then $u$ is $C^{\beta+\alpha}$ (unless $\beta+\alpha$ is an integer); see for instance \cite[Proposition 2.1]{Bass} or \cite[Lemma 5.6]{Caff-Cabre}.
This yields
\[ [u]_{C^{\bar\alpha+\beta}(B_\rho)} \le C\rho^{\alpha-\bar \alpha}.\]
Letting $\rho \to \infty$ we conclude that $[u]_{C^{\bar\alpha+\beta}(\R^n)}=0$, and the Lemma follows.
\end{proof}

We next show the following.

\begin{prop} \label{prop-interior}
Let $s_0\in(0,1)$, and let $\bar \alpha\in(0,s_0)$ be the constant given by Proposition \ref{thm1}.

Let $s\in(s_0,1)$, $\alpha\in(0,\bar\alpha)$, $\gamma\in(0,1)$, and $\beta\in(0,1)$ such that $\alpha+\beta\leq \gamma+2s$.
Assume in addition that $\alpha+\beta\neq 1$.

Let $w\in C^\beta(\R^n)$ be a solution of $\I(w,x)=0$ in $B_1$  where $\I$ is any fully nonlinear operator elliptic with respect to $\mathcal L_*(s)$ of the form
\[\I(w,x) = \inf_{b\in \mathcal B}\sup_{a\in\mathcal A} \bigl( L_{ab} w(x) + c_{ab}(x)\bigr).\]
Suppose that $L_{ab}$ are given by \eqref{L*1}-\eqref{L*2} and that
\begin{equation}
|\inf c_{ab}(x)|< \infty, \quad   [c_{ab}]_{\gamma; B_1}\leq 1.
\end{equation}

Then,
\begin{equation}\label{hola}
 \|w\|_{C^{\beta+\alpha}(B_{1/2})} \le C\|w\|_{C^\beta(\R^n)},
 \end{equation}
for some constant $C$ that depends only on $n$, $s_0$, ellipticity constants, $\alpha$, and $\beta$.
\end{prop}

\begin{proof}
It suffices to prove the estimate
\begin{equation}\label{pointest}
\sup_{r>0} r^{-\alpha} [w-P_r]_{C^\beta; B_r}\le C \|w\|_{\beta;\R^n},
\end{equation}
where
\[ P_{r} := {\rm arg \, min}_{P\in \mathcal P} \int_{B_r} \bigl( w_k -P\bigr)^2 \,dx,\]
$\mathcal P$ begin the linear space of polynomials of degree at most $\lfloor \alpha+\beta\rfloor$ with real coefficients.
Using \eqref{pointest}, \eqref{hola} follows easily.

The proof of \eqref{pointest} is by contradiction.
If it didn't hold there would be sequences $w_k$, $\I_k$, $s_k$, and $\gamma_k$ satisfying
\begin{itemize}
\item $\|w_k\|_{\beta; \R^n} \le 1$:
\item $\I_k (w_k,x) = 0$ in $B_1$;
\item $\I_k(w_k,x) = \inf_{b\in \mathcal B_k}\sup_{a\in \mathcal A_k} \bigl( L_{ab} u_k(x) + c_{ab}(x)\bigr)$;
\item $\{L_{ab} \,: \, a\in \mathcal A_k, \ b\in\mathcal B_k \}\subset \mathcal L(s_k)$ with $s_k \in [s_0,1]$;
\item $|\inf c_{ab}(x)|< \infty$, $[c_{ab}]_{ \gamma; B_1} \le 1$ for all $a\in \mathcal A_k$  and $b\in \mathcal B_k$;
\item $\gamma_k+2s_k \ge \alpha+\beta$.
\end{itemize}
for which
\begin{equation}\label{supsupapp}
\sup_k \sup_{r>0} r^{-\alpha}\ [ w_k -P_{k,r}]_{\beta;B_r} = +\infty ,
\end{equation}
where
\[ P_{k,r} := {\rm arg \, min}_{P\in \mathcal P} \int_{B_r} \bigl( w_k(x) -P\bigr) \,dx.\]

To prove that this is impossible we proceed similarly as in the Proof of Propostion~\ref{prop_contr_1+s+alpha}.
We define
\[
 \theta(r) := \sup_k  \sup_{r'>r}  (r')^{-\alpha}\,\bigl[w_k -P_{k,r}\bigr]_{\beta;B_{r'}}.
\]
The function $\theta$ is monotone nonincreasing, and $\theta(r)<+\infty$ for $r>0$ since
$\|w_k\|_{\beta; \R^n}\le 1$.
In addition, by \eqref{supsupapp} we have  $\theta(r)\nearrow +\infty$ as $r\searrow0$ and there are sequences $r'_m \searrow 0$ and $k_m$ for which
\begin{equation}\label{nondeg2}
(r'_m)^{-\alpha}  \bigl[w_k -P_{k_m,r'_m}\bigr]_{\beta; B_{r'_m}} \ge \frac{1}{2}  \theta(r'_m).
\end{equation}

From now on in this proof we will use the notations
\[u_m= u_{k_m}, \ P_m = P_{k_m, r'_m},\ s_m= s_{k_m},\ \gamma_m=\gamma_{k_m}.\]

We consider the blow up sequence
\begin{equation}\label{eqvmapp}
 v_m(x) = \frac{w_{m}(r'_m x)-P_{m} (r'_m x)}{(r'_m)^{\alpha+\beta}\theta(r'_m)}.
 \end{equation}
As in the proof of Proposition \ref{prop_contr_1+s+alpha}, for all $m\ge 1$ we have
\begin{equation}\label{2app}
 \int_{B_1} v_m(x)  P(x) \,dx =0\quad \mbox{for all } P\in \mathcal P,
\end{equation}
\begin{equation}\label{nondeg35app}
[v_m]_{\beta;B_1}\ge 1/2.
\end{equation}
and
\begin{equation}\label{growthc1app}
[v_{m}]_{\beta;B_R} \le  CR^{\alpha}
\end{equation}
for all $R\ge 1$.

Furthermore, we also have
\begin{equation} \label{growthc0app}
\|v_{m}\|_{L^\infty(B_R)} \le CR^{\alpha+\beta}.
\end{equation}

We next show that, replacing $v_m$ by appropriate rescalings, we may assume that instead of \eqref{nondeg35app}, the following holds
\begin{equation}\label{nondeg35appbis}
{\rm osc}_{B_1} v_m \ge 1/8.
\end{equation}
Indeed, if \eqref{nondeg35app} holds then there are $x_m\in B_1$ and $h_m\in B_{1-|x_m|}$ with $|h_m|>0$ such that
\[ \frac{ \bigl| v_m(x_m+h_m) +v_m(x_m-h_m) -2v_m(x_m) \bigr| }{ |h_m|^\beta}  \ge 1/4. \]
and we can always consider, instead of $v_m$, the function
\[\tilde v_m (x) :=  \frac{v_m (x_m +|h_m| x) - \tilde P_m(x)}{|h_m|^\beta},\]
where $\tilde P_m \in \mathcal P$ is chosen so that \eqref{2app} is satisfied with $v_m$ replaced by $\tilde v_m$.

Note that $\tilde P_m$  is the polynomial that approximates better (in the $L^2$ sense) $v_m(x_m +\, \cdot\,)$ in $B_{|h_m|}(x_m)$ and since $v_m \in C^{\beta}$ with the control \eqref{growthc1app} we have
\[  \bigl | v_m(x_m+|h_m|x) - \tilde P_m(x)\bigr| \le C|h_m|^{\beta} |x|^{\beta} .\]
Therefore, we readily show that $\tilde v_m$ also satisfies \eqref{growthc1app} and \eqref{growthc0app} (with $v_m$ replaced by $\tilde v_m$).

In summary, the new sequence $\tilde v_m$ satisfies the same properties as $v_m$ and, in addition, \eqref{nondeg35appbis}, as desired.

Next we prove the following

\vspace{5pt}
\noindent {\em Claim. Up to subsequences we have $s_m \to s\in [s_0,1]$ and $v_m\to v$ locally uniformly in $\R^n$, where $w\in C(\R^n)$.
Moreover, the limiting function $v$ satisfies the assumptions of the Liouville-type Lemma \ref{lem-liouv}.}
\vspace{5pt}

Since $\beta>0$, it follows from \eqref{growthc1app}, \eqref{growthc0app} and the Arzel\`a-Ascoli theorem that a subsequence of $v_m$ converges locally uniformly in $\R^n$ to some $v\in C(\R^n)$.

Passing to the limit \eqref{growthc1app} we find that the assumption \eqref{growthcontrol-1+sapp} of Theorem \ref{lem-liouv} is satisfied by $v$.

Similarly as in Proposition \ref{prop_contr_1+s+alpha}, using  that $[c_{ab}]_{ \gamma_k; B_1} \le 1$ we show that
\[
\begin{split}
0
&= \inf_{b\in \mathcal B_k}\sup_{a\in \mathcal A_k} \bigl( L_{ab} u_k(\bar x) + c_{ab}(\bar x)\bigr)\\
&\ge \inf_{b\in \mathcal B_k}\sup_{a\in \mathcal A_k} \bigl( L_{ab} u_k(\bar x) + c_{ab}(0)\bigr) + |\bar x|^{\gamma_k}
\end{split}
\]
and
\[
\begin{split}
0
&= \inf_{b\in \mathcal B_k}\sup_{a\in \mathcal A_k} \bigl( L_{ab} u_k(\bar x+\bar h) + c_{ab}(\bar x+\bar h)\bigr)\\
&\le \inf_{b\in \mathcal B_k}\sup_{a\in \mathcal A_k} \bigl( L_{ab} u_k(\bar x+ \bar h) + c_{ab}(0)\bigr) + |\bar x+\bar h|^{\gamma_k}
\end{split}
\]
and thus
\[ -|\bar x|^{\gamma_m} -|\bar x+\bar h|^{\gamma_m} \le  M^+_{\mathcal L_*(s_m)}   \bigl( v_m ( \cdot + h) -  v_m \bigr) \quad \mbox{in }B_{1/2}.\]

Therefore, rescaling we obtain
\begin{equation}\label{wwwwwwapp}
-\frac{3K^{\gamma_m} (r'_m)^{2s_m+\gamma_m}}{\theta(r'_m) (r'_m)^{\alpha+\beta}} \le M^+_{\mathcal L_*(s_m)}   \bigl( v_m ( \cdot + h) -  v_m \bigr)\quad \mbox{in } B_K
\end{equation}
whenever $|h|<K$ and $r'_m < \frac{1}{2K}$.

On the other hand, since $v_m \to v$ locally uniformly in $\R^n$ (up to subsequences), then we have
\begin{equation}\label{locallyuniformlyapp}
\bigl(v_m(\cdot+h)-v_m\bigr) \rightarrow \bigr(v(\cdot+h)-v\bigr)  \quad \mbox{locally uniformly in  }  \R^n.
\end{equation}
Also, similarly as in Proposition \ref{prop_contr_1+s+alpha}, \eqref{growthc1app} and the dominated convergence theorem imply that
\begin{equation}\label{L1weightedconvapp}
\bigl(v_m(\cdot+h)-v_m\bigr) \rightarrow \bigr(v(\cdot+h)-v\bigr)  \quad \mbox{in } L^1\bigr(\R^n, \omega_{s_0} \big),
\end{equation}
since $|v_m(\cdot+h)-v_m|\le C(1+|x|)^\alpha\le C(1+|x|)^{s_0} \in L^1(\R^n,\omega_{s_0})$.

Thus, using \eqref{L1weightedconvapp} and \eqref{locallyuniformlyapp} we can pass to the limit in \eqref{wwwwwwapp}  (for each $K\geq1$) to obtain
\[0\le M^+_{\mathcal L_*(s)}\bigl\{ v(\cdot + h)- v\bigr\} \quad \mbox{in } \R^n.\]

Analogously, we will have that
\[0\ge M^-_{\mathcal L_*(s)}\bigl\{ v(\cdot + h)- v\bigr\} \quad \mbox{in } \ \R^n_+.\]

\vspace{5pt}

Hence, $w$ satisfies all the assumptions of Lemma \ref{lem-liouv}, and thus $v$ is an affine function.
On the other hand, passing \eqref{2app} to the limit we obtain that $v$ is orthogonal to every affine function and hence it must be $v\equiv 0$.
But then passing \eqref{nondeg35appbis} to the (uniform) limit we obtain a contradiction.
\end{proof}

Finally, we give the:

\begin{proof}[Proof of Lemma \ref{lem-interior}]
The result follows by rescaling from Proposition \ref{prop-interior}.
Indeed, let
\[\bar w(x) = r^{-\alpha-\beta-s}w(rx).\]
We have
\[\|\bar w\|_{L^\infty(\R^n)} <\infty \quad \mbox{and}\quad \|\bar w\|_{L^{\infty}(B_R)} \le C_0 R^{\alpha+\beta+s} \]
for all $R\ge 1$.

Since the spectral measures $\mu_{ab}$  satisfy  $\|\mu_{ab}\|_{C^\gamma(S^{n-1})}\leq \Lambda$, we have
\[
\left[  \frac{\mu_{ab}(y/|y|)}{ |y|^{n+2s}} \right]_{C^\gamma(B_{2R} \setminus B_R)} \le \frac{C\Lambda}{R^{n+2s+\gamma}}.
\]
Hence, if $\eta\in C^\infty_c(B_1(e_n))$ is such that $\eta \equiv 1$ on $\overline{B_{5/6}}$, it follows that
\[\tilde \I( \bar w\eta ,\bar x) = \inf_{b\in \mathcal B}\sup_{a\in\mathcal A} \bigl( L_{ab} w(\bar x) + \tilde c_{ab}(\bar x)\bigr) = 0 \quad \mbox{in }B_{4/6},\]
where
\[ \tilde c_{ab}(\bar x) = r^{2s} r^{-\alpha-\beta-s} c_{ab}(r\bar x)+ L_{ab} (1-\eta)\bar w.\]
It then easily follows that
\[
[\tilde c_{ab}]_{C^\gamma(B_{4/6}(e_n))}\le Cr^{\gamma+s-\alpha-\beta}+ C C_0\leq C(1+C_0).
\]
Therefore, Proposition \ref{prop-interior} yields
\[\|\bar w\eta\|_{C^{\alpha+\beta}(B_{1/2}(e_n))}\leq C\|\bar w\eta\|_{C^\beta(B_1(e_n))}.\]
Rescaling back to $w$, we find
\[\|w\|_{C^{\alpha+\beta}(B_{r/2}(re_n))}\leq Cr^{-\alpha}\|w\|_{\beta;B_r(re_n)}\leq Cr^{-\alpha}r^{\alpha+s}=Cr^s,\]
and thus the Lemma is proved.
\end{proof}

\end{document}